\documentclass[11pt,leqno,oneside,letterpaper]{amsart}
\usepackage[letterpaper,left=1.75truein, top=1truein]{geometry}
\usepackage{amsmath,amsfonts,amssymb,amscd,amsthm,amsbsy,epsf,graphics}

\usepackage{color}

\usepackage[colorlinks,linkcolor=blue,citecolor=blue]{hyperref}
\usepackage{epsfig}
\usepackage{graphicx}

\newtheorem{thm}{Theorem}[section]
\newtheorem{cor}[thm]{Corollary}
\newtheorem{lemma}[thm]{Lemma}
\newtheorem{prop}[thm]{Proposition}
\newtheorem{conj}[thm]{Conjecture}

\newtheorem{assump}[thm]{Assumption}
\newtheorem{defin}[thm]{Definition}

\theoremstyle{definition}

\def\zed{{\mathbb Z}}
\def\D{{\mathcal D}}
\def\F{{\mathcal F}}
\def\Bar#1{\overline#1}
\def\ep{\varepsilon}
\renewcommand{\mod}{\operatorname{mod}}
\def\Int{\operatorname{int}}

\newtheoremstyle{cases}
  {12pt plus 6 pt}
  {2pt}
  {\bfseries}   
  {}
  {\bfseries}
  {.}
  {.5em}
  {}
\theoremstyle{cases}

\numberwithin{subcase}{case} \numberwithin{subsubcase}{subcase}
\numberwithin{equation}{subsection}
\usepackage{numbering}

\begin{document}

\def\G{{\Gamma}}
  \def\d{{\delta}}
  \def\ci{{\circ}}
  \def\e{{\epsilon}}
  \def\l{{\lambda}}
  \def\L{{\Lambda}}
  \def\m{{\mu}}
  \def\n{{\nu}}
  \def\o{{\omega}}
  \def\O{{\Omega}}
  \def\Th{{\Theta}}\def\s{{\sigma}}
  \def\v{{\varphi}}
  \def\a{{\alpha}}
  \def\b{{\beta}}
  \def\p{{\partial}}
  \def\r{{\rho}}
  \def\ra{{\rightarrow}}
  \def\lra{{\longrightarrow}}
  \def\g{{\gamma}}
  \def\D{{\Delta}}
  \def\La{{\Leftarrow}}
  \def\Ra{{\Rightarrow}}
  \def\x{{\xi}}
  \def\c{{\mathbb C}}
  \def\z{{\mathbb Z}}
  \def\2{{\mathbb Z_2}}
  \def\q{{\mathbb Q}}
  \def\t{{\tau}}
  \def\u{{\upsilon}}
  \def\th{{\theta}}
  \def\la{{\leftarrow}}
  \def\lla{{\longleftarrow}}
  \def\da{{\downarrow}}
  \def\ua{{\uparrow}}
  \def\nwa{{\nwtarrow}}
  \def\swa{{\swarrow}}
  \def\nea{{\netarrow}}
  \def\sea{{\searrow}}
  \def\hla{{\hookleftarrow}}
  \def\hra{{\hookrightarrow}}
  \def\sl{{SL(2,\mathbb C)}}
  \def\ps{{PSL(2,\mathbb C)}}
  \def\qed{{\hspace{2mm}{\small $\diamondsuit$}\goodbreak}}
  \def\pf{{\noindent{\bf Proof.\hspace{2mm}}}}
  \def\ni{{\noindent}}
  \def\sm{{{\mbox{\tiny M}}}}
   \def\sf{{{\mbox{\tiny F}}}}
   \def\sc{{{\mbox{\tiny C}}}}
  \def\ke{{\mbox{ker}(H_1(\partial M;\2)\ra H_1(M;\2))}}
  \def\et{{\mbox{\hspace{1.5mm}}}}

\title{Characteristic Submanifold Theory and Toroidal Dehn Filling}

\author[Steven Boyer]{Steven Boyer}
\thanks{Steven Boyer was partially supported by NSERC grant RGPIN 9446-2008}
\address{D\'epartement de Math\'ematiques, Universit\'e du Qu\'ebec \`a Montr\'eal, 201 avenue du Pr\'esident-Kennedy, Montr\'eal, QC H2X 3Y7.}
\email{boyer.steven@uqam.ca}
\urladdr{http://www.cirget.uqam.ca/boyer/boyer.html}

\author{Cameron McA. Gordon}
\thanks{Cameron Gordon was partially supported by NSF grant DMS-0906276.}
\address{Department of Mathematics, University of Texas at Austin, 1 University Station, Austin, TX 78712, USA.}
\email{gordon@math.utexas.edu}
\urladdr{http://www.ma.utexas.edu/text/webpages/gordon.html}

\author{Xingru Zhang}
\address{Department of Mathematics, University at Buffalo, Buffalo, NY, 14214-3093, USA.}
\email{xinzhang@buffalo.edu}
\urladdr{http://www.math.buffalo.edu/~xinzhang/}

\maketitle
\vspace{-.6cm}
\begin{center}
\today
\end{center}

\begin{abstract}
The exceptional Dehn filling conjecture of the second author concerning the relationship between exceptional slopes $\alpha, \beta$ on the boundary of a hyperbolic knot manifold $M$ has been verified in all cases other than small Seifert filling slopes. In this paper we verify it when $\alpha$ is a small Seifert filling slope and $\beta$ is a toroidal filling slope in the generic case where $M$ admits no punctured-torus fibre or semi-fibre, and there is no incompressible torus in $M(\beta)$ which intersects $\partial M$ in one or two components. Under these hypotheses we show that $\Delta(\alpha, \beta) \leq 5$. Our proof is based on an analysis of the relationship between the topology of $M$, the combinatorics of the intersection graph of an immersed disk or torus in $M(\alpha)$, and the two sequences of characteristic subsurfaces associated to an essential punctured torus properly embedded in $M$.
\end{abstract}

\section{Introduction}

This is the first of four papers concerned with the relationship between Seifert filling
slopes and toroidal filling slopes on the boundary of a hyperbolic knot manifold $M$.
Such results are part of the exceptional surgery problem, which we describe now.

A \textit{hyperbolic knot manifold} $M$ is a compact, connected, orientable,
hyperbolic $3$-manifold whose boundary is a torus. A \textit{slope} on $\partial M$ is a $\partial
M$-isotopy class of essential simple closed curves. Slopes can be
visualized by identifying them with $\pm$-classes of primitive
elements of $H_1(\partial M)$ in the \textit{surgery plane} $H_1(\partial
M; \mathbb R)$. The \textit{distance} $\Delta(\alpha_1, \alpha_2)$ between slopes
$\alpha_1, \alpha_2$ is the absolute value of the algebraic intersection number
of their associated classes in $H_1(\partial M)$. Given a set of slopes $\mathcal{S}$, set $\Delta(\mathcal{S}) = \sup\{\Delta(\alpha, \beta) : \alpha, \beta \in \mathcal{S}\}$.

To each slope $\alpha$ on $\partial M$ we associate the {\it $\alpha$-Dehn
  filling} $M(\alpha) = (S^1 \times D^2) \cup_f M$ of~$M$ where $f:
\partial (S^1 \times D^2) \to \partial M$ is any homeomorphism such
that $f(\{*\} \times \partial D^2)$ represents $\alpha$.

Set $\mathcal{E}(M) = \{ \alpha \mid M(\alpha) \mbox{ is not hyperbolic}\}$
and call the elements of $\mathcal{E}(M)$ \textit{exceptional slopes}. It
follows from Thurston's hyperbolic Dehn surgery theorem
that $\mathcal{E}(M)$ is finite, while
Perelman's solution of the geometrisation conjecture implies that
$$\mathcal{E}(M) = \{ \alpha \mid M(\alpha) \mbox{ is either reducible, toroidal, or small Seifert}\}$$
Here, a manifold is {\it small Seifert} if it admits a Seifert
structure with base orbifold of the form $S^2(a, b, c)$, where $a,b,c \geq 1$.

Much work has been devoted to understanding the structure of $\mathcal{E}(M)$
and describing the topology of $M$ when $|\mathcal{E}(M)| \geq  2$. For instance, sharp upper bounds are known for the distance between two reducible filling slopes \cite{GL}, between two toroidal filling slopes \cite{Go}, \cite{GW}, and between a reducible filling slope and a toroidal filling slope \cite{Oh}, \cite{Wu1}. More recently, strong upper bounds were obtained for the distance between a reducible filling slope and a small Seifert filling slope \cite{BCSZ2}, \cite{BGZ2}. In this paper, and its sequel, we examine the distance between toroidal filling slopes and small Seifert filling slopes.

Let $W$ be the exterior of the right-handed Whitehead link and $T$ one of its boundary components. Consider the following hyperbolic knot exteriors obtained by the indicated Dehn filling of $W$ along $T$: $M_1 = W(T; -1), M_2 = W(T; -2), M_3 = W(T; 5), M_4 = W(T; \frac52)$. One of the key conjectures concerning $\mathcal{E}(M)$ is the following:

\begin{conj}[C.McA.~Gordon] \label{conj} For any hyperbolic knot
    manifold $M$, we have $\#\mathcal{E}(M) \leq 10$ and
    $\Delta(\mathcal{E}(M)) \leq 8$.  Moreover, if $M \ne M_1, M_2,
    M_3, M_4$, then $\#\mathcal{E}(M) \leq 7$ and
    $\Delta(\mathcal{E}(M)) \leq 5$.
\end{conj}

It is shown in \cite{BGZ1} that the conjecture holds if the first
Betti number of $M$ is at least $2$. (By duality, it is at least $1$.)
Lackenby and Meyerhoff have proven that the first
statement of the conjecture holds in general \cite{LM}. See \S 2 of
their paper for a historical survey of results concerning upper bounds
for $\#\mathcal{E}(M)$ and $\Delta(\mathcal{E}(M))$. Agol has shown
that there are only finitely many hyperbolic knot manifolds $M$ with
$\Delta(\mathcal{E}(M)) > 5$ \cite{Ag}, though no practical
fashion to determine this finite set is known.

It follows from \cite{GL}, \cite{Oh}, \cite{Wu1}, \cite{Go} and \cite{GW} that Conjecture 1.1 holds if $\mathcal{E}(M)$ is replaced by $\mathcal{E}(M) \setminus \{\alpha \; \vert \; M(\alpha) \hbox{ is small Seifert}\}$. It remains, therefore, to consider pairs $\alpha, \beta$ such that $M(\alpha)$ is small Seifert and $M(\beta)$ is either reducible, toroidal or small Seifert. The first case is treated in \cite{BCSZ2}, where it is shown that, generically, $\Delta(\alpha, \beta) \le 4$. (See below for a more precise statement.) In the present paper we are interested in the second case. (We remark that if $M(\alpha)$ is toroidal Seifert fibred and $M(\beta)$ is toroidal then $\Delta(\alpha, \beta) \leq 4$ \cite{BGZ3}.) Here, Conjecture 1.1 implies

\begin{conj}\label{conj2} Suppose that $M$ is a hyperbolic knot
manifold $M$ and $\alpha, \beta$ are slopes on $\partial M$ such that $M(\alpha)$ is small Seifert and $M(\beta)$ toroidal. If $\Delta(\alpha, \beta) > 5$, then $M$ is the figure eight knot exterior.
\end{conj}

Understanding the relationship between Dehn fillings which yield small Seifert manifolds and other slopes in $\mathcal{E}(M)$ has proven difficult. The techniques used to obtain sharp distance bounds in other cases either provide realtively weak bounds here or do not apply at all. For instance, the graph intersection method (see e.g. \cite{CGLS}, \cite{GL}) cannot be used as typically, small Seifert manifolds do not admit closed essential surfaces. On the other hand, they usually do admit essential immersions of tori, a fact which can be exploited. Suppose that $\alpha$ and $\beta$ are slopes on $\partial M$ such that $M(\alpha)$ is small Seifert and $M$ contains an essential surface $F$ of slope $\beta$. It was shown in \cite{BCSZ1} how to construct an immersion $h: Y \to M(\alpha)$ where $Y$ is a disk or torus, a labeled ``intersection" graph $\Gamma_F = h^{-1}(F) \subset Y$, and, for each sign $\epsilon = \pm$, a sequence of characteristic subsurfaces
$$F = \dot{\Phi}_0^\epsilon \supset \dot{\Phi}_2^\epsilon \supset \dot{\Phi}_3^\epsilon \supset \ldots \supset \dot{\Phi}_n^\epsilon \supset \ldots $$
The relationship between the combinatorics of $\Gamma_F$, the two sequences of characteristic subsurfaces, and the topology of $M$  was exploited in \cite{BCSZ2} to show that if $M(\alpha)$ is small Seifert, $M(\beta)$ is reducible, and the (planar) surface $F$ is neither a fibre nor semi-fibre in $M$, then $\Delta(\alpha, \beta) \le 4$.
(See also \cite{CL}, \cite{Li} where a related method is used to study the existence of immersed essential surfaces in Dehn fillings of knot manifolds.) The main contributions of this paper are the further refinement of this technique and its application in the investigation of Conjecture \ref{conj2}.

When $M$ is the figure eight knot exterior there are (up to orientation-reversing homeomorphism of $M$) two pairs $(\alpha, \beta)$ with $\Delta(\alpha, \beta) > 5$ such that $M(\alpha)$ is small Seifert and $M(\beta)$ is toroidal, namely $(-3,4)$ and $(-2,4)$. The toroidal manifold $M(4)$ contains a separating incompressible torus which intersects $\partial M$ in two components. Moreover the corresponding punctured torus is not a fibre or semi-fibre in $M$. We show that if a hyperbolic knot manifold $M$ has a small Seifert filling $M(\alpha)$ and a toroidal filling $M(\beta)$ then $\Delta(\alpha, \beta) \leq 5$ in the generic case when $M$ admits no punctured-torus fibre or semi-fibre, and there is no incompressible torus in $M(\beta)$ which intersects $\partial M$ in one or two components. Our precise result is stated in \S \ref{assumptions} where we also detail our underlying assumptions and provide an outline of the paper. We will examine the non-generic cases of Conjecture \ref{conj2} in the forthcoming manuscripts \cite{BGZ3}, \cite{BGZ4}.

We are indebted to Marc Culler and Peter Shalen  for their role in the development of the ideas  in this paper.

\section{Basic assumptions and statement of main result} \label{assumptions}

Throughout the paper we work under the following assumptions.

\begin{assump} \label{assumptions on alpha and beta}
$M(\alpha)$ is a small Seifert manifold with base orbifold $S^2(a,b,c)$ where $a, b, c \geq 1$ and $M(\beta)$ is toroidal.
\end{assump}

\begin{assump} \label{assumption minimal}
Among all embedded essential tori in $M(\beta)$, $\widehat F$ is one whose intersection with $\partial M$  has the least number of components.
\end{assump}

Then $F=\widehat F\cap M$ is a
properly embedded essential punctured torus in $M$ with boundary
slope $\b$. Set $m=|\partial F| \geq 1$.

\begin{assump}\label{assumption sep}  If there is an essential separating torus in $M(\beta)$
satisfying Assumption \ref{assumption minimal}, $\widehat F$ has been chosen
to be separating.
\end{assump}

\begin{assump}\label{assumption twisted}
 If there is an essential torus in $M(\beta)$ satisfying
Assumption \ref{assumption minimal} which bounds a twisted
$I$-bundle over the Klein bottle in $M(\beta)$, $\widehat F$ has been chosen to
bound such an $I$-bundle.
\end{assump}

Note that it is possible that there are essential tori $\widehat F_1, \widehat F_2$ in $M(\beta)$ which bound twisted
$I$-bundles over the Klein bottle in $M(\beta)$, such that $\widehat F_1 \cap M$ is the frontier of a twisted $I$-bundle in $M$ but $\widehat F_2 \cap M$ is not.

\begin{assump}\label{assumption twisted 2}
 If there is an essential torus $\widehat F$ in $M(\beta)$ satisfying
Assumption \ref{assumption minimal} such that there is a twisted $I$-bundle in $M$ with frontier $F = \widehat F \cap M$, $\widehat F$ has been chosen so that $F $ is the frontier of a twisted $I$-bundle.
\end{assump}

Let $S$ be the surface in $M$ which is $F$ when $F$ is separating
and is a union of two parallel copies $F_1, F_2$ of $F$ when  $F$ is
non-separating. Then $S$ splits $M$ into two components $X^+$ and
$X^-$.

Let $\widehat S$ be a closed surface in $M(\beta)$ obtained by attaching disjoint
meridian disks of the $\beta$-filling solid torus to $S$. Then $\widehat S$ splits $M(\beta)$
into two compact submanifolds $\widehat X^+$ containing $X^+$ and $\widehat X^-$ containing $X^-$,
each having incompressible boundary $\widehat S$.

We call $F$ a {\it fibre} in $M$ if it is a fibre of a surface bundle map $M \to S^1$. Equivalently, the exterior of $F$ in $M$ is homeomorphic to $F \times I$. We call $F$ a {\it semi-fibre} in $M$ if it separates and splits $M$ into two twisted $I$-bundles.

\begin{assump}\label{assumption not (semi) fibre}
Assume that $F$, chosen as above, is neither a fibre nor a semi-fibre in $M$. In particular, assume that $X^+$ is not an $I$-bundle over a surface.
\end{assump}

Here is our main theorem.

\begin{thm}\label{main} Suppose that $M$ is a hyperbolic knot manifold and $\alpha, \beta$ slopes on $\partial M$ such that $M(\alpha)$ is a small Seifert manifold and $M(\beta)$ is toroidal. Let $F$ be an essential genus $1$ surface of slope $\beta$ which is properly embedded in $M$ and which satisfies the assumptions listed above. If $m \geq 3$, then
$\Delta(\alpha, \beta) \leq 5$.
\end{thm}

When the first Betti number of $M$ is at least $2$ or one of $M(\alpha)$ and $M(\beta)$ is reducible, Theorem \ref{main} holds by \cite{Ga}, \cite{BGZ1}, \cite{BCSZ2}, \cite{Oh}, \cite{Wu1}. Thus we make the following assumption.

\begin{assump} \label{assumption irreducible}
The first Betti number of $M$ is $1$ and both $M(\alpha)$ and $M(\beta)$ are irreducible.
\end{assump}

The paper is organised as follows. Section \ref{characteristic} contains background information on characteristic submanifolds associated to the pair $(X^\epsilon, S)$ ($\epsilon = \pm$). Sections \ref{essential annuli} and \ref{pairs of essential annuli} are devoted to exploring the restrictions forced on essential annuli in $(X^\epsilon, S)$ by our assumptions on $F$. These results will be applied in \S \ref{NT} to the study the structure of the characteristic submanifolds of $(X^\epsilon, S)$ and the topology of $\widehat X^\epsilon$. An analysis of the existence and numbers of certain characteristic subsurfaces of $S$ is made in \S \ref{section-tight}, \S \ref{f-hat essential annuli}, and \S \ref{existence tights}. The relation between the number of such surfaces and the length of essential homotopies in $(M, S)$ is determined in \S \ref{length}. Section \ref{graph} introduces the intersection graphs associated with certain immersions in $M(\alpha)$ and relates their structure to lengths of essential homotopies, leading to bounds on $\Delta(\alpha, \beta)$. Conditions which guarantee the existence of faces of the graph with few edges are investigated in \S \ref{triangle faces}, while the relations in the fundamental groups of $X^+$ and $X^-$ associated to these faces are considered in \S \ref{section-relation}. Theorem \ref{main} is proved when $F$ is non-separating in \S  \ref{section non-sep} and in the presence of ``tight" characteristic subsurfaces in \S \ref{4 or more tights} and \S \ref{2 tights}. The implications of certain combinatorial configurations in the intersection graph are  examined in \S \ref{extended s-cycles}. The proof of Theorem \ref{main} in the absence of tight components is achieved in the last two sections.

{\footnotesize \tableofcontents}

\section{Characteristic submanifolds of $(X^\epsilon, S)$} \label{characteristic}

\subsection{General subsurfaces of $S$} \label{subsurfaces}

A surface is called {\it large} if each of its components has negative Euler characteristic.

A connected subsurface $S_0$ of $S$ is called {\it neat} if it is either a collar on a boundary component of $S$ or it is large and each boundary component of $S$ that can be isotoped into $S_0$ is contained in $S_0$. A subsurface of $S$ is {\it neat} if each of its components has this property.

For each boundary component $b$ of $S$, let $\widehat b$ denote a meridian disk which it bounds in the $\beta$-filling solid torus of $M(\beta)$. The {\it completion} of a neat subsurface $S_0$ of $S$ is the surface $\widehat S_0 \subseteq M(\beta)$ obtained by attaching the disks $\widehat b$ to $S_0$ for each boundary component $b$ of $S$ contained in $S_0$.

A simple closed curve $c \subseteq S$ is called {\it outer} if it is parallel in $S$ to a component
of $\partial S$. Otherwise it is called {\it inner}.

A boundary component of a subsurface is called an {\it inner boundary component} if it is an inner curve, and an {\it outer boundary component} otherwise.

A neat subsurface $S_0$ of $S$ is called {\it tight}  if $\widehat S_0$ is a disk. Equivalently, $S_0$ is a connected, planar, neat subsurface of $S$ with at most one inner boundary component.

A simple closed curve $c \subseteq S$ which is essential in $\widehat S$
will be called $\widehat S$-{\it essential}.

We call a subsurface $S_0$ of $S$ an {\it $\widehat S$-essential annulus} if $\widehat S_0$ is an essential annulus in $\widehat S$.

Two essential annuli in $\widehat S$ are called {\it parallel} if their core
circles are parallel in $\widehat S$.

\subsection{Characteristic subsurfaces of $S$} \label{characteristic subsurfaces}

For the rest of the paper we use $\epsilon$ to denote either of the signs $\{\pm\}$.

A map $f$ of a path-connected space $Y$ to $S$ is called {\it large} if $f_\#(\pi_1(Y))$ contains a non-abelian free group.

A map of pairs $f:(Y, Z) \to (X^\epsilon, S)$ is called {\it essential} if it is not homotopic, as
a map of pairs, to a map $f':(Y, Z) \to (X^\epsilon, S)$ where $f'(Y) \subseteq S$.

An {\it essential annulus} in $(X^\epsilon, S)$ is the image of an essential proper embedding $(S^1 \times I, S^1 \times \partial I) \to (X^\epsilon, S)$.

An {\it essential homotopy of length $n$} in $(M, S)$ of $f: Y \to S$ which starts on the
$\epsilon$-side of $S$ is a homotopy
$$H:(Y \times [a, a + n], Y \times \{a,a+1,  \ldots, a+n\}) \rightarrow (M, S)$$
such that
\begin{enumerate}

\vspace{-.25cm} \item $H(y,a) = f(y)$ for each $y \in Y$;

\vspace{.25cm} \item $H^{-1}(S) = Y \times  \{a,a+ 1,  \ldots, a+n\}$;

\vspace{.25cm} \item for each $i \in \{1, 2, \ldots , n\}$, $H|Y \times [a + i-1, a + i]$ is an
essential map in $(X^{(-1)^{i-1} \epsilon}, S)$.

\end{enumerate}

Let $(\Sigma^\epsilon, \Phi^\epsilon)\subseteq (X^\epsilon,S)$ be the characteristic
$I$-bundle pair of $(X^\epsilon,S)$ \cite{JS}, \cite{Jo}. We shall use $\tau_\epsilon$ to denote the free involution
on $\Phi^\epsilon$ which interchanges the endpoints of the $I$-fibres of $\Sigma^\epsilon$.

The union of the components $P$ of $\Sigma^\epsilon$ for which $P \cap S$ is large is denoted by $\Sigma_1^\epsilon$. Set
$$\Phi_0^\epsilon = S$$
and
$$\Phi_1^\epsilon = \Sigma_1^\epsilon \cap S$$
More generally, for $j \geq 0$ we define $\Phi_j^\epsilon \subseteq S$ to be
the $j$-th characteristic subsurface with respect to the pair $(M,S)$ as defined
in \S 5 of \cite{BCSZ1}. We shall assume throughout the paper that $\Phi_j^\epsilon$ is neatly embedded in $S$. It is characterised up to ambient isotopy by the following property:
$$(*) \left \{ \begin{array}{ll}
\mbox{\rm a large function $f_0:Y \to S$ admits an essential homotopy of length $j$ which starts} \\
\mbox{\rm on the $\epsilon$-side of $S$ if and only if it is homotopic in $S$ to a map with image in $\Phi_j^\epsilon$ }
\end{array} \right. $$
See  \cite[Proposition 5.2.8]{BCSZ1}.

A compact connected $3$-dimensional submanifold $P$ of $X^\epsilon$ is called {\it neat} if
\begin{enumerate}

\vspace{- .15cm} \item $\partial P \cap S$ is a neat subsurface of $S$;

\vspace{.25cm} \item each component of $\overline{\partial P \setminus S}$ is an essential annulus in $(X^\epsilon, S)$;

\vspace{.25cm} \item some component of $\partial M \cap X^\epsilon$ is isotopic in $(X^\epsilon, S)$ into $P$, then it is contained in $P$.

\end{enumerate}
\vspace{- .15cm} A compact $3$-dimensional submanifold $P$ of $X^\epsilon$ is called {\it neat} if each of its components has this property.

Given a neat submanifold $P$ of $X^\epsilon$, we use $\widehat P$ to denote the submanifold of $\widehat X^\epsilon$ obtained by attaching to $P$ those components $H$ of $\widehat X^\epsilon \setminus \hbox{int}(M)$ for which $H \cap \partial M \subseteq P$.

For convenience we describe some properties of the characteristic $I$-bundle pair $(\Sigma^\epsilon, \Phi^\epsilon)$ which hold under the assumptions on $F$ listed in \S \ref{assumptions}, even though their justification will only be addressed in \S \ref{essential annuli}.

It follows from Proposition \ref{boundary-parallel} that if $c$ and $\tau_\epsilon(c)$ are two outer boundary components of ${\Phi}_1^\epsilon$, then ${\Phi}_1^\epsilon$ can be isotoped so that the annulus component in $\partial M \cap X^\epsilon$ bounded by $c$ and $\tau_\epsilon(c)$ is contained in
${\Sigma}_1^\epsilon$. We will therefore assume from \S \ref{pairs of essential annuli} on that ${\Sigma}_1^\epsilon$ is neatly embedded in $S$.

Let $\dot{\Phi}_j^\epsilon$ denote the union of the components
of $\Phi_j^\epsilon$ which contain some outer boundary components. We will see in
Proposition \ref{BdryToBdry} that $\tau_\epsilon$ preserves the set of outer, respectively inner, essential simple closed curves in $\Phi_1^\epsilon$. Hence, it restricts to a free involution on $\dot{\Phi}_1^\epsilon$,
which we continue to denote $\tau_\epsilon$. Let $\dot{\Sigma}_1^\epsilon$
denote the corresponding $I$-bundle.

Let $\breve{\Phi}_j^\epsilon$ be the neat subsurface in $S$ obtained from the union of
$\dot{\Phi}_j^\epsilon$ and a closed collar neighbourhood of $\partial S \setminus \partial  \dot{\Phi}_j^\epsilon$ in $S \setminus \dot{\Phi}_j^\epsilon$. It follows from the previous paragraph that there is an $I$-bundle pair $(\breve \Sigma_1^\epsilon, \breve \Phi_1^\epsilon)$ properly embedded in $(X^\epsilon, S)$ where $\breve \Sigma_1^\epsilon$ is the union of $\dot \Sigma_1^\epsilon$ and closed collar neighbourhoods of the annular components of $\partial M \cap X^\epsilon$ cobounded by components of $\partial S \setminus \partial  \dot{\Phi}_1^\epsilon$. Thus $\tau_\epsilon: \dot{\Phi}_1^\epsilon \to \dot{\Phi}_1^\epsilon$ extends to an involution $\tau_\epsilon: \breve{\Phi}_1^\epsilon \to \breve{\Phi}_1^\epsilon$.

A properly embedded annulus in $(X^\epsilon, S)$ is called {\it vertical} if it is a union of $I$-fibres of $\breve \Sigma_1^\epsilon$. A subsurface of $\breve \Phi_1^\epsilon$ is called {\it horizontal}.

It follows from the defining property ($*$) of the surfaces $\dot \Phi_j^\epsilon$ that if $(X^-, S)$ is an $I$-bundle pair, then for each $j\geq 0$,
$$(\breve{\Phi}_{2j}^-, \dot{\Phi}_{2j}^-) = (\breve{\Phi}_{2j+1}^-, \dot{\Phi}_{2j+1}^-)$$
$$(\breve{\Phi}_{2j+1}^+, \dot{\Phi}_{2j+1}^+) = (\breve{\Phi}_{2j+2}^+, \dot{\Phi}_{2j+2}^+)$$
$$(\breve{\Phi}_{2j+2}^-, \dot{\Phi}_{2j+2}^-) = (\tau_-(\breve{\Phi}_{2j+1}^+), \tau_-(\dot{\Phi}_{2j+1}^+))$$

Recall from \cite[Proposition 5.3.1]{BCSZ1} that for each $\epsilon$
and $j \geq 0$ there is a homeomorphism
$$h_j^\epsilon: (\breve{\Phi}_{j}^\epsilon, \dot{\Phi}_{j}^\epsilon) \to
(\breve{\Phi}_j^{(-1)^{j+1}\epsilon}, \dot{\Phi}_j^{(-1)^{j+1}\epsilon})$$
obtained by concatenating alternately restrictions of $\tau_+$ and $\tau_-$.
These homeomorphisms satisfy some useful properties:
$$h_1^\epsilon = \tau_\epsilon$$
$$h_{2j}^\epsilon: (\breve{\Phi}_{2j}^\epsilon, \dot{\Phi}_{2j}^\epsilon) \stackrel{\cong}{\lra}
(\breve{\Phi}_{2j}^{-\epsilon}, \dot{\Phi}_{2j}^{-\epsilon})$$
$$h_{2j+1}^\epsilon: (\breve{\Phi}_{2j+1}^\epsilon, \dot{\Phi}_{2j+1}^\epsilon) \stackrel{\cong}{\lra}
(\breve{\Phi}_{2j+1}^\epsilon, \dot{\Phi}_{2j+1}^\epsilon)\mbox{ is a free involution} $$

Finally, consider two large subsurfaces $S_0, S_1$ of $S$. Their {\it large essential intersection} is a large, possibly empty, subsurface $S_0 \wedge S_1$ of $S$ characterised up to isotopy in $S$ by the property:
$$(**) \left \{ \begin{array}{ll}
\mbox{\rm a large function $f: Y \to S$ is homotopic into both  } \\
\mbox{\rm $S_0$ and $S_1$ if and only if it is homotopic into $S_0 \wedge S_1$ }
\end{array} \right. $$
See  \cite[Proposition 4.2]{BCSZ1}. It follows from the defining property ($*$) of the surfaces $\dot \Phi_j^\epsilon$ that
\begin{equation}\label{eq surface calculus}
\text{\em $h_j^\epsilon(\dot{\Phi}_{j+k}^\epsilon)  = \dot{\Phi}_{j}^{(-1)^{j+1}\epsilon} \wedge \dot{\Phi}_{k}^{(-1)^{j}\epsilon}$}
\end{equation}

\section{Essential embedded annuli in $(X^\epsilon, S)$} \label{essential annuli}

The next two sections are devoted to exploring the restrictions forced on essential annuli in $(X^\epsilon, S)$ by our assumptions on $F$. These results will be applied in \S \ref{NT} to the study of the structure of $\dot{\Phi}_1^+$ and $\dot{\Phi}_1^-$.

\begin{lemma}\label{incomp}
Let $U$ be a submanifold of $M(\beta)$ which is homeomorphic to a
Seifert fibred space over the disk with two cone points. If
$U$ contains a closed curve which is non-null homotopic in
$M(\beta)$, then either

$(i)$ $\partial U$ is an incompressible torus in $M(\beta)$, or

$(ii)$ $\overline{M(\beta) \setminus U}$ is a solid torus and $M(\beta)$ is a torus bundle over the circle which admits a Seifert
structure with base orbifold of the form $S^2(a,b,c)$ where $\frac{1}{a} + \frac{1}{b} + \frac{1}{c} = 1$.

\end{lemma}

\pf Suppose that $\partial U$ is compressible in $M(\beta)$. Then
it is compressible in $\overline{M(\beta) \setminus U}$. The
surgery of the torus $\partial U$ using a compressing disk produces a
separating $2$-sphere. Since $M(\beta)$ is irreducible, this $2$-sphere
bounds a $3$-ball $B$ in $M(\beta)$. By hypothesis, $U$ is not
contained in $B$. Thus $\overline{M(\beta) \setminus U}$ is a solid torus.
As $M(\beta)$ is irreducible (Assumption \ref{assumption irreducible}) and a Dehn filling of $U$, it is a Seifert
fibred manifold over the $2$-sphere with at most three cone points.
But if such manifold contains an incompressible torus, it is a torus
bundle over the circle and admits a Seifert structure of the type described in (ii).
\qed

\begin{lemma}\label{three cases}
Suppose that $(A, \partial A)\subseteq (X^\e, S)$ is an embedded essential
annulus. Let $c_1, c_2$ be the two boundary components of $A$.

$(1)$ $c_1$ is essential in $\widehat S$ if and only if
$c_2$ is essential in $\widehat S$.

$(2)$ If $c_1$ and $c_2$ cobound an annulus $E$ in $\widehat S$, then
$A$ is not parallel in $\widehat X^\epsilon$ to $E$. Furthermore $E$  is
essential in $\widehat S$.

$(3)$ If  one of $c_1$ and  $c_2$ is not essential in $\widehat S$,
then $c_1$ and $c_2$ bound disjoint disks $D_1$ and $D_2$ in
$\widehat S$ such that $|D_1\cap \partial M|=|D_2\cap \partial M|$.
\end{lemma}

\pf (1)  This follows from the incompressibility of $\widehat S$ in
$\widehat X^\epsilon$.

(2) Suppose otherwise that  $A$ is parallel to $E$ in $\widehat
X^\epsilon$. Since $A$ is essential in $X^\epsilon$,  $E\cap \partial M$ is not empty.
But then we may consider $E$ as an annulus in $\widehat F$, and if
we replace $E$ in $\widehat F$  by $A$,  we get a torus in $M(\beta)$
which is incompressible (since it is isotopic to $\widehat F$) but has
fewer than $m$ components of intersection with $\partial M$. This
contradicts Assumption \ref{assumption minimal}.

Now we show that  $E$ is essential in $\widehat S$. Suppose
otherwise. Then one of $c_1$ and $c_2$, say $c_1$, bounds a disk $D$
in $\widehat S$ with interior disjoint from $E$. If $A$ is non-separating in
$\widehat X^\epsilon$ then $A \cup E$ is a non-separating Klein bottle or torus
with compressing disk $D$ with non-separating boundary. Compression of $A \cup E$
along $D$ yields a non-separating $2$-sphere in $\widehat X^\epsilon$,
which is impossible since $\widehat X^\epsilon$ is
irreducible (Assumption \ref{assumption irreducible}).
Thus $A$ is separating in $\widehat X^\epsilon$ and therefore
$T=E\cup A$ is a torus. Denote by $W_1$ and $W_2$ the two components of $\widehat X^\epsilon$
cut open along $A$ and assume that $W_1$ is the component whose boundary is $T$.
Assumption \ref{assumption irreducible} shows that a regular neighborhood $Y$ of
 $W_1\cup D$ in $\widehat X^\epsilon$ is a $3$-ball.
Hence the disk $E\cup D$ is isotopic in $Y$ to the disk $A\cup D$.
So by Assumption \ref{assumption minimal} we have $|E\cap \partial M|=0$.
Therefore $W_1$ is contained in $X^\epsilon$.

Since $F$ is not contained in a regular neighborhood of $\partial M$,
$T=\partial W_1$ is not parallel to $\partial M$.
The hyperbolicity of $M$ then implies that $T$ is compressible in $M$.
Since $F$ is essential, we may assume that
a compressing disk $D_*$ for $T$ in $M$ is contained in $X^\epsilon$.
If the interior of $D_*$ is disjoint from $W_1$, then $W_1$ is contained
in a $3$-ball in $M$ contrary to the fact that $c_1$ is essential in $S$.
Thus $D_* \subset W_1$. The $2$-sphere obtained by compressing
$T$ along $D_*$ bounds a $3$-ball contained in $W_1$.
(Otherwise $\partial M$ would be contained in $W_1\subseteq X^\epsilon$.)
Hence  $W_1$ is a solid torus. But $A$ is not parallel to $E$ in $W_1$.
Therefore $Y$ is once-punctured lens space with non-trivial fundamental group
and not a $3$-ball. This contradiction completes the proof of (2).

(3) By part (1) of this lemma,  $c_i$ bounds a disk $D_i$ in
$\widehat S$ for each of $i=1,2$. If $D_1$ and $D_2$ are not
disjoint, then one is contained in the other, say $D_1\subseteq D_2$.
Thus $c_1$ and $c_2$ bound an annulus $E$ in $D_2$. This contradicts
 part (2) of this theorem.

So $D_1$ and $D_2$ are disjoint. Let $d_i=|D_i\cap \partial M|$, $i=1,2$.
Suppose otherwise that $d_1\ne d_2$, say $d_1<d_2$. Since $\widehat
X^\epsilon$ is irreducible,  $A\cup D_1\cup D_2$ bounds a
$3$-ball $B$ in $\widehat X^\epsilon$ with the interior of $B$ disjoint
from $\widehat S$. Then it is not hard to see that the disk $D_2$
can be isotoped rel its boundary in $B$ to have at most $d_1$
intersection components with $\partial M$. This implies that the
incompressible torus $\widehat F$ can be isotoped in $M(\beta)$ to have
less than $m$ intersection components with $\partial M$, which again
contradicts our minimality assumption on $m=|\partial F|$.
\qed

A {\it root torus} in $(X^\epsilon, S)$ is a solid torus $\Theta
\subseteq X^\epsilon$ such that $\Theta \cap S$ is an incompressible
annulus in $\partial \Theta$ whose winding number in $\Theta$ is at least
$2$ in absolute value. For instance, a regular neighbourhood of an embedded M\"{o}bius band $(B, \partial B)\subseteq (X^\e, S)$ is a root torus. Note that for such a $\Theta$, $\overline{\partial \Theta \setminus (\Theta \cap F)}$ is an essential annulus in $(X^\epsilon, F)$.

Lemma \ref{three cases}(2) yields the following lemma.

\begin{prop} \label{Mband}
If $\Theta$ is a root torus in $(X^\epsilon, S)$, then $\Theta \cap S$ is an essential annulus in $\widehat S$. In particular, the boundary of a M\"{o}bius band properly embedded in $X^\epsilon$ is essential in $\widehat S$.
\qed
\end{prop}

\begin{prop} \label{BdryToBdry}
A simple closed curve $c \subseteq \breve{\Phi}_1^\epsilon$ is inner, respectively outer, if and only if
$\tau_\epsilon(c)$ is inner, respectively outer. In particular, the image by $\tau_\epsilon$ of a tight subsurface of $\breve{\Phi}_1^\epsilon$ is a tight subsurface of $\breve{\Phi}_1^\epsilon$.
\end{prop}

\pf Suppose that $c$ is inner. It suffices to see that $\tau_\epsilon(c)$ is inner as well. If $c$ is $\widehat S$-essential, so is $\tau_\epsilon(c)$ since they cobound a singular annulus. Thus $\tau_\epsilon(c)$ is inner. Otherwise, $c$ bounds a disk $D$ in $\widehat S$ containing at least two components of $\partial S$. If $c$ cobounds an annulus with $\tau_\epsilon(c)$ in $X^\epsilon$, for instance if $c$ is contained in a product bundle component of $\breve{\Sigma}_1^\epsilon$, then Lemma \ref{three cases}(3) shows that $\tau_\epsilon(c)$ is inner. In general, let $\phi$ be the component of $\breve{\Phi}_1^\epsilon$ which contains $c$ and $\Sigma$ the component of $\breve{\Sigma}_1^\epsilon$ which contains $\phi$. Each boundary component of $\phi$ cobounds a vertical annulus in $\Sigma$ with its image under $\tau_\epsilon$, so if $c$ is boundary-parallel in $\phi$, we are done.  On the other hand, if it is not  boundary-parallel in $\phi$, $\tau_\epsilon(c)$ is not boundary-parallel in $\tau_\epsilon(\phi)$ and therefore cannot be boundary-parallel in $S$. Thus $\tau_\epsilon(c)$ is inner.
\qed

\begin{prop} \label{tight not invariant}
If $\phi$ is a tight component of $\breve \Phi_{2j+1}^\epsilon$, then $h_{2j+1}^\epsilon(\phi) \cap \phi = \emptyset$.
\end{prop}

\pf If $h_{2j+1}^\epsilon(\phi) \cap \phi \ne \emptyset$, then $h_{2j+1}^\epsilon(\phi) = \phi$. Hence as $h_{2j+1}^\epsilon = g_j \circ \tau_\epsilon \circ g_j^{-1}$ where $g_j = \tau_{(-1)^j \epsilon} \circ \tau_{(-1)^{j-1} \epsilon} \circ \ldots \circ \tau_{-\epsilon}$, we have $\tau_\epsilon(\phi') = \phi'$ where $\phi'$ is the the tight subsurface $g_j(\phi)$ of $\breve \Phi_1^\epsilon$. It follows from Proposition \ref{BdryToBdry} that if $c$ is the inner boundary component of $\phi'$, then $\tau_\epsilon(c) = c$. Thus $c$ bounds a M\"{o}bius band properly embedded in $X^\epsilon$. But then Proposition \ref{Mband} implies that $c$ is $\widehat S$-essential, contrary to the tightness of $\phi'$. Thus $h_{2j+1}^\epsilon(\phi) \cap \phi = \emptyset$.
\qed

\begin{prop} \label{extend fibering}
Let $P$ be a component of $\Sigma_1^\epsilon$ and suppose that $c$ is an inner boundary component of $P \cap S$ which is inessential in $\widehat S$. Let $D \subset \widehat S$ be the disk with boundary $c$ and suppose that the component of $P \cap S$ containing $c$ is disjoint from $\hbox{int}(D)$. Then $P \cap D = c$ and if $H$ is the component of $\overline{\widehat X^\epsilon \setminus P}$ which contains $D$ and $A$ is the annulus $P \cap H$, then $(H, A) \cong (D^2 \times I, (\partial D^2) \times I)$. In particular, the $I$-bundle structure on $P$ extends over $P \cup H$.
\end{prop}

\pf Lemma \ref{three cases}(3) implies that there is a disk $D' \subset \widehat S$ disjoint from $D$ such that $\partial D' = \tau_\epsilon(c)$. Thus $P \cap D = c$. Note that $\partial A = c \cup \tau_\epsilon(c)$.  Then $D \cup A \cup D'$ is a $2$-sphere which bounds a $3$-ball $B \subseteq \widehat X^\epsilon$ such that $B \cap P = A$. The desired conclusions follow from this.
\qed

\begin{lemma} \label{non-sep}
Suppose that $(A, \partial A) \subseteq (X^\epsilon, S)$ is a non-separating essential annulus with boundary components $c_1, c_2$. Then $c_1$ and $c_2$ are essential in $\widehat S$. Further, either

$(i)$ $S = F$, $X^-$ is a twisted $I$-bundle, and $\partial A$ splits $\widehat F$ into two annuli $E_1, E_2$ such that $|E_j \cap \partial M| = m/2$ and $A \cup E_j$ is a Klein bottle for $j = 1, 2$;
or

$(ii)$ $S = F_1 \cup F_2$ has two components where $c_j$ is contained in $F_j$.
\end{lemma}

\pf The components of $\partial A$ are either both inessential in $\widehat S$ or both essential (Lemma \ref{three cases}(1)). In the former case, Lemma \ref{three cases}(3) implies that there are disjoint disks $D_1, D_2$ in $\widehat S$ such that $c_j = \partial D_j$. Then $D_1 \cup A \cup D_2$ is a $2$-sphere in the irreducible manifold $\widehat X^\epsilon$, which therefore bounds a $3$-ball. This is impossible since $A$ in non-separating. Thus the components of $\partial A$ are essential in $\widehat S$.

If conclusion (ii) does not hold, there is a component $S_0$ of $S$ such that $\partial A$ splits $\widehat S_0$ into two annuli: $\widehat S_0 = E_1 \cup_{\partial A} E_2$. We assume, without loss of generality, that $|E_1 \cap \partial M| \leq m/2$. Since $E_j \cup A$ is non-separating and intersects $\partial M$ in fewer than $m$ components, $E_1 \cup A$ is a Klein bottle. A regular neighbourhood $U$ of $E_1 \cup A$ is a twisted $I$-bundle over $E_1 \cup A$ and contains a loop which is not null-homotopic in $M$. Since $\widehat S$ is isotopic into $\overline{M(\beta) \setminus U}$, the latter cannot be a solid torus. Thus Lemma \ref{incomp} implies that $\partial U$ is an incompressible torus in $M(\beta)$. Hence $m \leq |\partial U \cap \partial M| = 2|(E_1 \cup A) \cap \partial M| = 2|E_1 \cap \partial M| \leq m$. It follows that $|E_1 \cap \partial M| = |E_2 \cap \partial M| = m/2$ and $|\partial U \cap \partial M|  = m$. In particular, $(E_1 \cup A) \cap M$ is an $\frac{m}{2}$-punctured Klein bottle properly embedded in $M$ with twisted $I$-bundle neighbourhood $U \cap M$. Assumptions \ref{assumption sep} and \ref{assumption twisted 2} then imply that $S = F$ and $X^-$ is a twisted $I$-bundle. Hence situation (i) holds.
\qed

\begin{lemma} \label{parallel}
Suppose that $(A_j, \partial A_j)$ $(j = 1, 2)$ are disjoint essential annuli contained in $(X^\epsilon, S)$. If a boundary component $c_1$ of $A_1$ cobounds an annulus $E \subseteq S$ with a boundary component $c_2$ of $A_2$ and $c_1$ is $\widehat S$-inessential, then $A_1$ is isotopic to $A_2$ in $X^\epsilon$.
\end{lemma}

\pf Let $\partial A_j = c_j \cup c_j'$ $(j = 1, 2)$. We can suppose that $c_j$ bounds a disk $D_j$ in
$\widehat S$ ($j = 1, 2$) such that $D_2 = D_1 \cup E$. According to Lemma \ref{three cases}(3), $c_j'$ bounds
a disk $D_j'$ in $\widehat S$ such that $D_j \cap D_j' = \emptyset$ ($j = 1, 2$). Since $M(\beta)$ is
irreducible, the $2$-sphere $\Pi_2 = D_2 \cup A_2 \cup D_2'$ bounds a $3$-ball $B_2 \subseteq \widehat
X^\epsilon$.

Since $c_1 \subseteq \hbox{int}(D_2)$ and $\hbox{int}(A_1)$ is disjoint from $\Pi_2$,  $A_1$ is contained in
$B_2$. If $D_1' \cap D_2 \ne \emptyset$, then $D_1' \subseteq \hbox{int}(E) \subseteq S$. But this is
impossible as $c_1'$ is essential in $S$. Thus $D_1' \subseteq \hbox{int}(D_2')$ and therefore $c_1'$ and
$c_2'$ cobound an annulus $E' \subseteq D_2' \subseteq \widehat S$. It then follows from Lemma \ref{three
cases}(3) that $E'\subset S$.

The torus $T = E \cup A_1 \cup E' \cup A_2 \subset X^\epsilon$ is not boundary-parallel in the hyperbolic
manifold $M$, so must compress in $X^\epsilon$. It cannot be contained in a $3$-ball in $M$ since $\partial A_1$ is essential in $S$. Hence it bounds a solid torus $\Theta$ in $X^\epsilon$. Proposition \ref{Mband} shows that $\Theta$ is not a root torus, so $A_1$ must be parallel to $A_2$ in $X^\epsilon$.
\qed

\begin{prop} \label{boundary-parallel}
Let $(A, \partial A)$ be an essential annulus in $(X^\epsilon, S)$ such that a component $c$ of $\partial A$ cobounds an annulus $E \subseteq S$ with a component $c'$ of $\partial S$. Then $(A, \partial A)$ is isotopic in $(X^\epsilon, S)$  to a component of $\partial M \cap X^\epsilon$.
\end{prop}

\pf Let $A'$ be the component of $\partial M \cap X^\epsilon$ which contains $c'$. Then $A'$ is a properly embedded essential annulus in $(X^\epsilon, S)$. Since $c$ is inessential in $\widehat S$, Lemma \ref{parallel} implies that $A$ is isotopic to $A'$ in $X^\epsilon$.
\qed

As mentioned in \S \ref{characteristic subsurfaces}, this corollary allows us to assume that $\Sigma_1^\epsilon$ is neatly embedded in $X^\epsilon$.

\begin{lemma} \label{root torus}
Let $\phi_1$ and $\phi_2$ be components of $\Phi_1^\epsilon$, possibly equal, and suppose that there are a component $c_1$ of $\partial \phi_1$, a component $c_2$ of $\partial \phi_2$, and an annulus $E \subseteq \overline{S \setminus \Phi_1^\epsilon}$ such that $\partial E = c_1 \cup c_2$. Then $E$ is essential in $\widehat S$.
\end{lemma}

\pf There are $I$-bundles $\Sigma_j \subseteq X^\epsilon$ such that $\phi_j \subseteq \Sigma_j \cap S$ is a component of the associated $S^0$-bundle ($j = 1, 2$).  Let $(A_j, \partial A_j) \subseteq (X^\epsilon, S)$ be the essential annulus in the frontier of $\Sigma_j$ in $X^\epsilon$ which contains $c_j$ ($j = 1, 2$).

If $A_1 = A_2$, then $A_1 \cup E$ is a torus in $X^\epsilon \subseteq M$ and so is either contained in a $3$-ball in $X^\epsilon$ or bounds a solid torus $\Theta \subseteq X^\epsilon$. Since $c_1$ is essential in $S$, the latter must occur, and since $A_1$ is an essential annulus in $(X^\epsilon, S)$, the winding number of $E$ in $\Theta$ is at least $2$. Thus $\Theta$ is a root torus of the type described. Proposition \ref{Mband} now implies that $E$ is essential in $\widehat S$.

Next suppose that $A_1 \ne A_2$, so these two annuli are disjoint. Note that they cannot be parallel as otherwise the $I$-bundle structures on $\Sigma_1$ and $\Sigma_2$ can be extended across an embedded $(E \times I, E \times \partial I) \subseteq (X^\epsilon, S)$, which contradicts the defining properties of $\Phi_1^\epsilon$. Hence Lemma \ref{parallel} implies that $E$ is essential in $\widehat S$.
\qed

\begin{prop} \label{parallel means essential}
Let $\phi_1$ and $\phi_2$ be components of $\breve{\Phi}_j^\epsilon$, possibly equal, and suppose that there are a component $c_1$ of $\partial \phi_1$, a component $c_2$ of $\partial \phi_2$, and an annulus $E \subseteq \overline{S \setminus \breve{\Phi}_j^\epsilon}$ such that $\partial E = c_1 \cup c_2$. Then $E$ is essential in $\widehat S$.
\end{prop}

\pf As $\breve{\Phi}_0^\epsilon = S$, there is an integer $k$ such that $1 \leq k \leq j$ and $E$ is contained in $\breve{\Phi}_{k-1}^\epsilon$ but not in $\breve{\Phi}_{k}^\epsilon$. Then $\phi_1 \cup \phi_2 \subseteq \breve{\Phi}_j^\epsilon \subseteq \breve{\Phi}_k^\epsilon$. Further, as $E$ is not contained in $\breve{\Phi}_{k}^\epsilon$, there must be inner boundary component of $\breve{\Phi}_k^\epsilon$, call it $c_0$, contained in $E$. If there is an arc $a$ in $E \cap \breve{\Phi}_{k}^\epsilon$ connecting $c_1$ and $c_2$, then $c_0 \subset (E \setminus a)$ and therefore $c_0$ is contained in a disk in $E \subset S$, which contradicts the essentiality of the inner components of $\partial \breve{\Phi}_{k}^\epsilon$ in $S$. Hence $(E, \partial E) \subseteq (\overline{\breve{\Phi}_{k-1}^\epsilon \setminus \breve{\Phi}_{k}^\epsilon}, \partial \breve{\Phi}_{k}^\epsilon)$. When $k = 1$ set $E_0 = E$ and when $k > 1$ set $E_0 = (\Pi_{i=1}^{k-1} \tau_{(-1)^{k - i - 1}\epsilon})(E)$ so that $(E_0, \partial E_0) \subseteq (\overline{S \setminus \Phi_1^{(-1)^{k-1}\epsilon}}, \partial \Phi_1^{(-1)^{k-1}\epsilon}).$ By Lemma \ref{root torus}, $E_0$ is $\widehat S$-essential, and therefore $E$ is as well.
\qed

\section{Pairs of embedded essential annuli in $(M, S)$} \label{pairs of essential annuli}

In this section we consider pairs of essential annuli lying  on either side of $S$ in $M$.

\begin{lemma}\label{nested-annuli} Suppose that $S = F$ and that there are
embedded, separating, essential annuli $(A^+, \partial A^+)\subseteq (X^+, F)$ and
 $(A^-, \partial A^-)\subseteq (X^-, F)$ such that $\partial A^+$ and $\partial A^-$ are four
parallel essential mutually disjoint curves in $\widehat F$. Then

$(1)$ $\partial A^\epsilon$ does not separate $\partial A^{-\e}$ in $\widehat F$.

$(2)$ Let $E$ be an annulus in $\widehat F$ bounded by a component of $\p
A^+$ and a component $\partial A^-$, with the interior of $E$ disjoint
from $A^+\cup A^-$. Then $|E \cap \partial M| = m/2$.
\end{lemma}

\pf For each $\epsilon$, the boundary $\partial A^\epsilon$ of $A^\epsilon$ separates
$\widehat F$ into two parallel essential  annuli, $E_1^\epsilon$ and
$E_2^\epsilon$, in $\widehat F$.

In order to prove the first assertion of the lemma, assume that $\partial A^\epsilon$ separates $\partial A^{-\e}$
in $\widehat F$, that is, $\partial A^-$ is not contained in $E_1^+$ or $E_2^+$. Then $|E_j^+ \cap \partial A^-| = 1$ for $j = 1, 2$. Hence $\partial A^-$ splits $E_1^+$
and $E_2^+$ into four  annuli, which we denote by $A_1, A_{2},
A_{3}, A_{4}$ with $A_1=E_1^+\cap E_1^-$, $A_2=E^+_1\cap E_2^-$,
$A_3=E_2^+\cap E_2^-$, $A_4=E_2^+\cap E_1^-$. Note that
$A_1,...,A_4$ are four parallel essential annuli in $\widehat F$ with
disjoint  interiors and with $A_1 \cup...\cup A_4=\widehat F$.

Suppose that $|A_{1}\cap \partial M| > 0$.
Then the torus $A^+\cup E_2^+$ bounds a solid torus $V^+$ in
$\widehat X^+$ (since it intersects $\partial M$ in fewer than $m$ components) such that $A^+$ is not parallel to $E_2^+$ in $V^+$ (Lemma \ref{three cases}). Similarly the torus $A^-\cup E_2^-$ bounds
a solid torus $V^-$ in $\widehat X^-$  such that $A^-$ is not
parallel to $E_2^-$ in $V^-$. Hence  $\displaystyle U=V^+\cup_{A_3}
V^-$ is a submanifold of $M(\beta)$ which is a Seifert fibred space
over the disk with two cone points. Also a core circle of $A_3$
is non-null homotopic in $M(\beta)$. If $\partial U$ compresses in $M(\beta)$, then  Lemma \ref{incomp} implies that
$V = \overline{M(\beta) \setminus U}$ is a solid torus. Hence $A_1$ is $\partial$-parallel in $V$ and therefore is isotopic to either $A^+ \cup A_2$ or to $A^- \cup A_4$ in $V$, contrary to construction. Thus $\partial
U$ is an incompressible torus in $M(\beta)$. But $\partial U=A_2\cup A^+\cup
A_4\cup A^-$ intersects $\partial M$ in fewer than $m$ components, which contradicts Assumption \ref{assumption minimal}. Thus $|A_{1}\cap \partial M| = 0$, and similarly $|A_{j}\cap \partial M| = 0$ for $j = 2, 3, 4$, which is impossible. This proves (1).

Next we prove the lemma's second assertion. By (1), we can suppose that $\partial A^-$ is contained in $E_1^+$ or $E_2^+$, say,  $E_1^+$. Then we
may assume that $E_1^-$ is contained in $E_1^+$ and that $E_2^+$ is
contained in $E_2^-$. Let $E, E_*$ be the two annulus components of
$E_2^-\cap E_1^+$. We need to show that $|E\cap \partial M|=|E_*\cap \p
M|=m/2$.

First we show that $|E_1^- \cap \partial M| = |E_2^+ \cap \partial M|=0$. Suppose
otherwise that $|E_1^-\cap \partial M|\ne 0$, say. Then  the torus
$A^+\cup E_2^+$ bounds a solid torus $V^+$ in $\widehat X^+$  such
that $A^+$ is not parallel to $E_2^+$ in $V^+$, and the torus
$A^-\cup E_2^-$ bounds a solid torus $V^-$ in $\widehat  X^-$ such
that $A^-$ is not parallel to $E_2^-$ in $V^-$. Hence
$U=V^+\cup_{E_2^+} V^-$ is a submanifold of $M(\beta)$ which is a
Seifert fibred space over the disk with two cone points. Also the
center circle of $E_2$ is non-null homotopic in $M(\beta)$. As in the proof of assertion (1), we can use Lemma \ref{incomp} to see that $\partial U$ is an incompressible torus in $M(\beta)$. But
$\partial U=A^-\cup E\cup A^+\cup E_*$ intersects $\partial M$ in fewer than $m$ components, contradicting Assumption \ref{assumption minimal}. Thus $|E_1^-\cap \partial M| = 0$ and a similar argument yields $|E_2^+ \cap \partial M|=0$. Hence $\partial F\subseteq E\cup E_*$.

Next we prove $|E\cap \partial M|=|E_*\cap \partial M|=m/2$. Suppose otherwise, say
$|E\cap \partial M|<m/2$ and $|E_*\cap \partial M|>m/2$. By the previous paragraph, the torus
$A^+\cup E_2^+$ bounds a solid torus $V^+$ in $\widehat X^+$ such
that $A^+$ is not parallel to $E_2^+$ in $V^+$, and the torus
$A^-\cup E_1^-$ bounds a solid torus $V^-$ in $\widehat  X^-$ such
that $A^-$ is not parallel to $E_1^-$ in $V^-$. Hence a regular
neighborhood $U$ of $V^+ \cup E \cup  V^-$ in $M(\beta)$ is a submanifold
of $M(\beta)$ which is a Seifert fibred space over the disk with two
cone points, and the core circle of $E$, which is contained in $U$, is non-null homotopic in
$M(\beta)$. As above, Lemma \ref{incomp} implies that $\partial U$ is incompressible in $M(\beta)$. But by construction, $|\partial U\cap \partial M|<m$, contradicting Assumption
\ref{assumption minimal}. Thus $|E\cap \partial M|=|E_*\cap \partial M|=m/2$, which completes the proof of the lemma.
 \qed

\begin{prop}\label{na2} Suppose that $S=F$ and that there are
embedded, separating, essential annuli $(A^+, \partial A^+)\subseteq (X^+, F)$ and
 $(A^-, \partial A^-)\subseteq (X^-, F)$ such that $\partial A^+$ and $\partial A^-$ are four
parallel essential mutually disjoint curves in $\widehat F$. Then no
component of $\partial A^+$ is isotopic in $F$ to a component of
$\partial A^-$.
\end{prop}

\pf Otherwise we may isotope $A^\epsilon$ in $X^\epsilon$, so that $\partial A^+$ and
$\partial A^-$ remain disjoint but $\partial A^+$  separates $\partial A^{-}$ in $\widehat
F$. This is impossible by Lemma \ref{nested-annuli}. \qed

\begin{prop} \label{na3}
Suppose that $S=F$ and that there is an
embedded M\"{o}bius band $(B,\partial B)\subseteq (X^\e, F)$. Then $\partial B$ cannot
be isotopic in $F$
 to a boundary component of an embedded, separating, essential annulus $(A, \partial A)\subseteq (X^{-\e}, F)$.
 \end{prop}

\pf Suppose otherwise. By Proposition \ref{Mband}, $\partial B$ is essential in
$\widehat F$. Let $P$ be a regular neighborhood of $B$ in $X^\epsilon$. Then
the frontier $A_*$ of $P$ in $X^\epsilon$ is an essential annulus in
$X^\epsilon$. Also $\partial A$ and $\partial A_*$ are essential curves in $\widehat F$
which can be assumed to be mutually disjoint since $\partial B$ is
isotopic in $F$ to a component of $\partial A$ and each component of $\p
A_*$ is isotopic to $\partial B$ in $F$. But such a situation is impossible
by Proposition \ref{na2}.
 \qed

\begin{lemma}\label{sep-non-sep nested-annuli}
Suppose that $S = F$ and that there are disjoint
embedded, essential annuli $(A^+, \partial A^+) \subseteq (X^+, F)$ and
 $(A^-, \partial A^-)\subseteq (X^-, F)$ such that $A^+$ is separating, $A^-$ is non-separating, and $\partial A^+$ and $\partial A^-$ are four parallel essential mutually disjoint curves in $\widehat F$ which split it into four annuli $E_1, E_2, E_3, E_4$ where $E_i \cap E_{i+1} \ne \emptyset$ for all $i$ (mod $4$).

$(1)$ Suppose that $\partial A^+$ does not separate $\partial A^-$ in $\widehat F$ and that the annuli $E_i$ are numbered so that $\partial A^+ = \partial E_1$ and $\partial A^- = \partial E_3$.
Then $|E_1 \cap \partial F| = 0$.

$(2)$ Suppose that $\partial A^+$ separates $\partial A^-$ in $\widehat F$ and that the annuli $E_i$ are numbered so that the components of $\partial A^+$ are $E_1 \cap E_2$ and $E_3 \cap E_4$. Then $|E_1 \cap \partial F| = |E_4 \cap \partial F|, |E_2 \cap \partial F| = |E_3 \cap \partial F|$, and $|E_1 \cap \partial F| + |E_2 \cap \partial F| = m/2$.
\end{lemma}

\pf The proof is based on Lemma \ref{non-sep}. Since we have assumed that $S = F$, conclusion (i) of this lemma holds.

Assume first that $\partial A^+$ does not separate $\partial A^-$ in $\widehat F$. Then Lemma \ref{non-sep} implies that $|E_3 \cap \partial F| = |E_1 \cap \partial F| + |E_2 \cap \partial F| + |E_4 \cap \partial F| = m/2$. Since $A^+ \cup E_2 \cup A^- \cup E_4$ is a Klein bottle, $|E_2 \cap \partial F| + |E_4 \cap \partial F| \geq m/2$ and therefore $|E_1 \cap \partial F| = 0$.

Next assume that $\partial A^+$ separates $\partial A^-$ in $\widehat F$. A tubular neighbourhood $U$ of the Klein bottle $A^+ \cup E_1 \cup A^- \cup E_3$ is a twisted $I$-bundle over the Klein bottle, and as no Dehn filling of $U$ is toroidal, the torus $\partial U$ must be incompressible in $M(\beta)$ (cf. Assumption \ref{assumption irreducible}). Hence $m \leq |\partial U \cap \partial M| \leq 2(|E_1 \cap \partial F| + |E_3 \cap \partial F|)$ and therefore $|E_1 \cap \partial F| + |E_3 \cap \partial F| \geq m/2$. Similarly, consideration of the Klein bottle $A^+ \cup E_2 \cup A^- \cup E_4$ shows that  $|E_2 \cap \partial F| + |E_4 \cap \partial F| \geq m/2$. On the other hand, Lemma \ref{non-sep} implies that $|E_1 \cap \partial F| + |E_2 \cap \partial F| = |E_3 \cap \partial F| + |E_4 \cap \partial F| = m/2$, from which we deduce the desired conclusion.
\qed

\section{The dependence of the number of tight components of $\breve{\Phi}_j^\epsilon$ on $j$} \label{section-tight}

Let $\mathcal{T}_j^\epsilon$ be the union of the tight components of $\breve{\Phi}_j^\epsilon$ and set
$$t_j^\epsilon = |\mathcal{T}_j^\epsilon|$$
If $j$ is odd, the free involution $h_j: \breve{\Phi}_j^\epsilon \to \breve{\Phi}_j^\epsilon$ preserves $\mathcal{T}_j^\epsilon$ but none of its components (cf. Proposition \ref{tight not invariant}). Thus $t_j^\epsilon$ is even for $j$ odd. Further, as $\breve{\Phi}_j^\epsilon \cong \breve{\Phi}_j^{-\epsilon}$ for $j$ even, $t_{2k}^+ = t_{2k}^-$ for all $k$.

\begin{lemma} \label{tight-in-disk}
Suppose that $C \subseteq (S \setminus \breve{\Phi}_j^\epsilon)$ is an essential simple closed curve which bounds a disk $D \subseteq \widehat S$. Then $D$ contains a tight component of $\breve{\Phi}_j^\epsilon$. Further, if $C$ is not isotopic in $S$ into the boundary of a tight component of $\breve{\Phi}_j^\epsilon$ (i.e. $D \cap S$ is not isotopic in $S$ to a tight component of $\breve{\Phi}_j^\epsilon$), then $D$ contains at least two tight components of $\breve{\Phi}_j^\epsilon$.
\end{lemma}

\pf Since $C$ is essential, $D$ contains at least one boundary component of $S$ and hence at least one component of $\breve{\Phi}_j^\epsilon$. Amongst all the inner boundary components of $\breve{\Phi}_{j}^\epsilon$ which are contained in $D$, choose one, $C_1$ say, which is innermost in $D$. It is easy to see that this circle is the inner boundary component of a tight component $\phi_1$ of $\breve{\Phi}_{j}^\epsilon$. This proves the first assertion of the lemma.

Next suppose that $C$ is not isotopic in $S$ into the boundary of a tight component of $\breve{\Phi}_j^\epsilon$. Then $C$ and $C_1$ do not cobound an annulus in $D \cap S$, so there is a component of $\partial S$ contained in $\overline{D \setminus \phi_1}$. Hence if $\phi_1, \phi_2, \ldots , \phi_n$ are the components of $\breve{\Phi}_{j}^\epsilon$ contained in $D \cap S$, then $n \geq 2$. If every inner boundary component of $\phi_2 \cup \phi_3 \cup  \ldots \cup  \phi_n$ is essential in the annulus $\widehat{\overline{D \setminus \phi_1}}$, some such boundary component cobounds an annulus $E \subseteq S$ with $C_1$. Without loss of generality we may suppose $\partial E = C_1 \cup C_2$ where $C_2 \subseteq \partial \phi_2$. But this is impossible as Proposition \ref{parallel means essential} would then imply that $E$ is essential in $\widehat S$. Hence some inner boundary component of $\phi_2 \cup \phi_3 \cup  \ldots \cup  \phi_n$ bounds a subdisk $D'$ of $D$ which is disjoint from $\phi_1$, the argument of the first paragraph of this proof shows that $D$ contains another tight component of $\breve{\Phi}_{j}^\epsilon$, so we are done.
\qed

An immediate consequence of the lemma is the following corollary.

\begin{cor} \label{tight-in-tight} $\;$

$(1)$ If $\breve{\Phi}_j^\epsilon$ has a component $\phi$ which is contained in a disk $D \subseteq \widehat S$, then either $\phi$ is tight or $D$ contains at least two tight components  of $\breve{\Phi}_j^\epsilon$.

$(2)(a)$ If $\phi_0$ is a tight component of $\breve{\Phi}_j^\epsilon$, there is a tight component $\phi_1$ of $\breve{\Phi}_{j+1}^\epsilon$ contained in $\phi_0$.

\indent \hspace{4mm} $(b)$ If $\phi_1$ is not isotopic to $\phi_0$ in $S$, there are at least two tight components of $\breve{\Phi}_{j+1}^\epsilon$ contained \indent \hspace{4mm} in $\phi_0$.
\qed
\end{cor}

\begin{prop} \label{tight increase} $\;$

$(1)(a) $ $t_j^\epsilon \leq t_{j+1}^\epsilon$ with equality if and only if $\mathcal{T}_j^\epsilon$ is isotopic to $\mathcal{T}_{j+1}^\epsilon$ in $S$.

\indent \hspace{3.5mm} $(b)$ $t_j^\epsilon \leq t_{j+1}^{-\epsilon}$

$(2)$ If $0 < t_j^\epsilon = t_{j+2}^\epsilon$, then $t_j^\epsilon = |\partial S|$, so $\mathcal{T}_j^\epsilon$ is a regular neighbourhood of $\partial S$.
\end{prop}

\pf Part (1)(a) follows immediately from Corollary \ref{tight-in-tight}. For part (1)(b), note that if $j$ is odd then $t_j^\epsilon \leq t_{j+1}^{\epsilon} = t_{j+1}^{-\epsilon}$, while if $j$ is even, $t_j^\epsilon = t_j^{-\epsilon} \leq  t_{j+1}^{-\epsilon}$.

Next we prove part (2). Suppose that $0 < t_j^\epsilon = t_{j+2}^\epsilon$. Then Lemma \ref{tight-in-disk} implies that up to isotopy, $\mathcal{T}_j^\epsilon = \mathcal{T}_{j+2}^\epsilon$. We claim that $(\tau_{-\epsilon} \tau_\epsilon)(\mathcal{T}_{j+2}^\epsilon) = \mathcal{T}_{j}^\epsilon$, at least up to isotopy fixed on $\partial S$. To see this, first note that $(\tau_{-\epsilon} \tau_\epsilon)(\breve{\Phi}_{j+2}^\epsilon) \subseteq \breve{\Phi}_{j}^\epsilon$. Fix a tight component $\phi_0$ of $\breve{\Phi}_{j}^\epsilon$ and let $\phi_1, \phi_2, \ldots, \phi_n$ be the components of $\breve{\Phi}_{j+2}^\epsilon$ such that for each $i = 1, 2, \ldots , n$, $\phi_i' = (\tau_{-\epsilon} \tau_\epsilon)(\phi_i) \subseteq \phi_0$. Since each component of $\partial S \cap \phi_0$ is contained in some $\phi_i'$, the argument of the first paragraph of the proof of Lemma \ref{tight-in-disk} shows that at least one of the $\phi_i'$, or equivalently $\phi_i$, is tight. Since $\phi_0$ is an arbitrary tight component of $\breve{\Phi}_j^\epsilon$ and $t_j^\epsilon = t_{j+2}^\epsilon$,
it follows that $(\tau_{-\epsilon} \tau_\epsilon)(\mathcal{T}_{j+2}^\epsilon) \subseteq \mathcal{T}_{j}^\epsilon$ and each component of $\mathcal{T}_{j}^\epsilon$ contains a unique component of $(\tau_{-\epsilon} \tau_\epsilon)(\mathcal{T}_{j+2}^\epsilon)$. Note as well that as $\mathcal{T}_j^\epsilon = \mathcal{T}_{j+2}^\epsilon$, we have $|\partial S \cap \mathcal{T}_j^\epsilon| = |\partial S \cap \mathcal{T}_{j+2}^\epsilon| = |\partial S \cap (\tau_{-\epsilon} \tau_\epsilon)(\mathcal{T}_{j+2}^\epsilon)|$, so if $\phi_1$ is a tight component of $\mathcal{T}_{j+2}^\epsilon$ and $\phi_0$ the tight component of $\mathcal{T}_j^\epsilon$ containing $\phi_1' = (\tau_{-\epsilon} \tau_\epsilon)(\phi_1)$, then $|\partial S \cap \phi_1'| = |\partial S \cap \phi_0|$. But then as $\phi_0$ and $\phi_1'$ are tight, $\phi_1'$ is isotopic to $\phi_0$ by an isotopy fixed on $\partial S$. Hence we can assume that $(\tau_{-\epsilon} \tau_\epsilon)(\mathcal{T}_{j+2}^\epsilon) = \mathcal{T}_{j}^\epsilon$, or in other words, $(\tau_{-\epsilon} \tau_\epsilon)(\mathcal{T}_{j}^\epsilon) = \mathcal{T}_{j}^\epsilon$. It follows that $(\tau_{-\epsilon} \tau_\epsilon)^k(\mathcal{T}_{j}^\epsilon) = \mathcal{T}_{j}^\epsilon \subseteq \breve{\Phi}_j^\epsilon$, and so $\mathcal{T}_{j}^\epsilon \subseteq \breve{\Phi}_{j+2k}^\epsilon$ for all $k$. But since $F$ is neither a fibre nor a semi-fibre, $\breve{\Phi}_{j+2k}^\epsilon$ is a regular neighbourhood of $\partial S$ for large $k$ (cf.  \cite[proof of Theorem 5.4.1]{BCSZ1}). Thus $\mathcal{T}_{j}^\epsilon$ is a union of annuli.

The boundary components of $S$ can be numbered $b_1, b_2, \ldots , b_{|\partial S|}$ so that they arise successively around $\partial M$ and $(\tau_{-\epsilon} \tau_\epsilon)(b_i) = b_{i + (-1)^i 2}$, where the indices are considered (mod $|\partial S|$). Hence as $(\tau_{-\epsilon} \tau_\epsilon)(\mathcal{T}_{j}^\epsilon) = (\tau_{-\epsilon} \tau_\epsilon)(\mathcal{T}_{j+2}^\epsilon) = \mathcal{T}_{j}^\epsilon$, $\partial \mathcal{T}_j^\epsilon \cap \partial S$ is the union of either all even-indexed $b_i$, or all odd-indexed $b_i$, or all the $b_i$ (recall that we have assumed $t_j^\epsilon > 0$). In particular, for either all even $i$ or all odd $i$, the component of $\breve{\Phi}_j^\epsilon$ containing $b_i$ is an annulus. If $j$ is odd, we have a free involution $\breve h_j^\epsilon: \breve{\Phi}_{j}^\epsilon \to \breve{\Phi}_{j}^\epsilon$ which the reader will verify preserves $\mathcal{T}_{j}^\epsilon$ and exchanges the even-indexed $b_i$ with the odd-indexed $b_i$. Hence for any $i$, the component of $\breve{\Phi}_j^\epsilon$ containing $b_i$ is an annulus. Thus $t_j^\epsilon = m$, so $\mathcal{T}_j^\epsilon$ is a regular neighbourhood of $\partial S$.

Next suppose that $j$ is even. After possibly adding $1$ (mod $|\partial S|$) to the indices of the labels of the components of $\partial S$, we can assume that $\partial \mathcal{T}_j^\epsilon \cap \partial S$ contains the union of all even-indexed $b_i$. Then as $\breve{\Phi}_{j+1}^{-\epsilon} = \tau_{-\epsilon}(\breve{\Phi}_j^\epsilon \wedge \breve{\Phi}_1^{-\epsilon}) \supseteq \tau_{-\epsilon}(\mathcal{T}_{j}^{\epsilon} \wedge \breve{\Phi}_1^{-\epsilon})$, the component of $\mathcal{T}_{j+1}^{-\epsilon}$ containing an odd-indexed $b_i$ is an annulus. Consideration of the free involution $\breve h_{j+1}^{-\epsilon}: \breve{\Phi}_{j+1}^{-\epsilon} \to \breve{\Phi}_{j+1}^{-\epsilon}$ shows that the same is true for the even-indexed $b_i$. Thus $m = t_{j+1}^{-\epsilon} \leq t_{j+2}^\epsilon = t_j^\epsilon$. It follows that $\mathcal{T}_j^\epsilon$ is a regular neighbourhood of $\partial S$.
\qed

\begin{cor} \label{inessential non-tight} $\;$

$(1)$ If some non-tight component of $\breve{\Phi}_1^\epsilon$ has an $\widehat S$-inessential inner boundary component, then $t_1^\epsilon \geq 4$.

$(2)$ If $\hbox{genus}(\breve \Phi_1^\epsilon) = 1$ but $\breve \Phi_1^\epsilon \ne S$, then $t_1^\epsilon \geq 4$.
\end{cor}

\pf Let $\phi$ be a non-tight component of $\breve{\Phi}_1^\epsilon$ and $c$ an $\widehat S$-inessential inner boundary component of $\phi$. Let $D \subseteq \widehat S$ be the disk with boundary $c$. Lemma \ref{tight-in-disk} implies that $D$ contains a tight component $\phi_0$. By Lemma \ref{root torus}, $c$ is not isotopic in $S$ into the boundary of $\phi_0$, so $D$ contains at least two tight components (Lemma \ref{tight-in-disk}). It follows from Lemma \ref{three cases}(3) that $\tau_\epsilon(c)$ bounds a disk $D' \subseteq \widehat S$ disjoint from $D$ and as above, $D'$ contains at least two tight components of $\breve{\Phi}_1^\epsilon$. Thus $t_1^\epsilon \geq 4$, which proves (1).

Next suppose that $\hbox{genus}(\breve{\Phi}_1^\epsilon) = 1$ but $\breve{\Phi}_1^\epsilon \ne S$. Then there is a component $\phi \ne S$ of $\breve{\Phi}_1^\epsilon$ of genus $1$. In particular, $\phi$ is not tight. Since $\phi \ne S$, it has inner boundary components. Since $\hbox{genus}(\phi) = 1$, each such inner boundary component is $\widehat S$-inessential. Hence part (2) of the corollary follows from part (1).
\qed

\section{The structure of $\dot{\Phi}_1^\epsilon$ and the topology of $X^\epsilon$} \label{NT}

In this section we study how the existence of a component $\phi$ of $\dot{\Phi}_1^\epsilon$ such that $\widehat \phi$ contains an $\widehat S$-essential annulus constrains $\dot{\Phi}_1^\epsilon$ and the topology of $\widehat X^\epsilon$. The following construction will be useful to our analysis.

Let $P$ be a component of $\dot{\Sigma}_1^\epsilon$. For each $\widehat S$-inessential inner component of $c$ of $P \cap S$ let $D_c \subset \widehat S$ be the disk with boundary $c$ and suppose that the component of $P \cap S$ containing $c$ is disjoint from $\hbox{int}(D_c)$. The component $H_c$ of $\overline{\widehat X^\epsilon \setminus P}$ containing $c$ satisfies $(H_c, H_c \cap P) \cong (D^2 \times I, (\partial D^2) \times I)$ (cf. Proposition \ref{extend fibering}). Let
\begin{equation}\label{definition of qp}
\text{\em $Q_P = \widehat P \cup (\cup_{c} H_c)$}
\end{equation}
where $c$ ranges over all $\widehat S$-inessential inner components of $P \cap S$ such that $P \cap \hbox{int}(D_c) = \emptyset$. The $I$-fibre structure on $P$ extends over $Q_P$.

We prove the following results.

\begin{prop} \label{sep-seifert}
Suppose that $F$ is separating in $M$, so $S = F$ is connected.

$(1)$ There is at most one component $P$ of $\dot{\Sigma}_1^\epsilon$ such that $\widehat{P \cap F}$ contains an $\widehat F$-essential annulus.

$(2)$ There is exactly one such component if  $\dot{\Sigma}_1^\epsilon$ contains a twisted $I$-bundle.

$(3)$ Suppose that $P$ is a component of $\dot{\Sigma}_1^\epsilon$ such that $\widehat{P \cap F}$ contains an $\widehat F$-essential annulus. Then $\widehat X^\epsilon$ admits a Seifert structure and if

\hspace{3mm} $(a)$ $\hbox{genus}(\widehat{P \cap F}) = 1$, then $\widehat X^\epsilon$ is a twisted $I$-bundle over the Klein bottle.

\hspace{3mm} $(b)$ $\hbox{genus}(\widehat{P \cap F}) = 0$, then an $\widehat F$-essential annulus in $\widehat{P \cap F}$ is vertical in the Seifert structure \\ \indent \hspace{3mm} and $Q_P$ splits $\widehat X^\epsilon$ into
a union of solid tori. Moreover, if

\hspace{9mm} $(i)$ $P$ is a twisted $I$-bundle, then $\widehat X^\epsilon$ has base orbifold a disk with two cone points, at \\ \indent \hspace{9mm} least one of which has order $2$.

\hspace{9mm} $(ii)$ $P$ is a product $I$-bundle and $Q_P$ separates $\widehat X^\epsilon$, then $\widehat X^\epsilon$ has base orbifold a disk with \\ \indent \hspace{9mm} two cone points.

\hspace{9mm} $(iii)$ $P$ is a product $I$-bundle and $Q_P$ does not separate $\widehat X^\epsilon$, then $X^-$ is a twisted $I$-bundle \\ \indent \hspace{9mm} and $\widehat X^\epsilon$ has base orbifold a M\"{o}bius band with at most one cone point.
\end{prop}

\begin{prop} \label{nonsep-seifert}
Suppose that $F$ is non-separating in $M$, so $S = F_1 \cup F_2$ is not connected.

$(1)$ $\dot{\Sigma}_1^+$ is a $($possibly empty$)$ product bundle and for each $j = 1, 2$ and component $P$ of $\dot{\Sigma}_1^+$, $\hbox{genus}(P \cap F_j) = 0$.

$(2)$ If $P$ is a component of $\dot{\Sigma}_1^+$ such that $P \cap S \subseteq P \cap F_j$ for some $j$, then $\widehat{P \cap F_j}$ contains no $\widehat S$-essential annulus.

$(3)$ If $t_1^+ = 0$, then $\dot{\Sigma}_1^+$ has exactly one component $P$ and for each $j = 1, 2$, $\widehat{P \cap F_j}$ is an annulus which is essential in $\widehat F_j$.  Further, $\widehat X^+$ admits a Seifert structure with base orbifold an annulus with exactly one cone point.
\end{prop}

We consider the cases $S$ connected and $S$ disconnected separately.

\subsection{$S$ is connected}

In this subsection we prove Proposition \ref{sep-seifert}.

\begin{lemma} \label{seifert}
Suppose that $F$ is separating in $M$ and $(A, \partial A) \subseteq (X^\epsilon, F)$ is an essential separating annulus whose boundary separates $\widehat F$ into two annuli $E_1$ and $E_2$. If $|E_1 \cap \partial M| < m$ then $A \cup E_1$ bounds a solid torus in $\widehat X^\epsilon$ in which $A$ has winding number at least $2$. Hence if either

\indent \hspace{3mm} $(a)$ $|E_1 \cap \partial M| < m$ and $|E_2 \cap \partial M| < m$; or

\indent \hspace{3mm} $(b)$ $|E_1 \cap \partial M| = m$ and $E_1 \cup A$ bounds a solid torus in $\widehat X^\epsilon$,

then $A$ splits $\widehat X^\epsilon$ into two solid tori in each of which $A$ has winding number at least $2$. In particular, $\widehat X^\epsilon$ admits a Seifert structure with base orbifold a disk with two cone points in which $A$ is vertical.
\end{lemma}

\pf If $|E_1 \cap \partial M| < m$, then $A \cup E_1$ is a torus which compresses in $M(\beta)$ but is not contained in a $3$-ball. Hence it bounds a solid torus $V$ which is necessarily contained in $\widehat X^\epsilon$. Lemma \ref{three cases}(2) shows that $A$ has winding number at least $2$ in $V$. It follows that if condition (a) holds, $\widehat X^\epsilon$ admits a Seifert structure with base orbifold a $2$-disk with two cone points. Note that $A$ is vertical in this structure and splits $\widehat X^\epsilon$ into two solid tori. A similar argument yields the same conclusion under condition (b).
\qed

{\bf Proof of part (3) of Proposition \ref{sep-seifert}}. First suppose that $\hbox{genus}(\widehat{P \cap F}) = 1$. Then $P$ is necessarily a twisted $I$-bundle and $\phi = \widehat{P \cap F}$ is connected. Further, each inner boundary component of $P \cap F$ is inessential in $\widehat F$. Thus $\widehat X^\epsilon = Q_P$ is a twisted $I$-bundle over the Klein bottle. Hence part (3)(a) of Proposition \ref{sep-seifert} holds.

Next suppose that $\hbox{genus}(\widehat{P \cap F}) = 0$ and $\phi$ is a component of $P \cap F$. Then $\phi$ has two inner boundary components, $c_1, c_2$ say, which are $\widehat F$-essential. Any other inner boundary component $c$ of $\phi$ is inessential in $\widehat F$ so $Q_P$ is either a twisted $I$-bundle over a M\"{o}bius band or product $I$-bundle over an annulus.

Let $A_1, A_2$ be the vertical annuli in the frontier of $P$, possibly equal, that contain $c_1, c_2$ respectively There are three cases to consider.

\noindent {\bf Case 1}. $P$ is a twisted $I$-bundle.

In this case, $A_1 = A_2$ and $Q_P$ is a twisted $I$-bundle over a M\"{o}bius band. In particular $Q_P$ is a solid torus in which a core of $\widehat \phi$ has winding number $2$. Lemma \ref{seifert} shows that Proposition \ref{sep-seifert}(3)(b)(i) holds.

\noindent {\bf Case 2}. $P$ is a product $I$-bundle and $Q_P$ separates $X^\epsilon$.

Then $A_1 \ne A_2$ where $A_1$ is separating in $X^\epsilon$ and $Q_P$ is a product $I$-bundle over an annulus.
Let $V$ and $W$ be the components of the exterior of $Q_P$ in $\widehat X^\epsilon$ and define $E_1, E_2$ to be the $\widehat F$-essential annuli $V \cap \widehat F,  W \cap \widehat F$. Since $|P \cap \partial F| > 0$, we have $|E_1 \cap \partial F| < m$ and $|E_2 \cap \partial F| < m$. Then Lemma \ref{seifert} implies that both $V$ and $W$ are solid tori and therefore that Proposition \ref{sep-seifert}(3)(b)(ii) holds.

\noindent {\bf Case 3}. $P$ is a product $I$-bundle and $Q_P$ does not separate $X^\epsilon$.

Here $A_1 \ne A_2$ where $A_1$ is non-separating and $X^-$ is a twisted $I$-bundle by Lemma \ref{non-sep}. Also the boundary of the complement of the interior of $Q_P$ in $X^\epsilon$ is a torus which intersects $\partial M$ in fewer than $m$ components but is not contained in any $3$-ball in $M(\beta)$. Thus it bounds a solid torus $V$ in $\widehat X^\epsilon$, from which we can see that Proposition \ref{sep-seifert}(3)(b)(iii) holds.
\qed

{\bf Proof of parts (1) and (2) of Proposition \ref{sep-seifert}}.
If $\dot{\Sigma}_1^\epsilon$ contains a twisted $I$-bundle $P$, then $P$ contains a subbundle homeomorphic to a M\"{o}bius band. Proposition \ref{Mband} then shows that $P \cap F$ contains an $\widehat F$-essential annulus. Thus part (1) implies part (2). We prove part (1) by contradiction.

Suppose that $\dot{\Sigma}_1^\epsilon$ has at least two
components $P_1, P_2$ such that $\widehat{P_i \cap F}$ contains an $\widehat F$-essential annulus for $i = 1, 2$. Let $\phi_i = P_i \cap F$ be the horizontal
boundary of $P_i$ ($i=1, 2$). Clearly, both $\phi_1$ and $\phi_2$ have genus $0$. Since each properly embedded  incompressible annulus in a solid torus is separating, Proposition \ref{sep-seifert}(3), which we proved above, implies that both $Q_{P_1}$ and $Q_{P_2}$ are separating in $\widehat X^\epsilon$ and split it into a union of solid tori.
We have three cases to consider.

\noindent{\bf Case 1}. $\phi_i$  is connected for $i=1,2$.

Then  $Q_{P_i}$ is a twisted $I$-bundle over a M\"{o}bius band whose frontier in $\widehat X^\epsilon$ is an essential annulus $A_i$ in $X^\epsilon$ which is
not parallel to the annulus $\psi_i = Q_{P_i} \cap \widehat F$ in $Q_{P_i}$. Let $E_1, E_2$
be the components of the closure of the complement of $\psi_1\cup
\psi_2$ in $\widehat F$. Then the torus $E_1\cup A_1\cup
E_2\cup A_2$ bounds a solid torus $V$ in $\widehat{X}^\epsilon$ in which
$A_1$ is not parallel to $A_2$. Therefore  $U = V\cup \widehat
P_1$  satisfies the hypotheses of Lemma \ref{incomp}. Since $\widehat F$ is isotopic into
$\overline{M(\beta) \setminus U}$, the latter cannot be a solid torus. Thus $\partial
U = \psi_1\cup E_1\cup A_2\cup E_2$ is incompressible in
$M(\beta)$. But this torus intersects $\partial M$ in fewer than $m$ components, which contradicts Assumption \ref{assumption minimal}.

\noindent{\bf Case 2}. $\phi_1$ is connected but $\phi_2$ is not.

Then  $Q_{P_1}$ is a twisted $I$-bundle over a M\"{o}bius band and
the frontier of $Q_{P_1}$ in $\widehat X^\epsilon$ is an essential annulus $A_1 \subseteq X^\epsilon$ which
is not parallel to the annulus $\psi_1 = Q_{P_1} \cap \widehat F$ in $Q_{P_1}$. Further, $\phi_2$
has two components, $\phi_{21}, \phi_{22}$ say, and $P_2$ is a product
$I$-bundle over $\phi_{21}$. The frontier of $Q_{P_2}$ is a pair of
essential annuli $A_{21}, A_{22} \subseteq X^\epsilon$. We noted above that $Q_{P_2}$ is separating in
$X^\epsilon$, and so the same is true for $A_{21}$ and $A_{22}$.

We may suppose that $A_{21}$ is adjacent to $A_1$. That is, $\partial A_1 \cup \partial
A_{21}$ cobounds the union of two disjoint annuli $E_1, E_2 \subseteq \widehat F$ whose interiors are
disjoint from $\phi_1, \phi_2$. Then the torus $A_1\cup E_1\cup
A_{21}\cup E_2$ bounds a solid torus
 $V$ in $\widehat{X}^\epsilon$ such that
$A_1$ is not parallel to $A_{21}$ in $V$. Therefore $U=V\cup
\widehat P_1$ is a submanifold of $M(\beta)$ satisfying the hypotheses
of Lemma \ref{incomp}. As in case 1, this lemma implies that
$\partial U=\psi_1\cup E_1\cup
A_{21}\cup E_2$ is incompressible in $M(\beta)$. But
this torus intersects $\partial M$ in fewer than $m$ components, contrary to Assumption \ref{assumption minimal}.

\noindent{\bf Case 3}. Neither $\phi_1$ nor $\phi_2$ is connected.

The frontier of $Q_{P_i}$ in $\widehat X^\epsilon$ is a pair of annuli $A_{i1},
A_{i2}$ contained in $X^\epsilon$. We may assume that $\partial A_{12}$ and $\partial A_{21}$ cobound two
annuli $E_1, E_2$ in $\widehat F$ whose interiors are disjoint from
$\phi_1\cup \phi_2$. The torus $A_{12}\cup E_1\cup A_{21}\cup
E_2$ bounds a solid torus
 $V$ in $\widehat{X}^\epsilon$ in which
$A_{12}$ is not parallel to $A_{21}$. Let $E_*$ be the annulus in $\widehat F$ with $\partial E_* = \partial A_{11}$ and whose interior
is disjoint from $\phi_1 \cup \phi_2$. The torus $A_{11}\cup E_*$ bounds
a solid torus $V_*$ in $\widehat X^\epsilon$ in which $A_{11}$ is not
parallel to $E_*$. Therefore $U = V_*\cup Q_{P_1} \cup V$
is a submanifold of $M(\beta)$ satisfying the hypotheses of Lemma
\ref{incomp} and as above, this lemma implies that $\partial U=E_*\cup (Q_{P_1} \cap \widehat F)\cup E_1\cup
E_2\cup A_{21}$ is incompressible in $M(\beta)$. But this is impossible as $|\partial U \cap \partial
M|<m$.
\qed

We can refine Proposition \ref{sep-seifert} somewhat in the absence of tight components of $\breve{\Phi}_1^\epsilon$.

\begin{lemma} \label{t1e=0 implies dot nonempty}
If $\dot{\Sigma}_1^\epsilon = \emptyset$, then $t_1^\epsilon = |\partial S|$.
\end{lemma}

\begin{proof}
If $\dot{\Sigma}_1^\epsilon$ is empty, then so is $\dot{\Phi}_1^\epsilon$ and therefore $\breve{\Phi}_1^\epsilon$ is a collar on $\partial S$, so $t_1^\epsilon = |\partial S|$.
\end{proof}

\begin{prop} \label{order two}
When $F$ is separating and $t_1^\epsilon = 0$, then $\dot{\Sigma}_1^\epsilon$ has a unique component $P$ and either $P = X^\epsilon$ or each component of $\dot{\Phi}_1^\epsilon = P \cap F$ completes to an essential annulus in $\widehat F$. Further, the base orbifold of the Seifert structure on $\widehat X^\epsilon$ described in Proposition \ref{sep-seifert} has

$(1)$ no cone points of order $2$ if $P$ is a product $I$-bundle,

$(2)$ one cone point of order $2$ if $P$ is a twisted $I$-bundle and $\dot{\Phi}_1^\e \ne F$,

$(3)$ two cone points of order $2$ if $P = X^\epsilon$, i.e. $X^\epsilon$ is a twisted $I$-bundle so $\epsilon = -$.
\end{prop}

\pf Since $t_1^\epsilon = 0$, $\dot{\Sigma}_1^\epsilon$ has at least one component (Lemma \ref{t1e=0 implies dot nonempty}) and each inner boundary component of $\dot{\Phi}_1^\epsilon$ is $\widehat F$-essential (Corollary \ref{inessential non-tight}). Proposition \ref{sep-seifert} then shows that $\dot{\Sigma}_1^\epsilon$ has exactly one component. Call it $P$. Proposition \ref{sep-seifert} also shows that either $P = X^\epsilon$ or $X^\epsilon \setminus P$ is a union of solid tori. Since the $I$-bundle structure on $P$ does not extend over these solid tori, the result follows.
\qed

\begin{cor} \label{both not (2,2)}
If $F$ is separating and $t_1^+ = 0$, the base orbifold of the Seifert structure on $\widehat X^+$ described in Proposition \ref{sep-seifert} is $D^2(a, b)$ where $(a, b) \ne (2,2)$. Further, $M(\beta)$ is not a union of two twisted $I$-bundles over the Klein bottle.
\end{cor}

\pf The first assertion follows from part (3) of the previous proposition. Suppose that $M(\beta)$ is a union of two twisted $I$-bundles over the Klein bottle along their common boundary $T$. Then $T$ is not isotopic to $\widehat F$ by the first assertion. Hence as $T$ splits $M(\beta)$ into two atoroidal Seifert manifolds, $M(\beta)$ must be Seifert. If $\widehat F$ is horizontal, it splits  $M(\beta)$ into two twisted $I$-bundles, necessarily over the Klein bottle, which contradicts Assumption \ref{assumption not (semi) fibre}. Thus it is vertical and $T$ is horizontal. It follows that the base orbifold $\mathcal{B}$ of $M(\beta)$ is Euclidean. Further, $\mathcal{B}$ is non-orientable as $T$ separates. Thus $\mathcal{B}$ is either a Klein bottle or $P^2(2,2)$. In either case $\widehat F$ splits $M(\beta)$ into the union of two twisted $I$-bundles over the Klein bottle, contrary to the first assertion of the corollary. This completes the proof.
\qed

\subsection{$S$ is not connected} In this subsection we prove Proposition \ref{nonsep-seifert}. It will follow from the four lemmas below.

\begin{lemma}\label{product-Sigma+}
When $F$ is non-separating, $\dot{\Sigma}_1^+$ is a $($possibly empty$)$ product $I$-bundle.
\end{lemma}

\pf Suppose that $\dot{\Phi}_1^+$ has a $\tau_+$-invariant component, $\phi$ say. Then there is a M\"{o}bius band $(B, \partial B) \subseteq (X^+, \phi)$. According to Proposition \ref{Mband}, $\partial B$ is essential in $\widehat S$. Our hypotheses imply that $\phi/\tau_+$ contains a once-punctured M\"{o}bius band. Its inverse image in $\widehat S$ is a $\tau_+$-invariant twice-punctured annulus $\phi_0 \subseteq \phi$ such that $\widehat \phi_0$ is essential in $\widehat S$. Without loss of generality we can suppose that $\widehat \phi_0 \subseteq \widehat F_1$.

Now $\phi_0$ has at least two outer boundary components and two inner ones. We denote the latter by
$c_1, c_2$. By construction $c_2 = \tau_+(c_1)$ and $c_1$ and $c_2$ cobound
an essential annulus $A$ in $(X^+, F_1)$. Note that $E = \overline{\widehat F_1 \setminus \widehat \phi_0}$ is an annulus and $A \cup E$ a non-separating torus in $M(\beta)$ which intersects $\partial M$ in fewer than $m$ components. Hence it is compressible. But then $M(\beta)$ contains a non-separating $2$-sphere, which is impossible by Assumption \ref{assumption irreducible}. Thus there is no $\tau_+$-invariant component of $\dot{\Phi}_1^+$.
\qed

\begin{lemma} \label{one-to-other}
Suppose $F$ is non-separating. Let $P$ be a component of
$\dot{\Sigma}_1^+$ and let $\phi_1,\phi_2 \subseteq S$ be the two horizontal
boundary components of $P$. If $\phi_1$ contains an $\widehat S$-essential annulus, then $\phi_1$ and $\phi_2$ are contained in different components of $S$.
\end{lemma}

\pf Suppose that both $\phi_1$ and $\phi_2$ are contained
in $F_1$, say. Choose a neat subsurface $\phi_{1,0}$ of $\phi_1$ such that $\widehat \phi_{1,0}$ is an $\widehat S$-essential annulus and $|\phi_{1,0} \cap \partial M| > 0$.
Set $\phi_{2,0} = \tau_+(\phi_{1,0})$.
Then $\widehat \phi_{1,0}$ and $\widehat\phi_{2,0}$ are
disjoint essential annuli  in $\widehat F_1$. The frontier
of $P$ in $X^+$  is a set of two essential annuli in $(X^+,F_1)$,
which we denote by $A_1$ and $A_2$. According to Lemma \ref{non-sep},
each $A_i$ is separating in $X^+$. For $i=1,2$, $\partial A_i$
bounds an annulus $E_i$ in $\widehat F_1$ whose interior is disjoint
from $\phi_{1,0} \cup \phi_{2,0}$.

The annulus $A_1$ splits $\widehat X^+$ into two components, which
we denote by $W_1$ and $W_2$. We may suppose that the torus $A_1\cup
E_1$ is a boundary component of $W_1$. Now $\widehat F_2 \subseteq \partial W_i$
for some $i$, and in this case $A_i \cup E_i$ is a non-separating
torus in $M(\beta)$ whose intersection with $\partial M$ has fewer than $m$ components, contrary to
Assumption \ref{assumption minimal}. Thus the conclusion of the lemma holds.
\qed

\begin{lemma} \label{4}
If there is a component $P$ of $\dot{\Sigma}_1^+$ and $j \in \{1, 2\}$ such that $|P \cap F_j| = 2$, then $t_1^+ \geq 4$.
\end{lemma}

\pf Without loss of generality, we can suppose that $j = 1$. Lemma \ref{one-to-other} implies that no component $\phi$ of $P \cap F_1$ contains an $\widehat S$-essential annulus. Thus there is a disk in $\widehat F_1$ containing $\phi$ and this disk must contain a tight component of $\breve{\Phi}_1^\epsilon$. The same is true for the other component of $P \cap F_1$, so the number of tight components of $\breve{\Phi}_1^+$ contained in $F_1$ is at least $2$. To see that the same is true for $F_2$, it suffices to show that there is a component $P'$ of $\dot{\Sigma}_1^+$ such that $|P' \cap F_2| = 2$. But it is clear that such a component exists since Lemma \ref{product-Sigma+} implies that the number of boundary components of $F_1$ contained in a component of $\dot{\Sigma}_1^+$ which intersects both $F_1$ and $F_2$ equals the number of such boundary components of $F_2$.
\qed

\begin{lemma}\label{one-comp}
Suppose $F$ is non-separating and $t_1^+ = 0$. Then there is a unique component $P$ of $\dot{\Sigma}_1^+$ such that $\widehat{P \cap F_1}$ contains an $\widehat F_1$-essential annulus. Further, $\widehat X^+$ admits a Seifert structure in which $\widehat{\dot{\Phi}_1^+}$ is vertical and whose base orbifold is an annulus with exactly one cone point.
\end{lemma}

\pf First observe that $\dot{\Sigma}_1^+$ has at least one component, $P$ say, since $t_1^+ = 0$. By Lemma \ref{4},
$|P \cap F_j| = 1$ for each $j$. Set $\phi_j = P \cap F_j$. Corollary \ref{inessential non-tight} implies that each inner boundary component of $\phi_j$ is $\widehat F_j$-essential. There must be such boundary components since $X^+$ is not a product. Thus $\widehat \phi_j$ is an $\widehat F_j$-essential annulus.

Let $P_1,...,P_k$ be the components of $\dot{\Sigma}_1^+$ and set
$\phi_{1i} = P_i\cap F_1$ and $\phi_{2i}=P_i\cap F_2$.  Then each $\widehat{\phi}_{j,i}$ is an
$\widehat F_j$-essential annulus. The
closure of the complement of $\cup_i \widehat \phi_{ji}$ in $\widehat F_j$ is
a set of annuli which we denote by $E_{ji}, i=1,...,k$. We may
assume that $\widehat \phi_{j1}, E_{j1}, \widehat \phi_{j2}, E_{j2},...,
\widehat \phi_{jk}, E_{jk}$ appear consecutively in $\widehat F_j$.

Let $d_i = |\phi_{1i} \cap \partial M| = |\phi_{2i} \cap \partial M|$. Since $\dot{\Phi}_1^+$ has no tight components, $d_1+...+ d_k = m$. We will assume that $k > 1$ in order to derive a contradiction.
Then without loss of generality, $2 d_1 \leq m$.

For each $i=1,...,k$, let $A_i, A_i'$ be the two components of the
frontier of $P_i$ in $X^+$. Then each of $A_i$ and $A_i'$ is an
essential annulus in $(X^+, S)$. We may assume that $\partial A_i' \cup \partial A_{i+1} =
\partial E_{1i} \cup \partial E_{2i}$, so $A_1,
A_1',A_2,A_2',..., A_k, A_k'$ appear consecutively in $X^+$.

Now $A_i'\cup E_{1i}\cup A_{i+1}\cup E_{2i}$ is a torus in $X^+$
which contains a curve which is null-homotopic in $M$. (Here the indices are defined (mod $k$).)
It therefore bounds a solid torus $V_i$ in $X^+$. Note that $A_i'$ is not parallel in $V$ to
$A_{i+1}$ as otherwise $P_i$ and $P_{i+1}$ would be contained in a component of $\dot{\Sigma}_1^+$.
Then $U_i = V_{i-1} \cup \widehat P_i \cup V_i$ is a
submanifold of $M(\beta)$ which is a Seifert fibred space over the disk
with two cone points. Since $\widehat S$ can be isotoped into $\overline{M(\beta) \setminus U_i}$,
Lemma \ref{incomp} implies that $\partial U_i$ is an incompressible torus in $M(\beta)$. By construction,
$\partial U_1$ contains $2d_1\leq m$ components of $\partial M$. Assumption \ref{assumption minimal} then implies that $2 d_1 = m$. But this is impossible by Assumption \ref{assumption sep}. Thus $k = 1$.

Finally note that the closure of the complement of $\dot{\Sigma}_1^+$ in $X^+$ is a solid torus $V$ such that $\dot{\Sigma}_1^+\cap V = A_1 \cup A_1'$, in $X^+$. Hence $\widehat X^+$ is
homeomorphic to the manifold obtained from $V$ by identifying $A_1$
with $A_1'$. It is therefore a Seifert fibred space over the annulus with
at most one cone point. If there is no cone point, then $F$ is a fibre in $M$, contrary to Assumption \ref{assumption not (semi) fibre}. This completes the proof.
\qed

\begin{cor} \label{not fibred}
If $F$ is non-separating and $t_1^+ = 0$, then $M(\beta)$ does not fibre over the circle with torus fibre.
\end{cor}

\pf Suppose otherwise and let $T$ be the fibre. Isotope $\widehat F$ so that it intersects $T$ transversally and in a minimal number of components. Since $T$ is a fibre, 	the previous lemma shows that $T \cap \widehat F \ne \emptyset$ and so $T$ cuts $\widehat F$ into a finite collection of incompressible annuli which run from one side of $T$ to the other. It follows that $M(\beta)$ admits a Seifert structure in which $T$ is horizontal. If $\widehat F$ is horizontal it is a fibre in  $M(\beta)$, which contradicts Assumption \ref{assumption not (semi) fibre}. Thus it is vertical. It follows that the base orbifold $\mathcal{B}$ of $M(\beta)$ is Euclidean. Further, the projection image of $\widehat F$ in $\mathcal{B}$ is a non-separating two-sided curve. Thus $\mathcal{B}$ is either a torus or Klein bottle. In either case $\widehat F$ splits $M(\beta)$ into the product of a torus and an interval, which is impossible by Lemma \ref{one-comp}. This completes the proof.
\qed

\section{$\widehat S$-essential annuli in $\dot \Phi_j^\epsilon$} \label{f-hat essential annuli}

\begin{prop} \label{sep annulus-in-2}
Suppose that $F$ is separating. If $\dot \Phi_2^+$ or $\dot \Phi_2^-$ contains an $\widehat F$-essential annulus, then $\widehat X^\epsilon$ admits a Seifert structure with base orbifold of the form $D^2(a,b)$ for some $a, b \geq 2$ for both $\epsilon$. Further, one of the following situations arises:

$(i)$ $t_1^+ + t_1^- \geq 4$.

$(ii)$ $X^-$ is a twisted $I$-bundle.

$(iii)$ $M(\beta)$ admits a Seifert structure with base orbifold $S^2(a, b, c, d)$. Further, if $t_1^\epsilon = 0$ for some $\epsilon$, then $(a,b,c,d) \ne (2,2,2,2)$.

\end{prop}

\pf As $h_2^+: \dot \Phi_2^+ \stackrel{\cong}{\longrightarrow} \dot \Phi_2^-$, we can suppose that $\dot \Phi_2^-$ contains an $\widehat F$-essential annulus. Since $\dot \Phi_2^- = \tau_-(\dot \Phi_1^- \wedge \dot \Phi_1^+)$, if $\dot \Phi_2^-$ contains an $\widehat F$-essential annulus, so do $\dot \Phi_1^+$ and $\dot \Phi_1^-$. Hence Proposition \ref{sep-seifert} implies that $\widehat X^\epsilon$ admits a Seifert structure with base orbifold of the form $D^2(a,b)$ for some $a, b \geq 2$ for both $\epsilon$.

If $\hbox{genus}(\dot \Phi_1^\epsilon) = 1$ for some $\epsilon$, then either $X^-$ is a twisted $I$-bundle or $t_1^\epsilon \geq 4$ (Corollary \ref{inessential non-tight}). Thus (i) or (ii) holds. Assume that $\hbox{genus}(\dot \Phi_1^\epsilon) = 0$ for both $\epsilon$, so $X^-$ is not a twisted $I$-bundle, and let $\varphi_{\epsilon}$ be the slope on $\widehat F$ of an $\widehat F$-essential annulus contained in $\dot \Phi_1^\epsilon$. Then $\varphi_\epsilon$ is the fibre slope of the Seifert structure on $\widehat X^\epsilon$ given by Proposition \ref{sep-seifert}. As $\dot \Phi_2^- = \tau_-(\dot \Phi_1^- \wedge \dot \Phi_1^+)$, we see that $\dot \Phi_1^-$ contains curves of slope $\varphi_+$ and $\varphi_-$. Hence if these slopes are distinct, $\hbox{genus}(\dot \Phi_1^-) = 1$, contrary to our assumptions. Thus $\varphi_+ = \varphi_-$ so $M(\beta)$ admits a Seifert structure with base orbifold of the form $S^2(a, b, c, d)$. Finally if $t_1^\epsilon = 0$ for some $\epsilon$, Proposition \ref{order two} shows that $(a,b,c,d) \ne (2,2,2,2)$.
\qed

\begin{prop} \label{sep annulus-in-3-+}
Suppose that $F$ is separating. If $\dot \Phi_3^+$ contains an $\widehat F$-essential annulus then either

$(i)$ $t_1^+ \geq 4$, or

$(ii)$ $X^-$ is a twisted $I$-bundle and $M(\beta)$ is Seifert with base orbifold $P^2(2,n)$ for some $n > 2$. Further, $t_1^+ = 0$, $\dot \Phi_1^+$ is an $\widehat F$-essential annulus, $\dot \Phi_3^+$ is the union of two $\widehat F$-essential annuli, and there are disjoint, non-separating annuli $A_1^-, A_2^-$ properly embedded in $X^-$ such that $\partial A_1^- \cup \partial A_2^- \subseteq \dot \Phi_1^+$ and for each $j$, $\partial \dot \Phi_1^+ \cap \partial A_j^-$ is a boundary component of $\dot \Phi_1^+$.
\end{prop}

\pf Assume that $t_1^+ \leq 2$. We will show that (ii) holds.

Suppose that some component $\phi_0$ of $\dot \Phi_3^+$ contains an $\widehat F$-essential annulus and let $\psi_0$ be the component of $\dot \Phi_1^+$ containing $\phi_0$. By Assumption \ref{assumption not (semi) fibre}, $\psi_0 \ne F$. Corollary \ref{inessential non-tight} then shows that $\hbox{genus}(\psi_0) = 0$ and $\widehat \psi_0$ completes to an $\widehat F$-essential annulus.

We can suppose that $\phi_0 \subseteq \hbox{int}(\psi_0)$. Set $\phi_1 = \tau_+(\phi_0) \subseteq \dot \Phi_1^+ \wedge \dot \Phi_2^-$. Now $h_3^+ = \tau_+ \circ \tau_- \circ \tau_+| \dot \Phi_3^+$ is a free involution of $\dot \Phi_3^+$. In particular, either $h_3^+(\phi_0) = \phi_0$ or $h_3^+(\phi_0) \cap \phi_0 = \emptyset$. Equivalently, either $\tau_-(\phi_1) = \phi_1$ or $\tau_-(\phi_1) \cap \phi_1 = \emptyset$. In the first case there are an essential annulus $A^-$ properly embedded in $(X^-, \phi_1)$ such that $\partial A^- = \partial \widehat \phi_1$ and a M\"{o}bius band $B$ properly embedded in $(X^-, \hbox{int}(\phi_1))$. Proposition \ref{na3} then implies that $\psi_0$ is $\tau_+$-invariant. Hence there is an annulus $A^+$ properly embedded in $(X^+, \psi_0)$ with $\partial A^+ = \partial \widehat \psi_0$. Lemma \ref{nested-annuli} implies that $\phi_1$, and therefore $\phi_0$, has no outer boundary components, which is impossible.

Next suppose that $\tau_-(\phi_1) \cap \phi_1 = \emptyset$. Then there is an embedding $(\phi_1 \times I, \phi_1 \times \{0\}, \phi_1 \times \{1\}) \to (X^-, \phi_1, \tau_-(\phi_1))$. First suppose that the components $A_1^-, A_2^-$ of the image of $\partial \widehat \phi_1 \times I$ are separating annuli in $X^-$. Let $A^+$ be a properly embedded annulus in $(X^+, \psi_0)$ such that at least one boundary component of $A^+$ is contained in $\partial \widehat \psi_0$. According to Lemma \ref{nested-annuli}(1), $\partial A^+$ does not separate $\partial A_j^-$ for $j = 1, 2$. Lemma \ref{nested-annuli}(2) then implies that $\phi_1$ has no outer boundary components, which is impossible. Thus $A_1^-$ and $A_2^-$ are non-separating in $X^-$. In particular, $X^-$ is a twisted $I$-bundle  (Lemma \ref{non-sep}).

Let $P$ be the unique component of $\dot \Sigma_1^+$ whose intersection with $F$ contains an $\widehat F$-essential annulus (Proposition \ref{sep-seifert}). Then $\psi_0$ is a component of $P \cap F$ and $\phi_0 \cup \phi_1 \subseteq P \cap F$. If $P$ is a product $I$-bundle, let $A_1^+, A_2^+$ be the annuli in its frontier in $X^+$, and consider the torus $T$ obtained from the union of $A_1^+, A_1^-, A_2^+, A_2^-$ and four annuli in $\widehat F$ disjoint from $\hbox{int}(\phi_1) \cup \hbox{int}(\tau_-(\phi_1))$. The reader will verify that $T$ bounds a twisted $I$-bundle over the Klein bottle in $M(\beta)$ and so is essential in $M(\beta)$ by Lemma \ref{incomp}. Hence it intersects $\partial M$ in at least $m$ components. But this implies $|\phi_1 \cap \partial F| = 0$, which is impossible. Thus $P$ is a twisted $I$-bundle. It follows from Proposition \ref{sep-seifert} that $\widehat X^+$ is Seifert with base orbifold $D^2(2,n)$ with $\partial A_1^-$ vertical. Thus $M(\beta)$ is Seifert with base orbifold $P^2(2,n)$. Let $A^+$ be the frontier of $P$ in $X^+$. By construction, $\partial A^+$ does not separate $\partial A_1^-$ or $\partial A_2^-$ in $\widehat F$. Lemma \ref{sep-non-sep nested-annuli} then shows that $|(\widehat F \setminus \psi_0) \cap \partial F| = 0$. Hence, $t_1^+ = 0$. This implies that $n > 2$ (Corollary \ref{both not (2,2)}) and $\tau_-$ is defined on $\overline{\widehat F \setminus \psi_0}$ and sends it into the interior of $\dot \Phi_1^+$. Thus, there are disjoint, non-separating annuli $E_1^-, E_2^-$ properly embedded in $X^-$ such that $\partial E_1^- \cup \partial E_2^- \subseteq \dot \Phi_1^+$ and for each $j$, $\partial \dot \Phi_1^+ \cap \partial E_j^-$ is a boundary component of $\dot \Phi_1^+$. Write $\partial E_j^- = c_j \cup c_j'$ where $\partial \dot \Phi_1^+ = c_1 \cup c_2$. Since $\overline{F \setminus \psi_0}$ is an annulus, it follows from our constructions that the disjoint subsurfaces of $\dot \Phi_1^+$ with inner boundaries $c_1 \cup c_2'$ and $c_2 \cup c_1'$ lie in $\dot \Phi_3^+$ and contain $\partial F$. Thus their union is $\dot \Phi_3^+$. This proves the proposition.
\qed

\begin{defin} {\rm (\cite[page 266]{BGZ1})} \label{singular slope}
{\rm  Given a closed, essential surface $G$ in $M,$ we let ${\mathcal
    C}(G)$ denote the set of slopes $\delta$ on $\partial M$ such that $S$
  compresses in $M(\delta).$  A slope $\eta$ on $\partial M$ is called a
  {\it singular slope} for $G$ if $\eta \in {\mathcal C}(G)$ and
  $\Delta(\delta, \eta) \leq 1$ for each $\delta \in {\mathcal C}(G).$}
\end{defin}
A fundamental result of Wu \cite{Wu1} states that if ${\mathcal C}(G) \ne
\emptyset,$ then there is at least one singular slope for $G.$

\begin{prop}\label{bgz1} Let $\eta$ and $\delta$ be slopes on the boundary of a hyperbolic knot manifold $M$.

$(1)$ {\rm (\cite[Theorem 1.5]{BGZ1})} If $\eta$ is a singular slope for some closed essential surface in $M$ and $M(\delta)$ is not hyperbolic, then
$\Delta(\delta, \eta) \leq 3.$

$(2)$ {\rm (\cite[Theorem 1.7]{BGZ1})} If $M(\eta)$ is a Seifert fibred manifold whose base orbifold is hyperbolic but a $2$-sphere with three cone points, then
$\eta$ is a singular slope for some closed essential surface in $M$.
\qed
\end{prop}

\begin{prop} \label{nonsep annulus-in-3}
Suppose that $F$ is non-separating. If $\dot \Phi_3^+$ contains an $\widehat S$-essential annulus and $t_1^+ = 0$, then $M(\beta)$ is Seifert fibred with base orbifold a torus or a Klein bottle with exactly one cone point. In particular, $\beta$ is a singular slope for a closed essential surface in $M$ and thus, $\Delta(\alpha, \beta) \leq 3$.
\end{prop}

\pf Since $t_1^+ = 0$, Proposition \ref{nonsep-seifert} implies that $\dot \Phi_1^+ = \phi_1 \cup \phi_2$ where $\phi_1, \phi_2$ lie in different components of $S$ and complete to $\widehat S$-essential annuli. This proposition also implies that $\widehat X^+$ admits a Seifert fibred structure with base orbifold an annulus with one cone point. Further, $\widehat \phi_j$ is vertical in this structure for both $j$. To see that $M(\beta)$ is Seifert with base orbifold as claimed, it suffices to show that the slope of $\widehat{\tau_-(\phi_j)}$ coincides with that of $\widehat \phi_{3-j}$. But this is an immediate consequence of the fact that $\tau_+(\dot \Phi_3^+) = \dot \Phi_1^+ \wedge \dot \Phi_2^- = \dot \Phi_1^+ \wedge \tau_-(\dot \Phi_1^+)$ contains an $\widehat S$-essential annulus.
\qed

\section{The existence of tight components in $\breve{\Phi}_j^\epsilon$ for small values of $j$} \label{existence tights}

In this section we examine the existence of tight components in $\breve{\Phi}_j^\epsilon$ for small values of $j$. Note that if $t_1^\epsilon \ne 0$ for some $\epsilon$, then Proposition \ref{tight increase} implies that $t_2^{-\epsilon} = t_2^\epsilon \geq t_1^\epsilon > 0$. Thus we examine the case $t_1^+ = t_1^- = 0$. Recall that under this hypothesis, $\dot{\Phi}_1^+$ and $\dot{\Phi}_1^-$ are non-empty (Lemma \ref{t1e=0 implies dot nonempty}).

\begin{lemma} \label{seifert means twisted}
Suppose that $t_1^+ = t_1^- = 0$. If $\Delta(\alpha, \beta) > 3$ and $M(\beta)$ is Seifert fibred, then its base orbifold is of the form $P^2(a,b)$ for some $(a, b) \ne (2,2)$ and $X^-$ is a twisted $I$-bundle.
\end{lemma}

\pf Since $\Delta(\alpha, \beta) > 3$, $\beta$ is not a singular slope of a closed essential surface in $M$ (\cite[Theorem 1.5]{BGZ1}). Hence, as $M(\beta)$ is toroidal, Seifert but not the union of two twisted $I$-bundles over the Klein bottle (Corollary \ref{both not (2,2)}), its base orbifold is of the form $P^2(a,b)$ (\cite[Theorem 1.7]{BGZ1}) where $(a,b) \ne (2,2)$. Each essential torus in $M(\beta)$ splits it into the union of a twisted $I$-bundle over the Klein bottle and a Seifert manifold with base orbifold $D^2(a,b)$. Since $t_1^+ = t_1^- = 0$, Proposition \ref{order two}(3) implies that $X^-$ is a twisted $I$-bundle.
\qed

\begin{lemma} \label{tight or essential annulus}
Let $S_1, S_2$ be large, neat, connected surfaces contained in the same component of $S$. Suppose, for each $j$, that either $S_j$ is tight or $\widehat S_j$ is an $\widehat S$-essential annulus.

$(1)$ Each component of $S_1 \wedge S_2$ is either tight or an $\widehat S$-essential annulus.

$(2)$ If we further assume that when both $\widehat S_1$ and $\widehat S_2$ are $\widehat S$-essential annuli, their slopes are distinct, then each component of $S_1 \wedge S_2$ is tight.

\end{lemma}

\pf Let $S_0$ be a component of $S_1 \wedge S_2$.

First suppose that $S_0$ is contained in a disk $D$ in $\widehat S$. If $S_0$ is not tight, it has at least two inner boundary components. Let $C$ be an inner boundary component of $S_0$ which is innermost in $D$ amongst all the other inner boundary components of $S_0$. Let $D_0 \subseteq D \subseteq \widehat S$ be the disk with boundary $C$. By construction, $D_0 \cap S_0 = C$. Further, the neatness of $S_1$ and $S_2$ implies that $D_0 \cap S$ is large. Since $\widehat S_j$ is either a disk or an $\widehat S$-essential annulus, $D_0 \cap S \subseteq S_j$ for each $j$. Hence it is contained in $S_1 \wedge S_2$ and therefore $S_0$, contrary to our construction. Thus $S_0$ must be tight. In particular, this proves (2).

Next suppose that $S_0$ contains an $\widehat S$-essential annulus but $S_0$ is not itself an $\widehat S$-essential annulus. Then $S_0$ has at least three inner boundary components and all but exactly two of them are inessential in $\widehat S$. Fix an inessential inner boundary component $C$ of $S_0$. The argument of the previous paragraph is easily adapted to this case and leads to a contradiction. Thus $\widehat S_0$ must be an $\widehat S$-essential annulus.
\qed

\begin{prop} \label{phi3 options}
If $t_1^+ = t_1^- = 0$, then one of the following three scenarios arises.

$(i)$ $\breve \Phi_3^+$ is a union of tight components.

$(ii)$ $t_3^+ = 0$, $X^-$ is a twisted $I$-bundle, and $M(\beta)$ is Seifert with base orbifold $P^2(2,n)$ for some $n >2$. Further, $\dot \Phi_1^+$ completes to an $\widehat F$-essential annulus, $\dot \Phi_3^+$ completes to the union of two $\widehat F$-essential annuli, and there are disjoint, non-separating annuli $A_1^-, A_2^-$ properly embedded in $(X^-, F)$ such that $\partial A_1^- \cup \partial A_2^- \subseteq \dot \Phi_1^+$ and for each $j$, $\partial \dot \Phi_1^+ \cap \partial A_j^-$ is a boundary component of $\dot \Phi_1^+$.

$(iii)$ $X^-$ is a product $I$-bundle and $M(\beta)$ is Seifert fibred with base orbifold a torus or a Klein bottle with exactly one cone point. In particular, $\beta$ is a singular slope for a closed essential surface in $M$ and thus, $\Delta(\alpha, \beta) \leq 3$.

\end{prop}

\pf Propositions \ref{sep-seifert} and \ref{nonsep-seifert} imply that $\breve \Phi_1^\epsilon = \dot \Phi_1^\epsilon$ is either $S$ or a union of subsurfaces whose completions are $\widehat S$-essential annuli. If no component of $\breve \Phi_3^+$ contains an $\widehat S$-essential annulus, Lemma \ref{tight or essential annulus} shows that (i) holds. If some component of $\breve \Phi_3^+$ does contain an $\widehat S$-essential annulus, Propositions \ref{sep annulus-in-3-+}  and \ref{nonsep annulus-in-3} show that (ii) and (iii) hold.
\qed

\begin{prop} \label{tightness for large j}
Suppose that $t_1^+ = t_1^- = 0$ and $\Delta(\alpha, \beta) > 3$.

$(1)$ If $X^-$ is not an $I$-bundle, each component of $\breve \Phi_j^\epsilon$ is tight for all $j \geq 2$ and both $\epsilon$.

$(2)$ If $X^-$ is a product $I$-bundle, or $X^-$ is a twisted $I$-bundle and $\breve \Phi_3^+$ does not contain an $\widehat S$-essential annulus, each component of $\breve \Phi_j^+$ is tight for all $j \geq 3$.

$(3)$ If $X^-$ is a twisted $I$-bundle and $\breve \Phi_3^+$ contains an $\widehat S$-essential annulus, then $t_3^+ = 0$, $M(\beta)$ is Seifert with base orbifold $P^2(2,n)$ for some $n > 2$, $\dot \Phi_1^+$ and $\dot \Phi_3^+$ are as described in Proposition \ref{phi3 options}(ii), and each component of $\breve \Phi_j^+$ is tight for all $j \geq 5$.

\end{prop}

\pf Propositions \ref{sep-seifert} and \ref{nonsep-seifert} show that for each $\epsilon$, $\breve \Phi_1^\epsilon = \dot \Phi_1^\epsilon$ is either $S$ or a union of subsurfaces whose completions are $\widehat S$-essential annuli.

First note that in order to prove assertion (1), it suffices to show that each component of $\dot \Phi_2^+$ is tight. For if this holds, the same is true of $\dot \Phi_2^- = h_2^+(\dot \Phi_2^+)$. Suppose inductively that each component of $\dot \Phi_j^\epsilon$ is tight for some $j \geq 2$ and both $\epsilon$. Lemma \ref{tight or essential annulus}  combines with the identity $\dot \Phi_{j+1}^{\epsilon} = \tau_\epsilon(\dot \Phi_1^\epsilon \wedge \dot \Phi_j^{-\epsilon})$ to show that each component of $\dot \Phi_{j+1}^\epsilon$ is tight.

Consider $\dot \Phi_2^+$ then. Since $t_1^+ = t_1^- = 0$ and $X^-$ is not an $I$-bundle, $S = F$ is separating. Proposition \ref{sep-seifert} implies that for each $\epsilon$ and component $\phi$ of $\dot \Phi_1^\epsilon$, $\widehat \phi$ is an $\widehat F$-essential annulus. Lemma \ref{seifert means twisted} implies that $M(\beta)$ is not Seifert, so the slopes of an $\widehat F$-essential annulus in $\widehat{\dot \Phi_1^+}$ and an $\widehat F$-essential annulus in $\widehat{\dot \Phi_1^-}$ are distinct. Hence Lemma \ref{tight or essential annulus} implies that each component of $\dot \Phi_2^+ = \tau_+(\dot \Phi_1^+ \wedge \dot \Phi_1^-)$ is tight.

Next consider the hypotheses of assertion (2). Proposition \ref {phi3 options} implies that each component of $\dot \Phi_3^+$ is tight. Since $X^-$ is an $I$-bundle, $\dot \Phi_3^+ = \dot \Phi_4^+$, so the lemma holds when $j = 3, 4$. But when $j \geq 5$, we have $\dot \Phi_{j+1}^+ = \tau_\epsilon(\dot \Phi_1^\epsilon \wedge \tau_-(\dot \Phi_{j-1}^+)$, so Lemma \ref{tight or essential annulus} combines with an inductive argument to show that (2) holds for all $j \geq 3$.

Finally consider assertion (3). Now $\tau_- \tau_+(\breve \Phi_5^+) = \breve \Phi_3^+ \wedge \breve \Phi_2^- = \breve \Phi_3^+ \wedge \tau_-(\breve \Phi_1^+) = \breve \Phi_3^+ \wedge \tau_+(\breve \Phi_3^+)$ where the latter identity follows from Proposition \ref{phi3 options}(ii). Now $\tau_+(\breve \Phi_3^+) = \phi_1 \cup \phi_2$ where $\widehat \phi_j$ is an $\widehat F$-essential annulus and $\phi_2 = \tau_- (\phi_1)$. Hence
$$\breve \Phi_3^+ \wedge \tau_+(\breve \Phi_3^+) = \phi_1 \wedge \tau_+(\phi_1) \sqcup \phi_1 \wedge \tau_+(\phi_2) \sqcup \phi_2 \wedge \tau_+(\phi_1) \sqcup \phi_2 \wedge \tau_+(\phi_2)$$
By construction, $\phi_j \cap \tau_+(\phi_{3-j})$ contains an inner boundary component of $\dot \Phi_1^+$ for both $j$. If $\phi_j \wedge \tau_+(\phi_{3-j}) = \phi_j$ for some $j$, then $\phi_j \wedge \tau_+(\phi_{3-j}) = \phi_j$ and $\phi_j \wedge \tau_+(\phi_{j}) = \emptyset$ for both $j$. Hence $\tau_- \tau_+(\breve \Phi_5^+) = \breve \Phi_3^+$ from which it follows that $\breve\Phi_3^+ = \breve \Phi_5^+$, which is impossible. Hence $\phi_j \wedge \tau_+(\phi_{3-j}) \ne \phi_j$ for both $j$. It follows that $\phi_j \wedge \tau_+(\phi_{j}) \ne \emptyset$ is a non-empty union of tight components for each $j$.

Next consider $\phi_j \wedge \tau_+(\phi_{3-j})$. By Lemma \ref{tight or essential annulus}, each of its components is either tight or completes to an $\widehat S$-essential annulus. Since $\tau_+(\phi_{3-j})$ contains the inner component $c_j$ of $\dot{\Phi}_1^+$ contained in $\phi_j$, there is exactly one component of $\phi_j \wedge \tau_+(\phi_{3-j})$, $E_j$ say,  which completes to an $\widehat S$-essential annulus. To complete the proof we need only show that $E_j \cap \partial S = \emptyset$.

By construction, $\tau_+(E_1)=E_2$ and $\tau_+(E_2)=E_1$. Let $E_0 = \overline{\dot \Phi_1^+ \setminus (E_1\cup E_2)}$. Then $E_0$ completes to an $\widehat F$-essential annulus $\widehat E_0$ which is invariant under $\tau_+$. Hence the associated $I$-bundle over $\widehat E_0$ is a solid torus $V_1$ whose frontier in $\widehat X^+$ is an essential annulus in $X^+$ which has winding number $2$ in $V_1$.

Next consider the solid torus $V_2= \overline{X^+\setminus \dot{\Sigma}_1^+}$. The frontier of $V_2$ in $X^+$ is an essential annulus in $X^+$ cobounded by $c_1$ and $c_2$. Let $A$ be the other annulus in $\partial V_2$ cobounded by $c_1$ and $c_2$. Then $\tau_-(A) = \overline{\dot{\Phi}_1^+ \setminus (\phi_1 \cup \phi_2)}$ is a core annulus in $\widehat E_0$. The $I$-bundle in $X^-$ over $A$ is a solid torus in which $A$ has winding number $1$. It follows that $W = V_1 \cup V_2 \cup V_3 \subset M(\beta)$ is a Seifert fibered space
over a disk with two cone points. In particular, $\partial W$ is incompressible
in $W$. The exterior of $W$ in $M(\beta)$ is a Seifert fibered space over a M\"{o}bius band so $\partial W$ is also incompressible in $M(\beta)$. If $E_j \cap \partial S \ne  \emptyset$ for some $j$ then $\partial W$ intersects $\partial M$ in fewer than $m$ components contrary to Assumption \ref{assumption minimal}. Thus we must have $E_j \cap \partial S = \emptyset$ for both $j$, which completes that proof of (3).
\qed

\section{Lengths of essential homotopies} \label{length}

It is clear that $\chi(\breve{\Phi}_j^\epsilon) = 0$ if and only if $\chi(\breve{\Phi}_j^\epsilon)$ is a regular neighbourhood of $\partial S$. Thus if we set
$$l_\epsilon = \hbox{max}\{ j : \chi(\breve{\Phi}_j^\epsilon) \ne 0\}$$
then $l_\epsilon$ is the maximal length of an essential homotopy in $(M, S)$ of a large function which begins on the $\epsilon$-side of $S$. Hence
$$l_S = \hbox{max}\{ l_+, l_- \}$$
is the maximal length of an essential homotopy in $(M, S)$ of a large function. It is evident that $|l_+ - l_-| \leq 1$ and therefore $|l_\epsilon - l_S| \leq 1$ for each $\epsilon$.

\begin{prop} \label{tight bound}
Suppose that $\Delta(\alpha, \beta) > 3$ if $F$ is non-separating.

$(1)$ $l_+ \leq |\partial S| - t_1^+$ if $t_3^+ > 0$ and $l_+ \leq |\partial S| - t_1^+ + 2$ otherwise. Hence,
$$l_S \leq \left\{ \begin{array}{ll} |\partial S| - t_1^+ + 1 & \mbox{if } t_3^+ > 0 \\  |\partial S| - t_1^+ + 3 & \mbox{otherwise} \end{array} \right. $$

$(2)$ If $X^-$ is not an $I$-bundle, then $l_- \leq |\partial S| - t_1^-$. Hence,
$$l_S \leq |\partial S| - \left\{\begin{array}{ll}  \hbox{max}\{t_1^+, t_1^-\} & \hbox{if } t_1^+ = t_1^- \\
& \\ \hbox{max}\{t_1^+, t_1^-\} - 1&  \hbox{if } t_1^+ \ne t_1^-
\end{array} \right. $$
\end{prop}

\pf For each $\epsilon$ we know that $\mathcal{T}_{l_\epsilon + 1}^\epsilon$ is a regular neighbourhood of $\partial S$ while $\mathcal{T}_{l_\epsilon}^\epsilon$ isn't, so $t_{l_\epsilon} < t_{l_\epsilon + 1} = |\partial S|$. If $t_3^+ > 0$, Proposition \ref{tight increase} implies that if $2k+1 \leq l_+$, then $t_1^+ < t_3^+ < \ldots < t_{2k+1}^+ < |\partial S|$. As each of these numbers is even, $|\partial S| > t_{2k+1}^+ \geq 2k + t_1^+$. Hence $2k + 2 \leq |\partial S| - t_1^+$. It follows that $l_+  \leq |\partial S| - t_1^+$ and therefore $l_S  \leq |\partial S| - t_1^+ + 1$. In general, Proposition \ref{tightness for large j} implies that $t_5^+ > 0$, which yields  $l_+ \leq |\partial S| - t_1^+ + 2$. Thus assertion (1) holds.

If $X^-$ is not an $I$-bundle, then Proposition \ref{tightness for large j} implies that $t_3^- > 0$, so the argument of the previous paragraph shows that $l_-  \leq |\partial S| - t_1^-$ and therefore

\indent \hspace{5mm} $\bullet$ $l_S  \leq |\partial S| - t_1^- + 1$; and

\indent \hspace{5mm} $\bullet$ $l_S  = \hbox{max}\{l_+, l_-\} \leq \hbox{max}\{|\partial S| - t_1^+, |\partial S| - t_1^-\} = |\partial S| - \hbox{min}\{t_1^+, t_1^-\}$.

Part (1) of the proposition combines with the first inequality to show that $l_S  \leq \hbox{min}\{|\partial S| - t_1^+ + 1, |\partial S| - t_1^- + 1\} = |\partial S| - \hbox{max}\{t_1^+, t_1^-\} + 1$. The latter combines with the second inequality to yield the upper bound for $l_S$ described in (2).
\qed

Proposition \ref {tightness for large j} and Proposition \ref{tight bound} imply the following corollary.

\begin{cor} \label{general length}
Suppose that $\Delta(\alpha, \beta) > 3$ if $F$ is non-separating.  Then
$$l_S \leq \left\{ \begin{array}{ll}
|\partial S|  & \mbox{if $X^-$ is not an $I$-bundle } \\
|\partial S| + 1 & \mbox{if $X^-$ is an $I$-bundle and $\dot \Phi_3^+$ contains no $\widehat S$-essential annulus  when it is twisted} \\
|\partial S| + 3 & \mbox{if $X^-$ is a twisted $I$-bundle and $\dot \Phi_3^+$ contains an $\widehat S$-essential annulus  \qed} \end{array}  \right. $$

\end{cor}

\section{The intersection graph of an immersed disk or torus} \label{graph}

We recall some of the set up from \cite[\S 12]{BCSZ2}.

A $3$-manifold is {\it very small} if its fundamental group does not contain a non-abelian free group.

By Assumption \ref{assumptions on alpha and beta}, $M(\alpha)$ is a small Seifert manifold with base orbifold $S^2(a,b,c)$ where $a, b, c \geq 1$. It is well-known that $M(\alpha)$ is very small if and only if $\frac{1}{a} + \frac1b + \frac1c \geq 1$. In this case, the non-abelian free group $\pi_1(F)$ cannot inject into $\pi_1(M(\alpha))$. Hence for either $\epsilon$ we can find maps $h: D^2 \to M(\alpha)$ such that the loop $h(\partial D^2)$ is contained in $X^\epsilon \setminus F$ and represents a non-trivial element of $\pi_1(X^\epsilon )$. This won't necessarily be possible when $\frac{1}{a} + \frac1b + \frac1c < 1$ since $M(\alpha)$ is not very small. Nevertheless, the inverse image in $M(\alpha)$ of an essential immersed loop contained in the exterior of the cone points of $S^2(a,b,c)$ will be an essential immersed torus in $M(\alpha)$. Hence we can find $\pi_1$-injective immersions $h: T \to  M(\alpha)$ where $T$ is a torus.

Let $V_\alpha$ be the filling solid torus used in forming $M(\alpha)$. It is shown in \cite[\S 12]{BCSZ2} that we can choose an immersion $h: Y \ra M(\alpha)$, where $Y$ is a disk $D$ if $M(\alpha)$ is very small or a torus $T$ if $M(\alpha)$ otherwise, such that
\begin{enumerate}

\vspace{-.3cm} \item When $Y$ is a disk $D$, $h(\partial D)\subseteq M \setminus F\subseteq
M\subseteq M(\alpha)$;

\vspace{.3cm} \item $h^{-1}(V_\alpha)$ is a non-empty set of embedded disks in
 the interior of $Y$ and $h$ is an embedding when restricted on
 $h^{-1}(V_\alpha)$;

\vspace{.3cm} \item $h^{-1}(F)$ is a set of arcs or circles properly embedded
  in the punctured surface $Y_0 = Y \setminus int(h^{-1}(V_\alpha))$;

\vspace{.3cm} \item If $e$ is an arc component of $h^{-1}(F)$, then $h|:e\ra F$
  is an essential (immersed) arc;

\vspace{.3cm} \item If $c$ is a circle component of $h^{-1}(F)$, then $h|:c\ra F$ is an essential (immersed)
  $1$-sphere.

\end{enumerate}
\vspace{-.3cm}
For any subset $s$ of $Y$, we use $s^*$ to denote its image under
the map $h$.  Denote the components of $\p(h^{-1}( V_\alpha))$ by
$a_1, ..., a_n$ so that $a_1^*, ..., a_n^*$ appear consecutively on
$\partial M$. Note again that $h|: a_i\ra a_i^*\subseteq \partial M$ is an
embedding and that $a_i^*$ has slope $\alpha$ in $\partial M$, for each
$i=1,...,n$. We fix an orientation on
$Y_0$ and let each component $a_i$ of $\partial Y_0$ have the induced
orientation. Two components $a_i$ and $a_j$ are said to have the
same orientation if $a_i^*$ and $a_j^*$ are homologous in $\partial M$.
Otherwise, they are said to have different orientations.

Denote the components of $\partial F$ by $b_1, ..., b_{m}$ so
that they appear consecutively in $\partial M$. Similar
definitions apply to the components of $\partial F$. Since $Y_0$, $F$
and $M$ are all orientable, one has the following

\noindent {\bf Parity rule}: An arc component $e$ of $h^{-1}(F)$
in $Y_0$ connects components of $\partial Y_0$ with the same orientation
(respectively opposite orientations) if and only if the corresponding
$e^*$ in $F$ connects components of $\partial F$ with opposite
orientations (respectively the same orientation).

We define an {\it intersection graph} $\Gamma_F$ on the surface $Y$ by taking
$h^{-1}(V_\alpha)$ as (fat) vertices and taking arc components of
$h^{-1}(F)$ as edges. Note that $\Gamma_F$ has no trivial loops, i.e.  no
$1$-edge disk faces. Also note that we can assume that each $a_i^*$ intersects each
component $b_j$ in $\partial M$ in exactly $\D(\alpha,\beta)$ points. If
$e$ is an edge in $\Gamma_F$ with an endpoint at the vertex $a_i$, then
the corresponding endpoint of $e^*$ is in $a_i^*\cap b_j$ for some
$b_j$, and the endpoint of $e$ is thus given the label $j$. So when
we travel around $a_i$ in some direction, we see the labels of the
endpoints of edges appearing in the order $1,...,m,..., 1,...,m$
(repeated $\D(\alpha,\beta)$ times). It also follows that each
vertex of $\Gamma_F$ has valency $m\D(\alpha,\beta)$.

Define the {\it double of $\Gamma_F$} to be the graph $D(\Gamma_F)$
in $Y$ as follows: the vertices of $D(\Gamma_F)$ are the vertices of $\Gamma_F$; the edges of $D(\Gamma_F)$ are obtained by doubling the edges of $\Gamma_F$ (i.e. each edge $e$ is replaced by two parallel copies of $e$). Finally we set
$$\Gamma_S = \left\{ \begin{array}{ll} \Gamma_F & \hbox{if $F$ separates} \\
D(\Gamma_F) & \hbox{if $F$ does not separate}
\end{array} \right.  \\
$$
It is clear that
\begin{enumerate}

\vspace{-.1cm} \item $\Gamma_S$ is a graph in $Y$ determined by the intersection of an immersed disk or torus with $S$.

\vspace{.3cm} \item each vertex of $\Gamma_S$ has valency $|\partial S| \Delta(\alpha, \beta)$.

\vspace{.3cm} \item if two faces of $\Gamma_S$ share a common edge, then they lie on different sides of $S$.

\vspace{.3cm} \item if $F$ does not separate, then a face of $\Gamma_S$ which is sent by $h$ into $X^-$ is a bigon bounded by parallel edges.

\end{enumerate}
\vspace{-.3cm}
Suppose that $e$ and $e'$ are two adjacent parallel edges of $\Gamma_S$.
Let $R$ be the bigon face between them, realizing the parallelism. Then $(R, e\cup e')$ is mapped into
$(X^\epsilon, S)$ by the map $h$ for some $\epsilon$. Moreover $h|_R$ provides a basic essential
homotopy between the essential paths $h|_e$ and $h|_{e'}$. We may
and shall assume that $R^*=h(R)$ is contained in the characteristic
$I$-bundle pair $(\dot{\Sigma}_1^\epsilon, \dot{\Phi}_1^\epsilon)$ of $(X^\epsilon,
S)$. We may consider $R$ as $e\times I$ and assume that the map $h:
R \ra \dot{\Sigma_1^\epsilon}$ is $I$-fibre preserving.

Let $\overline{\Gamma}_S$ be the {\it reduced graph of $\Gamma_S$} obtained from $\Gamma_S$ by
amalgamating each maximal family of parallel edges into a single edge. It is evident that $\overline{\Gamma}_S$ coincides with the similarly defined graph $\overline{\Gamma}_F$. The following lemma is a simple consequence of the construction of $\Gamma_S$.

\begin{lemma} \label{1-edge and 2-edge}
$\;$

$(1)$ There is at most  one $1$-edge face of $\overline{\Gamma}_S$ and if one, it is a collar on $\partial Y$ when $Y$ is a disk and a once-punctured torus when $Y$ is a torus.

$(2)$ A $2$-edge face of $\overline{\Gamma}_S$ is either

\indent \hspace{3mm} $(a)$ a collar on $\partial Y$ bounded by a circuit of two edges and two vertices when $Y$ is a disk;

\indent \hspace{3mm} $(b)$ a once-punctured torus bounded by a circuit of two edges and two vertices;

\indent \hspace{3mm} $(c)$ an annulus cobounded by two circuits, each with one edge and one vertex;

\indent \hspace{3mm} $(d)$  a twice-punctured torus bounded by a circuit of one edge and one vertex.
\qed

\end{lemma}

The {\it weight} of an edge $\bar e$ of $\overline{\Gamma}_S$ is
the number of parallel edges in $\Gamma_S$ that $\bar e$ represents.

Call the vertex of $\Gamma_S$ (or $\overline{\Gamma}_S$) with boundary $a_i$ {\it positive} if $a_i$ and $a_1$ are like-oriented on $\partial M$. Otherwise call it {\it negative}.

Call an edge $e$, respectively $\bar e$, of $\Gamma_S$, respectively $\overline{\Gamma}_S$, {\it positive}
if it connects two positive vertices or two negative vertices. Otherwise it is said to be
{\it negative}.

\begin{prop} \label{plus-minus vertices}
If $Y$ is a torus, the number of positive vertices of $\Gamma_S$ equals the number of negative vertices.
\end{prop}

\pf Up to taking absolute value, the difference between the number of positive vertices and the number of negative vertices is the intersection number between a class in $H_1(M(\alpha))$ carried by the core of the $\alpha$-filling torus and $h_*([Y]) \in H_2(M(\alpha))$. Thus the lemma holds as long as $H_2(M(\alpha)) = 0$. Suppose then that $H_2(M(\alpha)) \ne 0$. Since $M(\alpha)$ is small Seifert, we have $H_2(M(\alpha)) \cong \mathbb Z$ and is generated by an embedded horizontal surface, $G$ say, which is a fibre in a locally trivial surface bundle $M(\alpha) \to S^1$. Thus $h_*([Y])$ is a non-zero multiple of $[G]$. In particular, the Thurston norm of $[G]$ is zero. Hence $G$ is a torus (cf. Assumption \ref{assumption irreducible}). Thus $M(\alpha)$ is toroidal small Seifert. But then $M(\alpha)$ is very small contrary to our assumption that $Y$ is a torus. This completes the proof.
\qed

To each oriention of an edge $\bar e$ of $\overline{\Gamma}_S$ of weight $|\partial S|$ or more we can associate a permutation $\sigma$ of the labels as follows: if $e$ is an edge of $\Gamma_S$ in the $\bar e$-family and $j$ is the label of its tail, then $\sigma(j)$ is the label of its head. The parity rules implies that there is an integer $k$ such that
$$\sigma(j) \equiv \left\{ \begin{array}{ll}
j + 2k \hbox{ (mod $m$)} & \hbox{if $e$ is negative} \\
-j + 2k +1 \hbox{ (mod $m$)} & \hbox{if $e$ is positive}
\end{array} \right.$$

We say that a face of $\Gamma_S$ or $\overline{\Gamma}_S$ {\it lies on the $\epsilon$-side of $F$} if it is mapped to $X^\epsilon$ by $h$.

The discussion in \cite[\S 3.4]{BCSZ1} implies the conclusion of the following proposition.

\begin{prop} \label{long}
If \; $\overline{\Gamma}_S$ has an edge of weight $k$, then there is an essential homotopy in $(M,S)$ of length $k-1$ of a large map with image in $S$.
\qed
\end{prop}

This result combines with Corollary \ref{general length} to yield the following corollary.

\begin{cor} \label{max weight of edge}
Suppose that $\Delta(\alpha, \beta) > 3$.

$(1)$ If $X^-$ is not an $I$-bundle, the weight of an edge of $\overline{\Gamma}_S$ is at most $|\partial S| +1$.

$(2)$ If $X^-$ is a product $I$-bundle, or a twisted $I$-bundle and $\Phi_3^+$ does not contain an $\widehat F$-essential annulus, the weight of an edge of $\overline{\Gamma}_S$ is at most $|\partial S| + 2$.

$(3)$ If $X^-$ is a twisted $I$-bundle and $\Phi_3^+$ contains an $\widehat F$-essential annulus, the weight of an edge of $\overline{\Gamma}_S$ is at most $|\partial S| + 4$.
\qed
\end{cor}

We call a graph {\it hexagonal} if it is contained in a torus, each vertex has valency $6$, and each face is a triangle. We call it {\it rectangular} if it is contained in a torus, each vertex has valency $4$, and each face is a rectangle. Such graphs are connected.

The following proposition follows from simple Euler characteristic
calculations.

\begin{prop}\label{euler} $\;$

$(1)$ If each vertex of $\overline{\Gamma}_S$ has valency $6$ or more, then it is hexagonal, so $Y$ is a torus. Moreover there is a vertex of $\overline{\Gamma}_S$ incident to at least two positive edges.

$(2)$ If $\overline{\Gamma}_S$ has no triangle faces, it has a vertex of valency at most $4$. If it has no vertices of valency less than $4$, then  it is rectangular, so $Y$ is a torus.
\qed
\end{prop}

\pf We have
$$0 \leq \chi(Y) = \sum_{\hbox{{\footnotesize faces} } f \hbox{ {\footnotesize of} } \overline{\Gamma}_S} \{\chi(f) - \sum_{v \in \partial f} (\frac{1}{2} - \frac{1}{\hbox{valency}_{\overline{\Gamma}_S}(v)})\}$$
Set $\chi_f = \chi(f) - \sum_{v \in \partial f} (\frac{1}{2} - \frac{1}{\hbox{valency}_{\overline{\Gamma}_S}(v)})$.
Lemma \ref{1-edge and 2-edge} implies that $\chi_{f} \leq 0$ for each monogon and bigon $f$ in $\overline{\Gamma}_S$. The hypotheses of assertions (1) and (2) of the lemma imply that $\chi_f \leq 0$ for  faces with three or more sides. Thus under either set of  hypotheses, $\chi(Y) \leq 0$, so $Y$ is a torus. Then $\chi(Y) = 0$ so $\chi_f = 0$ for all faces $f$. It follows that $\overline{\Gamma}_S$ is hexagonal under the conditions of (1) and rectangular under those of (2). Finally, it is easy to check that when $\overline{\Gamma}_S$ is hexagonal, it has a vertex incident to at least two positive edges.
\qed

\begin{lemma}\label{even}
Suppose that $F$ is non-separating and each component of $\breve{\Sigma}_1^+$ intersects both $F_1$ and $F_2$. Then every edge of $\overline{\Gamma}_S$ is negative. Hence every face of $\overline{\Gamma}_S$ has an even number of edges. In particular this is true if $t_1^+ \leq 2$.
\end{lemma}

\pf If each component of $\breve{\Sigma}_1^+$ intersects both $F_1$ and $F_2$, all the boundary components
$b_1,...,b_m$ of $F$ have the same orientation. Hence by the parity
rule, every edge of $\overline{\Gamma}_S$ is negative. The second assertion
follows from the first, while the third is a consequence of Lemma \ref{4} and the others.
\qed

A  disk face  of $k$-edges in the graph $\Gamma_S$ is call a
{\it Scharlemann $k$-gon with label pair $\{j,j+1\}$}
if each edge of the face is positive with
the fixed label pair $\{j,j+1\}$ at its two endpoints.
The set of edges of a Scharlemann $k$-gon is called a {\it Scharlemann $k$-cycle}.
A Scharlemann $2$-cycle  is also called an {\it $S$-cycle}.
An $S$-cycle $\{e_1, e_2\}$ is called an {\it extended $S$-cycle} if
$m\geq 4$ and the two edges $e_1$ and $e_2$ are the middle edges
in a family of four adjacent parallel edges of $\Gamma_S$. An $S$-cycle $\{e_1, e_2\}$ is called a {\it doubly-extended $S$-cycle} if
$m \geq 6$ and the two edges $e_1$ and $e_2$ are the middle edges
in a family of six adjacent parallel edges of $\Gamma_S$.

The method of proof of \cite[Lemma 12.3]{BCSZ2} yields the following proposition.

\begin{prop} \label{Scycle}
Suppose that two vertices $v$ and $v'$
of $\Gamma_S$ have the same orientation and
are connected by a family of $n$ parallel consecutive edges
$e_1,...,e_n$.

$(1)$ If $n > m/2$, then there is an $S$-cycle in this family of edges.

$(2)(a)$ If $m \geq 4$ and $n > \frac{m}{2}+1$, then either there is an extended
$S$-cycle in this family of edges or both $\{e_1,e_2\}$ and $\{e_{n-1},e_n\}$ are $S$-cycles.

\indent \hspace{4mm} $(b)$ If $m \geq 4$ and $n> \frac{m}{2}+2$, then there is an extended $S$-cycle in this family of edges.

$(3)(a)$ If $m \geq 6$ and $n > \frac{m}{2}+3$, then either there is a doubly-extended
$S$-cycle in this family of edges or both $\{e_2,e_3\}$ and $\{e_{n-2},e_{n-1}\}$ are extended $S$-cycles.

\indent \hspace{4mm} $(b)$ If $m \geq 6$ and $n> \frac{m}{2}+4$, there is a doubly-extended $S$-cycle in this family of edges. \qed
\end{prop}

\begin{lemma}\label{SC}
Suppose that $\{e_1,e_2\}$ is an $S$-cycle in $\Gamma_S$ and $R$ the associated bigon face of $\Gamma_S$. If $R$ lies on the $\epsilon$-side of $F$, then $\dot{\Phi}^\epsilon_1$ contains a $\tau_\epsilon$-invariant component, so $F$ is separating. Further, this component contains an $\widehat S$-essential annulus  and $\widehat X^\epsilon$ admits a Seifert structure with base orbifold a disk with two cone points, at least one of which has order $2$.
\end{lemma}

\pf Suppose that the $S$-cycle has label pair $\{j,j+1\}$. Then
$\tau_\epsilon(b_j\cup e_1^*\cup b_{j+1})= b_{j+1}\cup e_2^*\cup b_j$. Hence
$ b_j\cup e_1^*\cup b_{j+1}$ and $b_{j+1}\cup e_2^*\cup b_j$ are
contained in the same component $\phi$ of $\dot{\Phi}^\epsilon$ and this
component is $\tau_\epsilon$-invariant. Proposition \ref{product-Sigma+}
implies that $S$ is connected and $\phi$ contains an $\widehat S$-essential annulus. Proposition \ref{sep-seifert} shows that $\phi$ is the unique component of $\dot{\Phi}^\epsilon$ to contain such an annulus. Finally, Proposition \ref{sep-seifert}(3) implies that $\widehat X^\epsilon$ is of the form described in (4).
 \qed

\section{Counting faces in $\overline{\Gamma}_S$} \label{triangle faces}

In this section we examine the existence of triangle faces of $\overline{\Gamma}_S$ incident to vertices of small valency.

For each vertex $v$ of $\Gamma_S$ let $\varphi_j(v)$ be the number of corners of $j$-gons incident to $v$. Then
$$\hbox{valency}_{\overline{\Gamma}_S}(v) = |\partial S| \Delta(\alpha, \beta) - \varphi_2(v)$$
Set
$$\psi_3(v) = \hbox{valency}_{\overline{\Gamma}_S}(v) - \varphi_3(v) \geq 0,$$
$$\mu(v) = \varphi_2(v) + \frac{\varphi_3(v)}{3} \in \{ \frac{k}{3} : k \in \mathbb Z\}$$

\begin{lemma} \label{size mu}
Suppose that $v$ is a vertex of $\Gamma_S$ and set $\mu(v) = |\partial S| \Delta(\alpha, \beta) - 4 + x$. Then
$$\hbox{valency}_{\overline{\Gamma}_S}(v) = 6 - \frac12 (3x + \psi_3(v))$$
and
$$\varphi_3(v) = 3(\hbox{valency}_{\overline{\Gamma}_S}(v) - 4 + x)$$
\end{lemma}

\pf We noted above that $\hbox{valency}_{\overline{\Gamma}_S}(v) = |\partial S| \Delta(\alpha, \beta) - \varphi_2(v)$. Thus
$$\hbox{valency}_{\overline{\Gamma}_S}(v) = |\partial S| \Delta(\alpha, \beta) - \mu(v) + \frac{\varphi_3(v)}{3}
= 4 - x + \frac{\hbox{valency}_{\overline{\Gamma}_S}(v)}{3} - \frac{\psi_3(v)}{3},$$
and therefore
$$ \hbox{valency}_{\overline{\Gamma}_S}(v) = \frac32 (4 - x - \frac{\psi_3(v)}{3}) = 6 - \frac{1}{2} (3x + \psi_3(v)).$$
On the other hand,
$$\varphi_3(v) = 3(\mu(v) - \varphi_2(v)) = 3(|\partial S| \Delta(\alpha, \beta) - 4 + x - \varphi_2(v)) =
3(\hbox{valency}_{\overline{\Gamma}_S}(v) - 4 + x).$$
Thus the lemma holds.
\qed

\begin{prop} \label{possible values for mu}
Suppose that $v$ is a vertex of $\Gamma_S$.

$(1)$ If $\mu(v) > |\partial S| \Delta(\alpha, \beta) - 4$, then $\hbox{valency}_{\overline{\Gamma}_S}(v) \leq 5$. Further,

\indent \hspace{5mm} $(a)$ if $\hbox{valency}_{\overline{\Gamma}_S}(v) = 3$, then $\varphi_3(v) \geq 0$.

\indent \hspace{5mm} $(b)$ if $\hbox{valency}_{\overline{\Gamma}_S}(v) = 4$, then $\varphi_3(v) \geq 1$.

\indent \hspace{5mm} $(c)$ if $\hbox{valency}_{\overline{\Gamma}_S}(v) = 5$, then $\varphi_3(v) \geq 4$.

$(2)$ If $\mu(v) = |\partial S| \Delta(\alpha, \beta) - 4$, then $ 4 \leq \hbox{valency}_{\overline{\Gamma}_S}(v) \leq 6$. Further,

\indent \hspace{5mm} $(a)$ if $\hbox{valency}_{\overline{\Gamma}_S}(v) = 4$ then $\varphi_3(v) = 0$.

\indent \hspace{5mm} $(b)$ if $\hbox{valency}_{\overline{\Gamma}_S}(v) = 5$ then $\varphi_3(v) = 3$.

\indent \hspace{5mm} $(c)$ if $\hbox{valency}_{\overline{\Gamma}_S}(v) = 6$ then $\varphi_3(v) = 6$.

\end{prop}

\pf Write $\mu(v) = |\partial S| \Delta(\alpha, \beta) - 4 + x$ where $x \geq 0$ is an element of $\{ \frac{k}{3} : k \in \mathbb Z\}$. By Lemma \ref{size mu} we have $\hbox{valency}_{\overline{\Gamma}_S}(v) \leq 6 - \frac{3x}{2}$. Thus
$\hbox{valency}_{\overline{\Gamma}_S}(v) \leq \left\{
\begin{array}{ll}
6 & \hbox{if } x = 0 \\
5 & \hbox{if } x > 0
\end{array} \right.$.
Further, if $x = 0$, the same lemma implies that $\hbox{valency}_{\overline{\Gamma}_S}(v) = 6 - \frac{\psi_3(v)}{2}$. Since $\psi_3(v) = \hbox{valency}_{\overline{\Gamma}_S}(v) - \varphi_3(v)$, this is equivalent to $\hbox{valency}_{\overline{\Gamma}_S}(v) = 4 + \frac{\varphi_3(v)}{3}$. Thus $\hbox{valency}_{\overline{\Gamma}_S}(v) \geq 4$. The remaining conclusions follow from the identity $\varphi_3(v) = 3(\hbox{valency}_{\overline{\Gamma}_S}(v) - 4 + x)$ of Lemma \ref{size mu}.
\qed

Let $V, E, F$ be the number of vertices, edges, and faces of $\overline{\Gamma}_S$.

\begin{prop} \label{mu sum} $\;$

$(1)$ If the immersion surface is a disk, then $\sum_v \mu(v) \geq (|\partial S| \Delta(\alpha, \beta) - 4)V + 4$.

$(2)$ If the immersion surface is a torus, $\sum_v \mu(v) \geq (|\partial S| \Delta(\alpha, \beta) - 4)V$.
\end{prop}

\pf First assume that $\Gamma_S$ has no monogon faces. Since its vertices each have valency $|\partial S| \Delta(\alpha, \beta)$ we have $2E = |\partial S| \Delta(\alpha, \beta) V$. Let $F_i$ be the number of $i$-faces so $F = \sum_i F_i$ and $2E = \sum_i iF_i$. Then
$$(|\partial S| \Delta(\alpha, \beta) - 4)V = 2E - 4V = 4(E - V) - 2E  =  4((\sum_{\hbox{{\footnotesize faces} } f} \chi(f)) - \chi(Y)) - 2E$$
Since $\chi(f) \leq 1$ for each face $f$, we have
\begin{eqnarray}
(|\partial S| \Delta(\alpha, \beta) - 4)V \leq 4(F - \chi(Y)) - 2E & = & \sum (4 - i) F_i - 4 \chi(Y)  \nonumber \\
& \leq & 2F_2 + F_3 - 4 \chi(Y) \nonumber \\
& = & \sum_v (\varphi_2(v) + \frac{\varphi_3(v)}{3})  - 4 \chi(Y) \nonumber \\
& = & \sum_v \mu(v) - 4 \chi(Y) \nonumber
\end{eqnarray}
Thus the lemma holds when there are no monogons.

If there are monogons, it is easily verified that there is only one, $f$ say, and that it is a collar on $\partial Y$ when $Y$ is a disk and a once-punctured torus when $Y$ is a torus. In either case, $\overline{Y \setminus f}$ is a disk containing $\Gamma_S$ without monogons. The first case implies that $\sum_v \mu(v) \geq (|\partial S| \Delta(\alpha, \beta) - 4)V + 4$, which implies the result.
\qed

\begin{cor} \label{mu constant} $\;$

$(1)$ If the immersion surface is a disk there is a vertex $v$ of $\Gamma_S$ such that $\mu(v) > |\partial S| \Delta(\alpha, \beta) - 4$.

$(2)$ If the immersion surface is a torus, then either there is a vertex $v$ of $\Gamma_S$ such that $\mu(v) > |\partial S| \Delta(\alpha, \beta) - 4$ or $\mu(v) = |\partial S| \Delta(\alpha, \beta) - 4$ for each vertex.
\qed
\end{cor}

\begin{prop} \label{local structure}
Suppose that $\mu(v) = |\partial S| \Delta(\alpha, \beta) - 4$ for each vertex $v$ of $\Gamma_S$. Then each face of $\Gamma_S$ is a disk. Further, if $v$ is a vertex of $\Gamma_S$ and

$(1)$ $\hbox{valency}_{\overline{\Gamma}_S}(v) = 4$, then $\varphi_4(v) = 4$.

$(2)$ $\hbox{valency}_{\overline{\Gamma}_S}(v) = 5$, then $\varphi_3(v) = 3$ and $\varphi_4(v) = 2$.

$(3)$ $\hbox{valency}_{\overline{\Gamma}_S}(v) = 6$, then $\varphi_3(v) = 6$.
\end{prop}

\pf Corollary \ref{mu constant} shows that $Y$ is a torus. Thus $0 = \chi(Y) = \sum_v  \chi(v)$ where
$$\chi(v) = 1 - \frac{\hbox{valency}_{\overline{\Gamma}_S}(v)}{2} + \sum_{v \in \partial f} \frac{\chi(f)}{|\partial f|}$$
and $f$ ranges over the faces of $\overline{\Gamma}_S$ containing $v$. From Proposition \ref{possible values for mu}(2) we see that $\chi(v) \leq 0$ for all $v$. Hence $\chi(v) = 0$ for all $v$. This is only possible if  the proposition holds.
\qed

\section{Proof of Theorem \ref{main} when $F$ is non-separating} \label{section non-sep}

We show that when $F$ is non-separating and $m \geq 3$, $\Delta(\alpha, \beta) \leq 4$ if $M(\alpha)$ is very small and $\Delta(\alpha, \beta) \leq 5$ otherwise. This follows from the two propositions below. Recall that $|\partial S| = 2m$ when $F$ is non-separating.

Proposition \ref{tight bound} shows that $l_S \leq 2m - t_1^+ + 1$, so the weight of each edge in the reduced graph $\overline{\Gamma}_S$ of $\Gamma_S$ is at most $2m - t_1^+ + 2$. Hence if $v$ is a vertex of $\overline{\Gamma}_S$, $2m \Delta(\alpha, \beta)/\hbox{valency}_{\overline{\Gamma}_S}(v) \leq 2m - t_1^+ + 2$, so
\begin{equation}\label{eqsection non-sep.1}
\text{\em $\Delta(\alpha, \beta) \leq \Big(\frac{ 2m - t_1^+ + 2}{2m}\Big) \hbox{valency}_{\overline{\Gamma}_S}(v)$}
\end{equation}

\begin{prop} \label{non-sep t_1^+ > 0}
Suppose that $F$ is non-separating and $t_1^+ > 0$. Then
$$\Delta(\alpha, \beta) \leq \left\{ \begin{array}{ll}
4 & \hbox{if } m \leq 5 \hbox{ or } M(\alpha) \hbox{ is very small} \\
5 & \hbox{if } m \geq 6 \\
\end{array} \right.$$

\end{prop}

\pf  If there is a vertex of $\overline{\Gamma}_S$ of valency $3$ or less, Inequality \ref{eqsection non-sep.1} yields $\Delta(\alpha, \beta) \leq 3$, so we are done. Suppose then that all vertices are of valency $4$ or more.

If $t_1^+ = 2$, then by Lemma \ref{even} there are no triangle faces of $\Gamma_S$ and therefore Lemma \ref{euler}(2) implies that $\overline{\Gamma}_S$ is quadrilateral. Thus $Y$ is a torus, so $M(\alpha)$ is not very small. Further, as all vertices have valency $4$, Inequality \ref{eqsection non-sep.1} implies that $\Delta(\alpha, \beta) \leq 4$. Thus we are done.

If $t_1^+ > 2$, then $2m \geq t_1^+ \geq 4$, so $m \geq 2$. Corollary \ref{mu constant} and Proposition \ref{possible values for mu} imply that there is a vertex $v$ of $\overline{\Gamma}_S$ of valency at most $5$ if $Y$ is a disk (e.g. if $M(\alpha)$ is very small) and at most $6$ if it is a torus. Inequality \ref{eqsection non-sep.1} then shows that the proposition holds.
\qed

\begin{prop} \label{non-sep t_1^+ = 0}
Suppose that $F$ is non-separating and $t_1^+ = 0$.

$(1)$ If $M(\alpha)$ is very small, then $\Delta(\alpha, \beta) \leq \left\{ \begin{array}{ll}
4 & \hbox{if } m \geq 2 \\
6 & \hbox{if } m = 1
\end{array} \right.$

$(2)$ If $M(\alpha)$ is not very small, then $\Delta(\alpha, \beta) \leq \left\{ \begin{array}{ll}
5 & \hbox{if } m \geq 3 \\
6 & \hbox{if } m = 2 \\
8 & \hbox{if } m = 1
\end{array} \right.$
\end{prop}

\pf Suppose that $t_1^+ = 0$. By Lemma \ref{even}, $\overline{\Gamma}_S$ has no triangle face, so $\varphi_3(v) = 0$ for each vertex of $\Gamma_S$. Hence $\overline{\Gamma}_S$ has a vertex $v$ of valency at most $4$ by Proposition \ref{euler}(2). If there is a vertex of valency $3$ or less, then Inequality \ref{eqsection non-sep.1} shows $\Delta(\alpha, \beta) \leq 4$ for $m \geq 2$ and $\Delta(\alpha, \beta) \leq 6$ for $m = 1$. If there are no vertices of valency less than $4$, Proposition \ref{euler}(2) implies that $\overline{\Gamma}_S$ is rectangular, so $Y$ is a torus and $M(\alpha)$ is not very small. Thus assertion (1) of the lemma holds. By Inequality \ref{eqsection non-sep.1}, $\Delta(\alpha, \beta) \leq 4 + 4/m$. It follows that $\Delta(\alpha, \beta) \leq 5$ if $m \geq 3$, $\Delta(\alpha, \beta) \leq 6$ if $m = 2$, and $\Delta(\alpha, \beta) \leq 8$ if $m = 1$.
\qed

\section{Proof of Theorem \ref{main} when $F$ is separating and  $t_1^+ + t_1^- \geq 4$} \label{4 or more tights}

\begin{prop} \label{sep t_1^+ + t_1^- > 2}
Suppose that $F$ is separating and $t_1^+ + t_1^- \geq 4$. Then
$$\Delta(\alpha, \beta) \leq \left\{
\begin{array}{ll}
4 & \hbox{if $M(\alpha)$ is very small} \\
5 & \hbox{otherwise.} \end{array} \right.$$
\end{prop}

\begin{proof}
Since $F$ is separating, $S = F$ and $|\partial S| = m$.

If $t_1^\epsilon \geq 4$ for some $\epsilon$ then Proposition \ref{tight bound} shows that $l_S \leq m - 3$. Thus the weight of each edge in $\overline{\Gamma}_S$ is at most $m-2$. If $t_1^\epsilon = 2$ for both $\epsilon$, then $l_+, l_- \leq m - 2$, so $l_S \leq m - 2$. Thus the weight of each edge in $\overline{\Gamma}_S$ is at most $m-1$. In either case, it follows that for each vertex $v$ of $\Gamma_S$, $\frac{m \Delta(\alpha, \beta)}{\hbox{valency}_{\overline{\Gamma}_S}(v)} \leq m-1$. Hence
\begin{equation}\label{eq 4 or more tights.1}
\text{\em $\Delta(\alpha, \beta) \leq \Big(\frac{ m - 1}{m}\Big) \hbox{valency}_{\overline{\Gamma}_S}(v) < \hbox{valency}_{\overline{\Gamma}_S}(v)$}
\end{equation}
Corollary \ref{mu constant} and Proposition \ref{possible values for mu} imply that there is a vertex $v$ of valency $5$ or less if $Y$ is a disk, in particular if $M(\alpha)$ is very small, and of valency at most $6$ otherwise. Inequality \ref{eq 4 or more tights.1} then shows that the conclusion of the proposition hold.
\end{proof}

\section{The relation associated to a face of $\Gamma_S$} \label{section-relation}

The proof of Theorem \ref{main} when $F$ is separating and $t_1^+ + t_1^- \leq 2$ necessitates a deeper use of the properties of the intersection graph $\Gamma_S$. We begin with a description of the relations associated to its faces.

Recall that the boundary components of $F$ have been indexed (mod $m$): $b_1, b_2, \ldots, b_m$ so that they appear successively around $\partial M$. For each $\epsilon$ we use $\tau_\epsilon(j) (= j \pm 1)$ to be the index such that $\tau_\epsilon(b_j) = b_{\tau_\epsilon(j)}$. Let $\sigma_{j}$ be a path which runs from $b_j$ to $b_{\tau_\epsilon(j)}$ in the annular component of $\partial M \cap X^\epsilon$ containing $b_j \cup b_{\tau_\epsilon(j)}$. Fix a base point $x_0 \in F$ and for each $j$ a path $\eta_j$ in $F$ from $x_0$ to $b_j$. The loop $\eta_j * \sigma_{j} * \eta_{\tau_\epsilon(j)}^{-1}$ determines a class $x_{j} \in \pi_1(\widehat X^\epsilon; x_0)$ well-defined up to our choice of the $\eta_j$. Clearly $x_{j} x_{\tau_\epsilon(j)} = 1$. (The use of $x_j$ to describe this class is ambiguous in that it does not specify which value $\epsilon$ takes on. Nevertheless, whenever we use it the value of $\epsilon$ will be understood from the context.)

Recall that if $t_1^\epsilon = 0$, $\widehat X^\epsilon$ admits a Seifert structure with base orbifold $D^2(p,q)$. There is a projection homomorphism $\pi_1(\widehat X^\epsilon) \to \pi_1(D^2(p,q))$ obtained by quotienting out the normal cyclic subgroup of $\pi_1(\widehat X^\epsilon)$ determined by the class of a regular Seifert fibre. We denote the image in $\pi_1(D^2(p,q))$ of an element $x \in \pi_1(\widehat X^\epsilon)$ by $\bar x$.  Fix generators $a, b$ of $\mathbb Z/p, \mathbb Z/q$ such that $ab$ represents the class of the boundary circle of $D^2(p,q)$ in $\pi_1(D^2(p,q)) \cong \mathbb Z/p * \mathbb Z/q$.

\begin{prop} \label{not peripheral}
If $t_1^\epsilon = 0$, then no $x_{j}$ is peripheral in $\widehat X^\epsilon$. Indeed, there are integers $k, l$ and $\delta \in \{\pm 1\}$ such that $x_{j}$ is sent to an element $\bar x_{j}$ of the form $(ab)^k a^\delta (ab)^l$ in $\pi_1(D^2(p,q))$.
\end{prop}

\pf It follows from the method of proof of Proposition \ref{order two} that $\widehat X^\epsilon = V \cup W$ where $V$ and $W$ are solid tori whose intersection is an essential annulus $(A, \partial A) \subset (\widehat X^\epsilon, \widehat F)$. Further, if $K_\beta$ is the core of the $\beta$ filling solid torus, we can assume that $K_\beta \cap X^\epsilon$ is a finite union of arcs properly embedded in $A$. Consideration of the Seifert structure on $\widehat X^\epsilon$ then shows that the image of the projection of $\sigma_{j}$ to $D^2(p,q)$ is a properly embedded arc which separates the two cone points. Thus there are integers $k, l$ and $\delta \in \{\pm 1\}$ such that $\bar x_{j} = (ab)^k a^\delta (ab)^l \in \pi_1(D^2(p,q))$. Such an element is peripheral if and only if it equals $(ab)^n$ for some $n$. But then $a = (ab)^{\pm(n-k-l)}$ would be peripheral, which is false.
\qed

Consider an $n$-gon face $f$ of $\Gamma_S$ lying to the $\epsilon$-side of $F$ with boundary $c_1 \cup e_1 \cup c_2 \cup \ldots \cup c_n \cup e_n$ where each $c_i$ is a corner of $f$, $e_i$ an edge of $f$, and they are indexed as they arise around $\partial f$. In this ordering, let $b_{j_i}$ be the boundary component of $F$ at $c_i$ corresponding to $c_i \cap e_i$ and $b_{j_i'}$ that corresponding to $c_{i+1} \cap e_i$. (See Figure \ref{relation}.)

\begin{figure}[!ht]
\centerline{\includegraphics{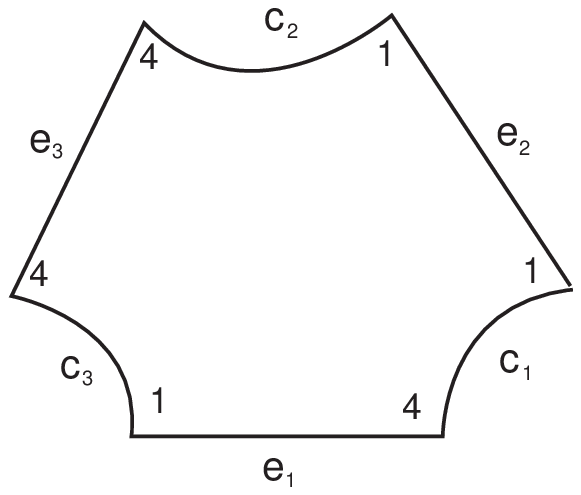}} \caption{ }\label{relation}
\end{figure}

The relation
$$\Pi_{i=1}^n w_i x_{j_{i}'}  = 1$$
holds in $\pi_1(\widehat X^\epsilon)$ where $w_i$ is represented by the loop $\eta_{j_i}  * e_i^* * \eta_{j_{i}'}^{-1}$.

For each boundary component $b_j$ of $F$, let $\widehat b_j$ denote the meridional disk it bounds in $\widehat F$.

\begin{cor} \label{F hat essential edge}
Suppose that $e_1$ is a negative edge of $\Gamma_S$ whose end labels are the same. Suppose as well that $e_1$ is a boundary edge of a triangle face lying on the $\epsilon$-side of $F$ where $t_1^\epsilon = 0$. If the boundary label of $e$ is $j$, then the loop $\widehat b_j \cup e_1^*$ is essential in $\widehat F$.
\end{cor}

\pf The relation from the given face reads $x_{j}^{-1} w_1 x_{j} w_2 x_{k} w_3 = 1$ where $k \in \{1, 2, \ldots , m\}$ and $w_1, w_2, w_3$ are the peripheral elements of $\pi_1(\widehat X^\epsilon)$ defined above. If $\widehat b_j \cup e_1^*$ is inessential in $\widehat F$, then $w_1 = 1$, so the relation gives $x_{k} = (w_3 w_2)^{-1}$ is peripheral, which contradicts Proposition \ref{not peripheral}. Thus the corollary holds.
\qed

As an immediate consequence of this corollary we have:

\begin{cor} \label{triangle not tight}
Suppose that $e$ is a negative edge of $\Gamma_S$ whose end labels are the same. Suppose as well that $e$ is a boundary edge of a triangle face lying on the $\epsilon$-side of $F$ where $t_1^\epsilon = 0$. If the weight of $e$ in the reduced graph $\overline{\Gamma}$ is $k+1$, then $e^*$ is contained in a component of $\dot \Phi_k^{-\epsilon}$ which contains an $\widehat F$-essential annulus.
\qed
\end{cor}

\section{Proof of Theorem \ref{main} when $F$ is separating and $t_1^+ + t_1^- = 2$} \label{2 tights}

We assume that $F$ is separating and $t_1^+ + t_1^- = 2$ in this section. There is an $\epsilon$ such that $t_1^\epsilon = 2$ and $t_1^{-\epsilon} = 0$. Without loss of generality we can suppose that $\epsilon = +$.

\begin{prop} \label{sep equal 2}
If $F$ is separating and $t_1^+ = 2, t_1^- = 0$, then
$$\Delta(\alpha, \beta) \leq \left\{ \begin{array}{ll} 5 & \hbox{if } m \geq 4
\\ 6  & \hbox{if } m = 2
\end{array} \right.$$
\end{prop}

\pf Proposition \ref{tight bound} shows that $l_+ \leq m - 2$ and $l_S \leq m - 1$. Thus the weight of each edge in $\overline{\Gamma}_S$ is at most $m$. Hence if there is a vertex of $\overline{\Gamma}_S$ of valency $k$, then $m \Delta(\alpha, \beta) \leq km$, so $\Delta(\alpha, \beta) \leq k$. In the case that $\mu(v) > m \Delta(\alpha, \beta) - 4$ for some vertex $v$ of $\overline{\Gamma}_S$, Proposition \ref{possible values for mu}(1) implies that $\Delta(\alpha, \beta) \leq 5$. In particular this is true when $M(\alpha)$ is very small by Corollary \ref{mu constant}(1). By Corollary \ref{mu constant}(2) we can therefore suppose that $Y$ is a torus and $\mu(v) = m \Delta(\alpha, \beta) - 4$ for all vertices $v$ of $\overline{\Gamma}_S$. Then $\Delta(\alpha, \beta) \leq 6$ by  Proposition \ref{possible values for mu}(2).

To complete the proof we shall suppose that $\Delta(\alpha, \beta) = 6$ and show that $m = 2$. In this case $\overline{\Gamma}_S$ has no vertices of valency $5$ or less. Thus it is hexagonal (Proposition \ref{euler}). As no edge of $\overline{\Gamma}_S$ has weight larger than $m$, each of its edges has weight $m$. It follows that $l_+ \geq m-2$, and since we noted above that $l_+ \leq m - 2$, we have $l_+ = m-2$. Thus each face of $\overline{\Gamma}_S$ lies on the $+$-side of $F$.

Note that $t_{m-3}^+ < m$ since $l_+ = m - 2$. The fact that $t_{2j+1}^+$ is even couples with Proposition \ref{tight increase} to show that $t_{m-3}^+ = m-2$. Thus $\breve \Phi_{m-3}^+$ has at least $m-2$ components. If some such component $\phi_0$ contains at least three boundary components of $F$, $\breve\Phi_{m-3}^+$ has at most $m - 3$ other components. But then $\phi_0$ is tight, so there must be another component of $\breve \Phi_{m-3}^+$ containing at least three boundary components and therefore $m - 2 \leq |\breve \Phi_{m-3}^+| \leq m - 4$, a contradiction. Thus each component of $\breve \Phi_{m-3}^+$ contains at most two boundary components of $F$.

Let $b_1, b_2, \ldots, b_m$ be the boundary components of $F$ numbered in successive fashion around $\partial M$. Fix a triangle face $f$ of $\overline{\Gamma}_S$ and let $v_1, v_2$ be two of its vertices. They are connected by a family $e_1, e_2, \ldots, e_m$ of mutually parallel edges of $\Gamma_S$ successively numbered around $v_1$ so that $e_1$ is the boundary edge of $f$ thought of as a face in $\Gamma_S$.

We can suppose that the tail of each $e_i$ lies on $v_1$ and is labeled $i$. Let $j$ be the label of the head of $e_2$. If $v_1$ and $v_2$ are like-oriented, then $j$ is odd and $b_2 \cup e_2^* \cup b_j$ is contained in a component $\phi$ of $\dot \Phi_{m-3}^+$. From above, $b_2$ and $b_j$ are the only boundary components of $F$ $\phi$ contains. Similarly $b_{m-1} \cup e_{m-1}^* \cup b_{j+3}$ is contained in a component of $\dot \Phi_{m-3}^+$ and $b_{m-1}$ and $b_{j+3}$ are the only boundary components of $F$ it contains. Let $v_3$ be the third vertex of $f$ and consider the family of $m$ edges of $\Gamma_S$ parallel to the edge of $f$ connecting $v_1$ and $v_3$. The second edge from $f$ in this family has label $m-1$ at $v_1$ and its label at $v_3$ must be $j+3$ if $v_1$ and $v_3$ are like-oriented and $m-1$ otherwise. Similarly, the second edge from $f$ in the family of parallel edges corresponding to the edge of $f$ connecting $v_2$ and $v_3$ has label $m-1$ at $v_3$ if $v_2$ and $v_3$ are like-oriented and $j-3$ otherwise. Since the orientations of $v_1$ and $v_3$ coincide if and only if those of $v_2$ and $v_3$ do,
it follows that $j = m - 1$ whatever the relative orientations of $v_1$ and $v_3$. This implies that the head of each $e_i$ has label $m + 1 - i$.

A similar argument shows that the head of $e_i$ is labeled $i$ if $v_1$ and $v_2$ are oppositely-oriented.

Suppose that $m \geq 4$ and fix a triangle face $f$ of $\Gamma_S$ with one positive boundary edge and two negative ones (Proposition \ref{plus-minus vertices}). Let $v_1, v_2, v_3$ be the vertices of $f$ chosen so that the edge between $v_1$ and $v_2$ is positive. Number the family of $m$ parallel edges of $\Gamma_S$ connecting $v_1$ and $v_2$ as in the previous paragraph. In particular $e_1$ is an edge of $f$. Let $\phi$ be the component of $\dot \Phi_{m-3}^+$ containing $b_2 \cup b_{m-1}$. Consideration of the $m-2$ successive bigons connected by $e_3, e_4, \ldots , e_{m-2}$ shows that $h_{m-3}^+(\phi) = \phi$ (cf. the end of \S \ref{characteristic subsurfaces}). Equivalently, if $\epsilon = (-1)^{\frac{m}{2}}$ and $\phi' = (\tau_{-\epsilon} \circ \tau_{\epsilon} \circ \tau_{-\epsilon} \circ \ldots \circ \tau_+)(\phi)$ (a composition of $\frac{m}{2} - 2$ factors), then $\tau_{\epsilon}(\phi') = \phi'$. Hence $\phi'$, and therefore $\phi \subset \dot \Phi_{m-3}^+$ contains an $\widehat F$-essential annulus. It follows that the same is true for $\dot \Phi_j^+$ for each $j \leq m - 3$. (See  \ref{eq surface calculus}.)
Proposition \ref{sep-seifert} now implies that $\widehat X^+$ admits a Seifert structure with base orbifold of the form $D^2(a,b)$ where $a, b \geq 2$. Furthermore, Proposition \ref{sep annulus-in-3-+} implies that $m - 3 \leq 2$. Thus $m \leq 4$. We assume now that $m = 4$ and show that this leads to a contradiction. This will complete the proof.

Consideration of the family of parallel positive edges adjacent to $f$ shows that there is an $S$-cycle in $\Gamma_S$ lying on the $+$-side of $F$. Hence Lemma \ref{SC} implies that $\widehat X^+$ admits a Seifert structure with base orbifold $D^2(2,b)$ and $\dot \Phi_1^+$ has a unique component which completes to an $\widehat F$-essential annulus. It's not hard to see then that $\breve \Phi_{1}^+$ has three components: two boundary parallel annuli and a $4$-punctured sphere with two inner boundary components and two outer ones. If $\varphi_+$ denotes the slope on $\widehat F$ of the latter component, it is the slope of the Seifert structure on $\widehat X^+$.

Since $\Delta(\alpha, \beta) > 3$, $\beta$ is not a singular slope and therefore $M(\beta)$ is not Seifert with base orbifold $S^2(a,b,c,d)$ where $(a,b,c,d) \ne (2,2,2,2)$. Hence as $\dot \Phi_3^-$ contains an $\widehat F$-essential annulus, Proposition \ref{sep annulus-in-2} implies that $X^-$ is a twisted $I$-bundle. In particular, $\breve \Phi_{3}^- = \tau_-(\breve \Phi_{1}^+)$. Hence if $A$ is an $\widehat F$-essential annulus containing $\breve \Phi_{3}^-$, its slope $\varphi_-$ is given by $(\tau_-)_*(\varphi_+)$. It follows that $\Delta(\varphi_+, \varphi_-) \equiv 0$ (mod $2$).
Thus either $\Delta(\varphi_+, \varphi_-) = 0$ and $\widehat X^+(\varphi_-)$ is the connected sum of two non-trivial lens spaces or $\Delta(\varphi_+, \varphi_-) \geq 2$ and $\widehat X^+(\varphi_-)$ is a Seifert manifold with base orbifold $S^2(2,b, \Delta(\varphi_+, \varphi_-))$. In either case, $\pi_1(\widehat X^+(\varphi_-))$ is non-abelian.

Let $H_{(14)}$ be the component of $(\overline{M(\beta) \setminus M}) \cap \widehat X^+$ containing $b_1 \cup b_4$ and $\partial_0 H_{(14)}$ the annulus $H_{(14)} \cap X^+$. Then the image in $X^+$ of $\partial f$ lies in $A \cup \partial_0  H_{(14)}$. Moreover, once oriented, $\partial f$  intersects $\partial_0  H_{(14)}$ in three disjoint arcs exactly two of which are like-oriented. By an application of the Loop Theorem (see \cite[Theorem 4.1]{He}), there is a properly embedded disk $(D, \partial D) \subseteq (X^+, A \cup \partial_0  H_{(14)})$ such that $\partial D \cap \partial_0  H_{(14)} \subseteq \partial f \cap \partial_0  H_{(14)}$ and $\partial D$ algebraically intersects a core of $\partial_0  H_{(14)}$ a non-zero number of times (mod $3$). There are two possibilities,

\indent \hspace{.5cm} (1) $\partial D \cap \partial_0  H_{(14)}$ consists of two like-oriented arcs, or

\indent \hspace{.5cm} (2) $\partial D \cap \partial_0  H_{(14)}$ consists of three arcs, two like-oriented and one oppositely-oriented.

Suppose that (1) arises. Then $\partial D = e_1 \cup a_1 \cup e_2 \cup a_2$ where $a_1, e_1, a_2, e_2$ are arcs arising successively around $\partial D$ and $a_1, a_2$ are properly embedded in $H_{(14)}$ while $e_1, e_2$ are properly embedded in $F$. Let $\widehat b_i$ be the disk in $\widehat F$ with boundary $b_i$ and fix a (fat) basepoint in $\widehat X^+$ to be $\widehat b_1 \cup e_1 \cup \widehat b_4$. We take $\eta_1, \eta_4$ to be constant paths (see \S \ref{section-relation}). The ``loop" $e_2$ carries a generator $t$ of $\pi_1(A)$ as otherwise $M(\beta)$ would contain a $P^3$ connected summand. Thus the relation associated to $D$ is $x_1^2 = t$. Further, there is a M\"{o}bius band $B$ properly embedded in $X^+$ whose core carries  the class $x_1$. Consequently, $t$ represents the class of the slope $\varphi_+$. But by construction, it represents the class of $\varphi_-$. It follows that $\varphi_+ = \varphi_-$ so that $\widehat X^+(\varphi_-)$ is the connected sum of lens spaces $L(2,1)$ and $L(n, m)$ for some $n, m$. Note that $x_1$ represents a non-trivial class in $H_1(\widehat X^+(\varphi_-))$. But the relation associated to $f$ is of the form $t^a x_1 t^b x_1 t^c x_1^{-1} = 1$. In particular, $x_1$ is trivial in $H_1(\widehat X^+) / \langle t \rangle = H_1(\widehat X^+(\varphi_-))$. Thus possibility (1) cannot occur.

Suppose that (2) arises. Then $\partial D = e_1 \cup a_1 \cup e_2 \cup a_2 \cup e_3 \cup a_3$ where $a_1, e_1, a_2, e_2, a_3, e_3$ are arcs arising successively around $\partial D$ and $a_1, a_2, a_3$ are properly embedded in $H_{(14)}$ while $e_1, e_2, e_3$ are properly embedded in $F$. We can suppose that the indices are chosen so that $e_1$ connects $b_1$ to $b_4$, $e_2$ is a loop based at $b_4$ and $e_3$ is a loop based at $b_1$. These loops are essential as otherwise we could isotope $\partial D$ so that it intersects a core of $\partial_0 H_{(14)}$ once transverselly. This would imply that we could isotope $\widehat F$ in $M(\beta)$ to remove two points of intersection with the core of the $\beta$-filling solid torus contrary to Assumption \ref{assumption minimal}. Fix the basepoint in $\widehat X^+$ to be $\widehat b_1 \cup e_1 \cup \widehat b_4$ and take $\eta_1, \eta_4$ to be constant paths. The arcs $a_1, a_2, a_3$ determine triples of distinct points, one on $b_1$ and one on $b_4$, which we denote $1,2,3$. The reader will verify that these triples are oppositely orientated on $A$. From this it follows an orientation on $\partial D$ determines the same orientation on the loops $b_1 \cup e_2$ and $b_4 \cup e_3$. In particular they  yield the same generator $t$ of $\pi_1(A)$. (See Figure \ref{attachments}.)

\begin{figure}[!ht]
\centerline{\includegraphics{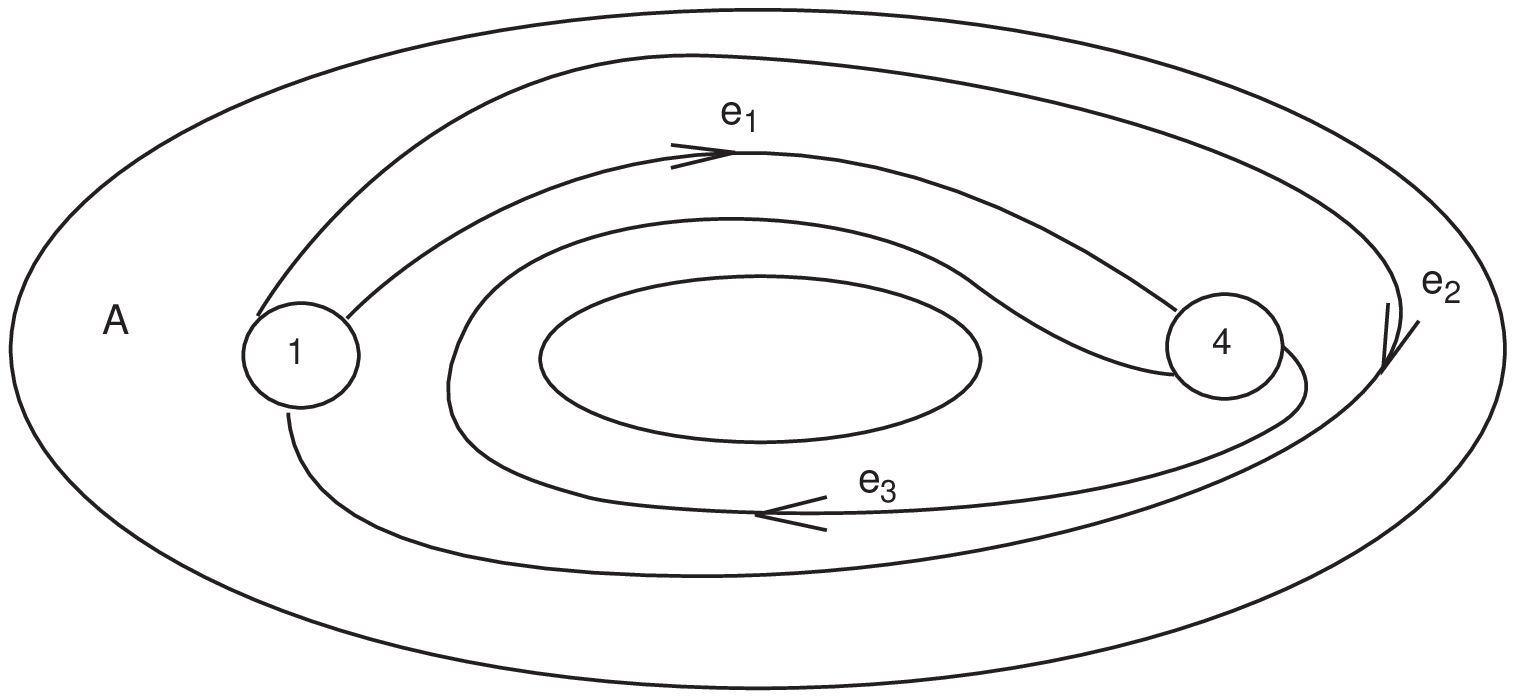}} \caption{ }\label{attachments}
\end{figure}

Hence the relation associated to $D$ is
$$1 = x_1^2 t x_1^{-1} t = x_1^3 (x_1^{-1} t)^2$$
Let $N(A)$ be a collar of $A$ in $\widehat X^+$ and $N(D)$ a tubular neighbourhood of $D$ in $X^+$. Set $Q = N(A) \cup H_{(14)} \cup N(D)$. Then the boundary of $Q$ is a torus and its fundamental group is presented by $\langle x_1, t : x^3 (x^{-1} t)^2\rangle$. It follows that $Q$ is a trefoil complement contained in $\widehat X^+$. Since the latter has a Seifert structure with base orbifold $D^2(2, n)$, $Q$ must be isotopic in $M(\beta)$ to $\widehat X^+$. It follows from the presentation that $t$ normally generates $\pi_1(\widehat X^+)$. Since the slope of $A$ is $\varphi_-$ we have $\widehat X^+(\varphi_-)$ is simply connected, contrary to our observation that it has a non-abelian fundamental group. Thus possibility (2) is also impossible. Therefore $\Delta(\alpha, \beta) \leq 5$ when $m = 4$.
\qed

\section{Extended $S$-cycles in $\Gamma_S$} \label{S-cycles} \label{extended s-cycles}

In this section we examine the implications of the existence of extended and doubly-extended $S$-cycles in $\Gamma_S$ when $t_1^+ = t_1^- = 0$. (See \S \ref{graph}.)

\begin{prop}\label{EScycle}
Suppose that $t_1^+ = t_1^- = 0$ and $\{e_0, e_1, e_2, e_3\}$ is an extended $S$-cycle in $\Gamma_S$ where
$\{e_1,e_2\}$ is an  $S$-cycle. Let $R$ be the bigon face  between $e_1$ and $e_2$ and suppose that $R$ lies on the $\epsilon$-side of $\widehat S$. Then either

$(i)$ $\beta$ is a singular slope so $\Delta(\alpha, \beta) \leq 3$.

$(ii)$ $\epsilon = +$, $X^-$ is a twisted $I$-bundle, and $\widehat X^+$ admits a Seifert structure with base orbifold a disk with two cone points, exactly one of which has order $2$.
\end{prop}

\pf Suppose that the $S$-cycle $\{e_1,e_2\}$ has label pair $\{j, j+1\}$. By Lemma \ref{SC},
$S=F$ is connected and $\dot{\Phi}_1^{\e}$ has a unique component $\phi$ which contains an $\widehat S$-essential annulus. Further, $\phi$ is also $\tau_\epsilon$-invariant and $\widehat X^\epsilon$ admits a Seifert structure with base orbifold a disk with two cone points, at least one of which has order $2$. The proof of Lemma \ref{SC} also shows that a regular neighbourhood $N$ of $b_j \cup e_1^* \cup b_{j+1} \cup e_2^*$ in $S$ is also $\tau_\epsilon$-invariant, at least up to isotopy in $S$, and so there is a M\"{o}bius band $B$ properly embedded in $(\widehat X^\epsilon, N)$. Thus $N$ contains an $\widehat S$-essential annulus with core $\partial B$ which is vertical in $\widehat X^\epsilon$.

Since $\{e_0, e_1, e_2, e_3\}$ is an extended $S$-cycle, $b_j \cup e_1^* \cup b_{j+1} \cup e_2^*$, and therefore $N$, is contained in $\dot{\Phi}_1^{-\epsilon}$. Let $\phi'$ be the component of $\dot{\Phi}_1^{-\epsilon}$ which contains $N$ and $\Sigma'$ the component of $\dot{\Sigma}_1^{-\epsilon}$ which contains $\phi'$. The genera of $\phi$ and $\phi'$ cannot both be $1$ as otherwise, Proposition \ref{sep-seifert} implies that both $\widehat X^+$ and $\widehat X^-$ are twisted $I$-bundles over the Klein bottle, contrary to Corollary \ref{both not (2,2)}.

Suppose first that $\hbox{genus}(\phi) = \hbox{genus}(\phi') = 0$. Then $\Phi_1^- \ne S$, so $X^-$ cannot be a twisted $I$-bundle. In particular, $X^{-\epsilon}$ does not admit a properly embedded non-separating annulus (cf. Lemma \ref{non-sep}). Proposition \ref{sep-seifert}(3) then shows that $\widehat X^-$ admits a Seifert structure with base orbifold a disk with two cone points in which $\partial B \subseteq N \subseteq \phi'$ is vertical. Thus $M(\beta)$ admits a Seifert structure with base orbifold the $2$-sphere with four cone points. Their orders cannot all be $2$ by Corollary \ref{both not (2,2)}. Hence $\beta$ is a singular slope (\cite[Theorem 1.7]{BGZ1}) so $\Delta(\alpha, \beta) \leq 3$ (\cite[Theorem 1.5]{BGZ1}).

Suppose next that $\hbox{genus}(\phi) = 1$ and $\hbox{genus}(\phi') = 0$. Then Corollary \ref{inessential non-tight} implies that $\phi = S$. Thus $X^\epsilon$ is a twisted $I$-bundle, so $\epsilon = -$. If $\Sigma'$ is either a product $I$-bundle which separates $X^{-\epsilon} = X^+$ or a twisted $I$-bundle, then by Proposition \ref{sep-seifert}(3), $\widehat X^+$ admits a Seifert structure with base orbifold a disk with two cone points in which $\partial B \subseteq N \subseteq \phi'$ is vertical. Corollary \ref{both not (2,2)} shows that at least one of the cone points has order larger than $2$. Thus $M(\beta)$ admits a Seifert structure with base orbifold the $2$-sphere with four cone points, at least one of which has order larger than $2$. Thus (i) occurs. If, on the other hand, $\Sigma'$ is a product $I$-bundle which does not separate $X^+$, Proposition \ref{sep-seifert}(3) implies that there is a Seifert structure on $\widehat X^+$ for which $\partial B \subseteq \phi'$ contains a fibre and whose base orbifold is a M\"{o}bius band with at most one cone point. Since $\widehat X^+$ is not a twisted $I$-bundle over the Klein bottle (Corollary \ref{both not (2,2)}), there is exactly one cone point. It follows that $M(\beta)$ admits a Seifert structure with base orbifold a projective plane with three cone points. Thus (i) holds.

Finally suppose that $\hbox{genus}(\phi) = 0$ and $\hbox{genus}(\phi') = 1$. Then Corollary \ref{inessential non-tight} implies that $X^{-\epsilon}$ is a twisted $I$-bundle, so $\epsilon = +$. From above, $\widehat X^+$ admits a Seifert structure with base orbifold a disk with two cone points, at least one of which has order $2$. They cannot both have order $2$ by Corollary \ref{both not (2,2)}. This is case (ii).
\qed

\begin{prop} \label{2SC}
Suppose that $t_1^+ = t_1^- = 0$. If $\Gamma_S$ contains a doubly-extended $S$-cycle, then
$\Delta(\alpha, \beta) \leq 3$.
\end{prop}

\pf Suppose that $\Delta(\alpha, \beta) > 3$. Proposition \ref{EScycle} implies that $X^-$ is a twisted $I$-bundle, the image of the $S$-cycle rectangle is contained in $X^+$, and $\widehat X^+$ admits a Seifert structure with base orbifold a disk with two cone points, exactly one of which has order $2$.

Suppose that the $S$-cycle $\{e_1,e_2\}$ has label pair $\{j, j+1\}$. It follows from the proof of Lemma \ref{SC} that there is a $\tau_+$-invariant regular neighbourhood $N$ of $e_1^* \cup b_j \cup e_{2}^* \cup b_{j+1}$ contained in a $\tau_+$-invariant component $\phi_0$ of $\dot \Phi_1^+$ such that $\widehat \phi_0$ is an $\widehat F$-essential annulus. Further, $\phi_0$ is the unique component of $\dot \Phi_1^+$ to contain an $\widehat F$-essential annulus. Since $\{e_1,e_2\}$ is a doubly-extended $S$-cycle, $\tau_-(N) \subseteq \dot \Phi_1^+$. But $\tau_-(N)$ contains an $\widehat F$-essential annulus, so $\tau_-(N) \subseteq \phi_0$. Since $N$ contains a core of $\widehat \phi_0$, it follows that $\widehat{\tau_-(\phi_0)}$ is isotopic to $\widehat \phi_0$ in $\widehat F$. In particular, $\widehat \phi_0$ is vertical in some Seifert structure on $\widehat X^-$. Thus $M(\beta)$ is Seifert with base orbifold either $P^2(2,n)$ or $S^2(2,2,2,n)$ where $n > 2$. Since $\Delta(\alpha, \beta) > 3$, $\beta$ is not a singular slope (\cite[Theorem 1.5]{BGZ1}), so $M(\beta)$ has base orbifold $P^2(2,n)$ (\cite[Theorem 1.7]{BGZ1}).

Since $N$ is $\tau_+$-invariant and connected, it contains a $\tau_+$-invariant simple closed curve $C$ which is necessarily a core of $\widehat \phi_0$. There is a M\"{o}bius band $B$ properly embedded in $X^+$ with boundary $C$. First suppose that $C \cap \tau_-(C) = \emptyset$. Then there is an annulus $A_-$ properly embedded in $X^-$ with $\partial A_- = C \cup \tau_-(C)$. Since $C$ is vertical in $M(\beta)$, $A_-$ is non-separating in $X^-$. Hence $C \cup \tau_-(C)$ splits $\widehat F$ into two annuli, each containing $m/2$ boundary components of $F$. Let $A_+$ be the properly embedded annulus in $X^+$ which is the frontier of the component of $\dot \Sigma_1^+$ containing $\phi_0$. Since $C$ and $\tau_-(C)$ are disjoint curves in $\phi_0$, each isotopic to a core of $\widehat \phi_0$, $\widehat F$ is the union of four annuli $B_1, B_2, B_3, B_4$ with disjoint interiors such that $\widehat \phi_0 = B_1 \cup B_2 \cup B_3$, $B_4 = \overline{\widehat F \setminus \widehat \phi_0}$, and $\partial B_2 = C \cup \tau_-(C)$. Let $b_j = |B_j \cap \partial F|$. By construction, $b_2 = b_1 + b_3 + b_4 = m/2$. Since $C$ is a $\tau_+$-invariant curve in $\phi_0$, $m/2 \geq b_1 = b_2 + b_3 = m/2 + b_3$. Hence $b_3 = 0$. There are solid tori $V_1, V_2 \subseteq X^+$ where $V_1$ is a regular neighbourhood of $B$ and $V_2$ has boundary $A_+ \cup B_4$. By Lemma \ref{three cases}, $B_4$ has winding number at least $2$ in $V_2$. It follows that a regular neighbourhood of $V_1 \cup A_- \cup B_3 \cup V_2$ in $M$ is Seifert with incompressible boundary (Lemma \ref{incomp}), which is impossible.

Next suppose that $C \cap \tau_-(C) \ne \emptyset$. Since $C \cup \tau_-(C)$ is connected, $\tau_-$-invariant, and contained in $\phi_0$, there is a $\tau_-$-invariant simple closed curve $C'$ in $\phi_0$, necessarily a core of $\widehat \phi_0$. In particular $C'$ is vertical in $\widehat X_+$. It follows that there is a M\"{o}bius band $B'$ properly embedded in $X^-$ with boundary $C'$. Since $C'$ is vertical in the Seifert structure on $\widehat X^-$ with base orbifold $D^2(2,2)$, $M(\beta)$ admits a Seifert structure with base orbifold $S^2(2,2,2,n)$, contrary to our previous deductions. This final contradiction completes the proof.
\qed

\section{Proof of Theorem \ref{main} when $X^-$ is not an $I$-bundle and $t_1^+ = t_1^- = 0$}
\label{not twisted}

Throughout this section we assume
\begin{equation}\label{eq not twisted}
\text{\em $F$ is separating, $\Delta(\alpha, \beta) > 3$, $t_1^+ = t_1^- = 0, X^-$ is not a twisted $I$-bundle, and $m \geq 4$}
\end{equation}
By Proposition \ref{tightness for large j} there is a disk $D_\epsilon \subseteq \widehat F$ containing $\breve \Phi_2^\epsilon$. We choose our base point and the images of the paths $\eta_j$ to lie in $D_\epsilon \cap F$ (cf. \S \ref{section-relation}) when we are interested in a  relation associated to a face lying to the $-\epsilon$-side of $F$.

\subsection{Background results} \label{background 1}

\begin{lemma} \label{not identity permutation version 1}
Suppose that conditions \ref{eq not twisted} hold and $e$ is a negative edge of $\Gamma_S$ whose end labels are the same. Suppose as well that $e$ is a boundary edge of a triangle face $f$ of $\Gamma_S$. Then  the weight of the corresponding edge $\bar e$ in  $\overline{\Gamma}_S$ is at most $2$.
\end{lemma}

\pf Suppose that $f$ lies on the $\epsilon$-side of $F$. Then if the weight of $\bar e$ is at least $3$, the image of $e$ in $F$ is contained in $\widehat{\dot \Phi_2^{-\epsilon}} \subseteq D_{-\epsilon}$. Corollary \ref{triangle not tight} then shows the labels at the ends of $j$ are different.
\qed

Proposition \ref{Scycle} combines with the fact that $\Gamma_S$ contains no extended $S$-cycles (Proposition \ref{EScycle}) to imply the following lemma.

\begin{lemma} \label{weight of positive version 1}
Suppose that conditions \ref{eq not twisted} hold. Then the weight of a positive edge of $\overline{\Gamma}_S$ is at most $\frac{m}{2} + 2$. In particular, its weight is less than $m$ if $m \geq 6$ and less than or equal to $m$ if $m = 4$.
\qed
\end{lemma}

\begin{lemma} \label{no more than m and 2 version 1}
Suppose that conditions \ref{eq not twisted} hold and that $\overline{\Gamma}_S$ has a triangle face $f$ with edges $\bar e, \bar e'$ where $wt(\bar e') > 2$. Then $wt(\bar e) \leq m$. Further, if $\bar e$ is negative of weight $m$, the permutation associated to the corresponding family of edges has order $\frac{m}{2}$.
\end{lemma}

\pf Let $v$ be the common vertex of $\bar e$ and $\bar e'$, and let
$v'$ be the other vertex of $\bar e$. Let $e'$ be the lead edge of $\bar e'$
incident to $f$. Suppose that $f$ lies on the $\epsilon$-side of $F$.

Suppose otherwise that $wt(\bar e) > m$ and let $e_1, e_2,..., e_m, e_{m+1}$ be the $m+1$ consecutive edges in $\bar e$-family with $e_1$ as the lead edge incident to $f$. We may assume that the labels of $e_{1}, e_{2},..., e_{m}, e_{m+1}$ at $v$ are $1, 2,..., m, 1$ respectively. So the label of $e'$ at $v$ is $m$.

Lemma \ref{weight of positive version 1} implies that $\bar e$ is a negative edge so the parity rule implies that the labels of $e_{1}, e_{2},..., e_{m}, e_{m+1}$ at $v'$ are $1+2k, 2+2k,..., m, 1, 2, ..., 2k, 1+2k$ respectively, for some $0< k< m/2$ (Lemma \ref{not identity permutation version 1}).

As both $e_{1}$ and $e'$ are contained in $\widehat{\dot \Phi_2^{-\epsilon}}$ and $f$ is on the $\epsilon$-side of $F$, $f$ gives the relation
\begin{equation}\label{eq 2k m}
\text{\em $x_{2k} x_m^{-1} x_j \in \pi_1(\widehat F)$}
\end{equation}
for some $j$. Let $B_i$ be the bigon face between $e_{i}$ and $e_{i+1}$ for $i=1,..., m$. Note that $B_i$ is on the $\epsilon$-side of $F$ if and only if $i$ is even. Also note that for each even $i$ with $2<i<m$, the images in $F$ of the two edges of $B_i$ both lie in $\widehat{\dot \Phi_2^{-\epsilon}}$. Also $e_3^*$ is  contained in $\widehat{\dot \Phi_2^{-\epsilon}}$. So  for each  $2 < i = 2p < m$, $B_i$ gives the relation $x_{2p} x^{-1}_{2p+2k} =1$, so
\begin{equation}\label{eq 2k m j}
\text{\em $x_{2p} = x_{2p+2k} \; \hbox{ for } \; 2 < 2p < m$}
\end{equation}
Similarly $B_2$ gives the relation
\begin{equation}\label{eq 2 2+2k}
\text{\em $x_2 x_{2+2k}^{-1} = u \in \pi_1(\widehat F)$}
\end{equation}
Now consider the permutation given by the first $m$ edges $e_1,...,e_m$. The orbit of the label $2k$ is $\{2k, 4k , 6k, \ldots, m\}$ where we consider the labels (mod $m$).  Applying \ref{eq 2k m j} successively shows that if $2$ is not in this orbit (i.e. the permutation has order less than $m/2$), then
$$x_{2k} = x_{4k} = \ldots = x_{m}$$
Thus $x_{2k} x_{m}^{-1} = 1$. But comparing with \ref{eq 2k m} shows that $x_j \in P$, which contradicts Proposition \ref{not peripheral}. On the other hand, if $2$ is in this orbit then by \ref{eq 2 2+2k},
$$x_{2k} = x_{4k} = \ldots = x_2 = u x_{2k +2} = u x_{4k + 2} = \ldots = u x_{m}$$
Thus $x_{2k} x_{m}^{-1} = u \in \pi_1(\widehat F)$ which combines with \ref{eq 2k m} to yield a similar contradiction. This proves the first assertion of the lemma.

Next suppose that $\bar e$ is negative of weight $m$ and let $e_1, e_2,..., e_m$ be the $m$ consecutive edges in $\bar e$-family with $e_1$ as the lead edge incident to $f$. As above we take the labels of $e_{1}, e_{2},..., e_{m}$ at $v$ to be  $1, 2,..., m$ respectively and those at $v'$ to be $1+2k, 2+2k,..., m, 1, 2, ..., 2k$ respectively, for some $0 < k < m/2$ (Lemma \ref{not identity permutation version 1}). Similar to identity \ref{eq 2k m j} we have $x_{2p} = x_{2p+2k}$ for $2 < 2p < m-2$ and $x_2 x_{2+2k}^{-1} = u \in \pi_1(\widehat F)$. If the permutation $j \mapsto j + 2k$ (mod $m$) does not have order $\frac{m}{2}$ then neither $2$ nor $m-2$ lie in the orbit $\{2k, 4k , 6k, \ldots, m\}$ of the label $2k$. Thus $x_{2k} = x_{4k} = \ldots = x_{m}$ so plugging $x_{2k} x_{m}^{-1} = 1$ into \ref{eq 2k m} yields the contradiction $x_j \in \pi_1(\widehat F)$. This completes the proof of the lemma.
\qed

\begin{lemma} \label{hexagonal version 1}
Suppose that conditions \ref{eq not twisted} hold. If $\Delta(\alpha, \beta) > 5$, then $\overline{\Gamma}_S$ is hexagonal.
\end{lemma}

\pf By Proposition \ref{euler}, it suffices to show that there is no vertex of valency $5$ or less.

Suppose otherwise that $v$ is a vertex of $\overline{\Gamma}_S$ of valency $5$ or less. Since $m \Delta(\alpha, \beta) / \hbox{valency}_{\overline{\Gamma}_S}(v) \leq m + 1$ (Proposition \ref{tight bound}), we have
\begin{equation}\label{eq delta 1}
\text{\em $6 \leq \Delta(\alpha, \beta) \leq \hbox{valency}_{\overline{\Gamma}_S}(v) + \frac{\hbox{valency}_{\overline{\Gamma}_S}(v)}{m}$}
\end{equation}
Hence $\hbox{valency}_{\overline{\Gamma}_S}(v) = 5, m = 4$, and $\Delta(\alpha, \beta) = 6$. It follows that the weights of the edges incident to $v$ are $4, 5, 5, 5, 5$. Lemma \ref{no more than m and 2 version 1} implies that there can be no triangle faces incident to $v$. In other words, $\varphi_3(v) = 0$. Then
$$\mu(v) = \varphi_2(v) = m \Delta(\alpha, \beta) - \hbox{valency}_{\overline{\Gamma}_S}(v) = m \Delta(\alpha, \beta) - 5$$
Hence by Corollary \ref{mu constant} there is a vertex $v_0$ of $\overline{\Gamma}_S$ with $\mu(v_0) > m \Delta(\alpha, \beta) - 4$. Then Proposition \ref{possible values for mu} shows that $v_0$ has valency $5$ or less.
As above we have $\hbox{valency}_{\overline{\Gamma}_S}(v_0) = 5$ and the weights of the edges incident to $v_0$ are $4, 5, 5, 5, 5$. By Corollary \ref{possible values for mu}, $\varphi_3(v_0) \geq 4$, so in particular there is a triangle face of $\overline{\Gamma}_S$ with two edges of weight $5$, which is impossible by Lemma \ref{no more than m and 2 version 1}. Thus there is no vertex $v$ of $\overline{\Gamma}_S$ of valency $5$ or less, so the lemma holds.
\qed

\subsection{Proof} \label{proof is not}
We prove Theorem \ref{main} under conditions \ref{eq not twisted}.

Assume that $\Delta(\alpha, \beta) > 5$ in order to derive a contradiction. Recall that $\bar \Gamma_S$ is hexagonal by Lemma \ref{hexagonal version 1}. In particular $Y$ is a torus.

Since $X^-$ is not a twisted $I$-bundle, Proposition \ref{sep-seifert} implies that for each $\epsilon$, $\dot \Sigma_1^\epsilon$ has a unique component and $\breve \Phi_1^\epsilon$ is the union of at most two components, each an $\widehat F$-essential annulus.

Suppose that there is an edge $\bar e$ of weight $m + 1$ incident to a vertex $v$ of $\overline{\Gamma}_S$. Since $\overline{\Gamma}_S$ is hexagonal, Lemma \ref{no more than m and 2 version 1} implies that the two edges of $\overline{\Gamma}_S$ incident to $v$ which are adjacent to $\bar e$ have weights at most $2$. Then                                                                                                                                                                                                                                                                                                                                                                                                                                                                                                                                                                                                                                                                                                                                                                                                                                                                                                                                                                                                                                                                                                                                                                                                                                                                                                                                                                                                                                                                                                                                                                                                                                                                                                                                                       the sum of the weights of the six edges incident to $v$ is at most $4m + 5$. On the other hand, this is $m \Delta(\alpha, \beta) \geq 6m$. Hence $6m \leq 4m + 5$, which is impossible for $m \geq 4$. Thus the weight of each edge in $\overline{\Gamma}_S$ is at most $m$. But then $6m \leq m \Delta(\alpha, \beta) \leq 6m$, so each edge of $\overline{\Gamma}_S$ has weight $m$ and $\Delta(\alpha, \beta) = 6$.

As $\overline{\Gamma}_S$ is hexagonal, it has positive edges. Then Lemma \ref{weight of positive version 1} implies that $m \leq \frac{m}{2} + 2$. Thus $m = 4$ and the weight of any edge in $\overline{\Gamma}_S$ is $4$. Proposition \ref{plus-minus vertices} implies that there is a triangle face $f$ with one positive edge $\bar e_1$ and two negative edges $\bar e_2, \bar e_3$. Let $v_1, v_2, v_3$ be its vertices where $v_1$ is determined by $\bar e_1$ and $\bar e_2$, $v_2$ is determined by $\bar e_2$ and $\bar e_3$, and $v_3$ is determined by $\bar e_2$ and $\bar e_3$. Let $e_1, e_2, e_3$ denote the lead edges of $\bar e_1, \bar e_2, \bar e_3$ at $f$. Without loss of generality we can take the label of $e_1$ at $v_1$ to be $1$ and that of $e_2$ to be $4$. Lemma \ref{not identity permutation version 1} shows that $e_2$ has label $2$ at $v_2$, so $e_3$ has label $3$ there. Lemma \ref{not identity permutation version 1} then shows that the label of $e_3$ at $v_3$ is $1$, so the label of $e_1$ at $v_3$ is $4$. But then the four edges of $\Gamma_S$ parallel to $\bar e_1$ form an extended $S$-cycle, which contradicts Proposition \ref{EScycle}. This final contradiction completes the proof of Theorem \ref{main}  when $X^-$ is not a twisted $I$-bundle.
\qed

\section{Proof of Theorem \ref{main} when $X^-$ is a twisted $I$-bundle and $t_1^+ = 0$}

We assume throughout this section that
\begin{equation} \label{eq background 2}
\text{\em $F$ is separating, $\Delta(\alpha, \beta) > 3$, $t_1^+ = 0, X^-$ is a twisted $I$-bundle, and $m \geq 4$}
\end{equation}
Note that $t_1^- = 0 $ when $X^-$ is a twisted $I$-bundle.

\subsection{Background results}
\label{background 2}

As $X^+$ is not a twisted $I$-bundle, Proposition \ref{sep-seifert} implies that $\dot \Sigma_1^+$ has a unique component and $\breve \Phi_1^+ = \dot \Phi_1^+$ is the union of one or two components each of which completes to an $\widehat F$-essential annulus. Let $\phi_+$ be the slope on $\widehat F$ of these annuli and note that it is the slope of the Seifert structure on $\widehat X^+$ (Proposition \ref{sep-seifert}). Set $\alpha_- = \widehat \tau_-(\phi_+)$, the slope on $\widehat F$ determined by $\dot \Phi_3^- = \tau_-(\dot \Phi_1^+)$. As $\widehat \tau_-$ is a fixed-point free orientation reversing involution, $\Delta(\phi_+, \alpha_-)$ is even.

For the rest of this section we take $A_- = \dot \Phi_3^- = \dot \Phi_2^- = \tau_-(\dot \Phi_1^+) \subset F$. Also we take a disk $D$ in $\widehat{A_-}$ containing all $\widehat b_j$,  choose the paths $\eta_j$ (defined in Section \ref{section-relation}) in $D$, and define the elements $x_j$ of $\pi_1(\widehat X^+)$ as in  Section \ref{section-relation}. It follows that if $e$ is an edge of a face $f$ of $\Gamma_S$ lying on the $+$-side of $F$ and the image of $e$ in $F$ lies in $\dot \Phi_2^-$, then the associated element of $\pi_1(\widehat F)$ determined by $e$ is a power of $t$, the element determined by a core of $\widehat A_-$.

\begin{lemma}  \label{not abelian}
Suppose that conditions \ref{eq background 2} hold. Then $\pi_1(\widehat X^+(\alpha_-))$ is not abelian.
\end{lemma}

\pf The base orbifold of $\widehat X^+$ has the form $D^2(p,q)$ for some $p, q \geq 2$. Since $\Delta(\phi_+, \alpha_-)$ is even, it can't be $1$. Thus $\widehat X^+(\alpha_-)$ is either $L_p \# L_q$ or is Seifert fibred with base orbifold $S^2(p,q, \Delta(\phi_+, \alpha_-))$ having three cone points. In either case, its fundamental group is not abelian.
\qed

For each $y \in \pi_1(\widehat X^+)$ we use $\bar y$ to denote its image in $\pi_1(\widehat X^+(\alpha_-))$.

\begin{defin} \label{P}
{\rm Define $P \leq \pi_1(\widehat X^+(\alpha_-))$ to be the subgroup generated by $\pi_1(\widehat F)$.}
\end{defin}

\begin{lemma} \label{not in P}
Suppose that conditions \ref{eq background 2} hold. For no $j$ is the image of $x_j$ in $\pi_1(\widehat X^+(\alpha_-))$ contained in $P$.
\end{lemma}

\pf We noted in \S \ref{section-relation} that there are generators $a, b$ of $\pi_1(D^2(p,q)) \cong \mathbb Z/p * \mathbb Z/q$ such that $ab$ generates its peripheral subgroup and the image of each $x_j$ in $\pi_1(D^2(p,q))$ is of the form $(ab)^{r_j} a^{\epsilon_j} (ab)^{s_j}$ where $r_j, s_j \in \mathbb Z$ and $\epsilon_j \in \{\pm 1\}$. Thus if the image of some $x_j$ in $\pi_1(\widehat X^+(\alpha_-))$ is contained in $P$, the image of each $x_j$ in $\pi_1(\widehat X^+(\alpha_-))$ is contained in $P$, contrary to Lemma \ref{not abelian}.
\qed

\begin{lemma} \label{even on negative side}
Suppose that conditions \ref{eq background 2} hold. Each face of $\Gamma_S$ which lies on the $-$-side of $F$ has an even number of edges. In particular, each triangle face lies on the $+$-side of $F$.
\end{lemma}

\pf The boundary of the face intersects the Klein bottle core of $\widehat X^-$ transversely in $k$ points where $k$ is the number of edges of the face. Since this curve is homologically trivial, $k$ is even.
\qed

Given this lemma, the fact that $\Gamma_S$ contains no doubly-extended $S$-cycles (Proposition \ref{2SC}), and the fact that the $S$-cycle bigon in an extended $S$-cycles lies to the $+$-side of $F$ (Proposition \ref{EScycle}) we deduce the following lemma.

\begin{lemma} \label{weight of positive version 2}
Suppose that conditions \ref{eq background 2} hold. The weight of a positive edge of $\overline{\Gamma}_S$ is at most $\frac{m}{2} + 3$ if $m$ is not divisible by $4$ and $\frac{m}{2} + 4$ otherwise. In particular, its weight is less than $m$ if $m \geq 10$ and less than or equal to $m$ if $m \geq 6$.
\qed
\end{lemma}

\begin{lemma} \label{not identity permutation version 2}
Suppose that conditions \ref{eq background 2} hold. Suppose that $e$ is a negative edge of $\Gamma_S$ whose end labels are the same. Suppose as well that $e$ is a boundary edge of a triangle face $f$ of $\Gamma_S$. Then  the weight of the corresponding edge $\bar e$ in $\overline{\Gamma}_S$ is at most $2$.
\end{lemma}

\pf Suppose otherwise that the weight of $\bar e$ is at least $3$. Then $e^*$ is contained in $\widehat{\dot \Phi_2^-}$. Let $j$ be the label of $e$ at its two end
points. The triangle face $f$ is on the $+$-side of $F$ and the associated relation is
$$x_{j}^{-1} t^a x_{j} w_2 x_{k} w_3 = 1$$
in $\pi_1(\widehat X^+)$, where $t$ is the class of a loop in the annulus
$\widehat{\dot \Phi_2^-}$ corresponding to the slope $\a_-$, and $w_1, w_2 \in \pi_1(\widehat F)$. Hence the relation implies that the image of $x_k$ is contained in $P \leq \pi_1(\widehat X^+(\alpha_-))$, contrary to Lemma \ref{not in P}. Thus the lemma holds.
\qed

\begin{lemma} \label{no more than m and 2 version 2}
Suppose that conditions \ref{eq background 2} hold and that $\overline{\Gamma}_S$ has a triangle face $f$ with edges $\bar e, \bar e'$ where $wt(\bar e') > 2$.

$(1)$ If $\bar e$ is negative, then $wt(\bar e) \leq m$. Further, if $wt(\bar e) = m$, then the permutation associated to the $\bar e$-family of parallel edges in $\Gamma_S$ has order $\frac{m}{2}$.

$(2)$ If $\bar e$ is positive, then $wt(\bar e) \leq m-1$ for $m > 6$ and $wt(\bar e) \leq m$ for $m = 4, 6$.

\end{lemma}

\pf Let $v$ be the common vertex of $\bar e$ and $\bar e'$, and let
$v'$ be the other vertex of $\bar e$. Let $e'$ be the lead edge of $\bar e'$
incident to $f$.

The proof of assertion (1) mirrors that of Lemma \ref{no more than m and 2 version 1}. The only difference is that we work with the images $\bar x_j \in \pi_1(\widehat X^+(\alpha_-))$ rather than the $x_j \in \pi_1(\widehat X^+)$ and replace the contradiction to Proposition \ref{not peripheral} with one to Lemma \ref{not in P}.

Next we consider assertion (2). Suppose that $\bar e$ is a positive edge of $\overline{\Gamma}_S$ with $wt(\bar e) > m-1$. First we show that $m \leq 6$.

Let $e_1, e_2,..., e_m$ be the $m$ consecutive edges in the $\bar e$-family with $e_1$ as the lead edge incident to $f$. We may assume that the labels of $e_{1}, e_{2},..., e_{m}$ at $v$ are $1, 2,..., m$ respectively. The label of $e'$ at $v$ is then $m$.

By the parity rule, the labels of $e_{1}, e_{2},..., e_{m}$ at $v'$ are $2k, 2k-1,...,1, m, m-1,..., 2k+1$ respectively, for some $0 < k \leq m/2$. So the triangle face $f$  gives the relation
\begin{equation}\label{eq 2k+1 1 j}
\text{\em $\bar x_{2k+1}\bar  x_1  \bar x_j \in P$}
\end{equation}

Again let $B_i$ be the bigon face between $e_{i}$ and $e_{i+1}$ for $i=1,..., m-1$.
If $2<2k<m$, then the image of $e_{2k}$ in $F$ is contained in
$\widehat{\dot \Phi_2^-}$ and so the bigon face $B_{2k}$ gives the relation
\begin{equation}\label{eq 2k+1 1}
\text{\em $\bar x_{2k+1}\bar x_{1} \in P$}
\end{equation}
Equations \ref{eq 2k+1 1 j} and \ref{eq 2k+1 1} imply that $\bar x_j$ is peripheral, a contradiction.
 Thus $2k=2$ or $2k=m$.

If $m > 6$ and $2k=2$, then $\{e_{m/2+1}, e_{m/2+2}\}$ is a doubly extended $S$-cycle, giving a contradiction.

If $m > 4$ and $2k=m$, then $\{e_{m/2}, e_{m/2+1}\}$ is a doubly extended $S$-cycle, giving a contradiction again.

If $m = 6$, $wt(\bar e) > 6$, and $2k = 2$, then $\{e_{4}, e_{5}\}$ is a doubly extended $S$-cycle, giving a contradiction.

Finally suppose $m = 4$ and $wt(\bar e) > m = 4$. The label of $e$ at $v'$ is either $2$ or $4$. If it is $2$, $\{e_{3}, e_{4}\}$ is an extended $S$-cycle lying on the $-$-side of $F$, giving a contradiction. If it is $4$, then the face $f$ shows $\bar x_{1}^2 \bar x_j \in P$ for some $j$ while the bigon between $e_4$ and $e_5$ gives $\bar x_{1}^2 \in P$. But then $\bar x_j \in P$, which contradicts Lemma \ref{not in P}.
\qed

\subsection{Proof when $\dot \Phi_3^+$ is a union of tight components} \label{proof all tight}

The hypothesis $\dot \Phi_3^+$ is a union of tight components implies that the edges of $\overline{\Gamma}_S$ have weight bounded above by $m+2$ (Corollary \ref{max weight of edge}). Thus, as in the proof of Lemma \ref{hexagonal version 1} we have
\begin{equation}\label{eq proof all tight}
\text{\em $6 \leq \Delta(\alpha, \beta) \leq \hbox{valency}_{\overline{\Gamma}_S}(v) + 2 \Big(\frac{\hbox{valency}_{\overline{\Gamma}_S}(v)}{m}\Big)$}
\end{equation}

\begin{lemma} \label{options for mu}
Suppose that conditions \ref{eq background 2} hold and $\dot \Phi_3^+$ is a union of tight components. If $\Delta(\alpha, \beta) > 5$, then $\mu(v) = m \Delta(\alpha, \beta) -  4$ for all vertices $v$ of $\overline{\Gamma}_S$.
\end{lemma}

\pf Assume that there is some vertex $v$ with $\mu(v) > m \Delta(\alpha, \beta) -  4$. Proposition \ref{possible values for mu} implies that $\hbox{valency}_{\overline{\Gamma}_S}(v)$ is at most $5$ while inequality \ref{eq proof all tight} shows that $\hbox{valency}_{\overline{\Gamma}_S}(v)$ is at least $4$ and if it is $4$, then $m = 4, \Delta(\alpha, \beta) = 6$, and each edge incident to $v$ has weight $6$. Since Proposition \ref{possible values for mu} implies $\varphi_3(v) \geq 1$, Lemma \ref{no more than m and 2 version 2} shows that this case is impossible. Assume then that $\hbox{valency}_{\overline{\Gamma}_S}(v) = 5$. Proposition \ref{possible values for mu} implies $\varphi_3(v) \geq 4$ and therefore the sum of the weights of the edges incident to $v$ is at most $5m + 2$ by Lemma \ref{no more than m and 2 version 2}. But this sum is bounded below by $\Delta(\alpha, \beta)m \geq 6m$, which is impossible for $m \geq 4$. Corollary \ref{mu constant} then implies the desired conclusion.
\qed

\begin{lemma} \label{even weight}
Suppose that conditions \ref{eq background 2} hold and $\dot \Phi_3^+$ is a union of tight components. If $\Delta(\alpha, \beta) > 5$, then $\Delta(\alpha, \beta) = 6$. Further, one of the following two situations occurs.

$(i)$ $\overline{\Gamma}_S$ is hexagonal and its edges have weight $m$.

$(ii)$ $m = 4$, $\overline{\Gamma}_S$ is rectangular and its edges have weight $6$.

\end{lemma}

\pf Since $\mu(v) = m \Delta(\alpha, \beta) - 4$ for each vertex $v$ of $\overline{\Gamma}_S$ (Lemma \ref{options for mu}), Proposition \ref{possible values for mu} implies that the valency of each vertex of $\overline{\Gamma}_S$ is at most $6$. First we show that the vertices of $\overline{\Gamma}_S$ have valency $4$ or $6$.

Let $v$ be a vertex of $\overline{\Gamma}_S$. Inequality \ref{eq proof all tight} shows that $\hbox{valency}_{\overline{\Gamma}_S}(v) \geq 4$. Suppose that $\hbox{valency}_{\overline{\Gamma}_S}(v) = 5$. By Proposition \ref{possible values for mu}, $\varphi_3(v) = 3$. Then Lemma \ref{no more than m and 2 version 2} implies that the sum of the weights of the edges incident to $v$ is at most $5m + 2$. As this sum is the valency of $v$ in $\Gamma_S$, we have $6m \leq \Delta(\alpha, \beta) \leq 5m + 2$, which is impossible since $m \geq 4$. Thus no vertex of $\overline{\Gamma}_S$ has valency $5$, so each vertex $v$ either has valency $4$ and $\varphi_3(v) = 0$ or valency $6$ and $\varphi_3(v) = 6$ (Proposition \ref{possible values for mu}). In particular, no edge connects a vertex of valency $4$ with one of valency $6$. It follows that the union of the open star neighbourhoods of the vertices of valency $6$ equals the union of the closed star neighbourhoods of these vertices. Thus this union is either $\widehat F$ or empty. It follows that either each vertex of $\overline{\Gamma}_S$ has valency $6$, so $\overline{\Gamma}_S$ is hexagonal (Proposition \ref{euler}), or each has valency $4$. In the latter case there are no triangle faces so $\overline{\Gamma}_S$ is rectangular by Proposition \ref{euler}.

Suppose that $\overline{\Gamma}_S$ is rectangular. Then inequality \ref{eq proof all tight} shows that $m = 4, \Delta(\alpha, \beta) = 6$, \ and therefore each edge of $\overline{\Gamma}_S$ has weight $6$. This is case (ii) of the lemma.

Suppose next that $\overline{\Gamma}_S$ is hexagonal. As each of its faces is a triangle, they lie on the $+$-side of $F$ (Lemma \ref{even on negative side}). Hence each edge of $\overline{\Gamma}_S$ has even weight. Suppose that some such edge $\bar e$ has weight $m + 2$. Lemma \ref{no more than m and 2 version 2} implies that if $f $ is a face of $\overline{\Gamma}_S$ incident to $\bar e$, then each of the two edges of $\partial f \setminus \bar e$ has weight $2$. Thus there is a vertex of $\overline{\Gamma}_S$ having successive edges of weight $2$ incident to it. But then the remaining four edges have weights adding to at least $m \Delta(\alpha, \beta) - 4 \geq 6m - 4$. On the other hand, Proposition \ref{possible values for mu} shows that the maximal weights of these four edges are either $m, m, m, m$, or $2, m, m, m+2$, or $2, 2, m+2, m+2$. Each possibility implies that $m < 4$. Thus each edge of $\overline{\Gamma}_S$ has weight $m$ or less. Then $6 m \leq m \Delta(\alpha, \beta) \leq 6m$. It follows that each edge of $\overline{\Gamma}_S$ has weight $m$ and $\Delta(\alpha, \beta) = 6$. This is case (i).
\qed

\begin{lemma} \label{m+2}
Suppose that conditions \ref{eq background 2} hold and $\dot \Phi_3^+$ is a union of tight components. If $\Delta(\alpha, \beta) = 6$, then $m = 4$.
\end{lemma}

\pf By Lemma \ref{even weight} we can suppose that $\overline{\Gamma}_S$ is hexagonal and each of its edges has weight $m$. Proposition \ref{euler} implies that it has positive edges. Thus $m \leq \frac{m}{2} + 4$ (Lemma \ref{weight of positive version 2}), so $m \leq 8$.

There are negative edges in $\overline{\Gamma}_S$ (Proposition \ref{plus-minus vertices}) so we can choose a triangle face $f$ of $\overline{\Gamma}_S$ with edges $\bar e_1, \bar e_2, \bar e_3$ where $\bar e_1$ is positive and $\bar e_2, \bar e_3$ are negative. Let $e_1, e_2, e_3$ be the edges of $\Gamma_S$ incident to $f$ and contained, respectively, in $\bar e_1, \bar e_2, \bar e_3$. Let $v_1$ be the vertex of $\Gamma_S$ determined by $e_1$ and $e_2$, $v_2$ that determined by $e_1$ and $e_3$, and $v_3$ that determined by $e_2$ and $e_3$. We can suppose that $e_1$ has label $1$ at the vertex $v_1$ and $e_2$ has label $m$ there.

Suppose $m = 8$. Since there are no doubly-extended $S$-cycles in $\Gamma_S$, the permutation associated to any positive edge is of the form $i \mapsto 5 - i$ (mod $8$). As $\bar e_1$ is positive, $e_1$ has label $4$ at $v_2$, so $e_3$ has label $5$ there. Then $f$ yields the relation $\bar x_5 \bar x_1 \bar x_j = 1$.
But the fourth bigon from $f$ in the $\bar e_1$ family of edges implies that $\bar x_5 \bar x_1 = 1$. Thus $\bar x_j = 1$, which is impossible. Thus $m \ne 8$.

Suppose then that $m = 6$. As $\bar e_1$ is positive and $\Gamma_S$ has no doubly-extended $S$-cycles, $e_1$ has label $2$ or $4$ at $v_2$. We will deal with the first case as the second is similar. Thus $e_1$ has label $2$ at $v_2$ so $e_3$ has label $3$ there. The label of $e_2$ at $v_3$ cannot be $6$ by Lemma \ref{not identity permutation version 2} and the same lemma shows that it cannot be $2$ as otherwise the label of $e_3$ at $v_3$ would be $3$. Hence this label must be $4$. Examination of the labels of the $\Gamma_S$-edges in $\bar e_2, \bar e_3$ at $v_3$ shows that $b_1 \cup b_3 \cup b_5$ and $b_2 \cup b_4 \cup b_8$ lie in components of $\dot \Phi_5^-$. But consideration of the lead edge of $\bar e_1$ at $f$ shows that $b_1 \cup b_2$ lie in the same component of $\dot \Phi_5^-$. Thus $\dot \Phi_5^- = \tau_-(\dot \Phi_3^+)$ is connected, contrary to Proposition \ref{tightness for large j}. Thus $m \ne 6$, which completes the proof of the lemma.
\qed

The previous two lemmas reduce the proof of Theorem \ref{main} under assumption \ref{eq background 2} to the cases described in the following two subsections.

\subsubsection{The case $m = 4$, $\Delta(\alpha, \beta) = 6$ and $\overline{\Gamma}_S$ hexagonal with edges of weight $4$} \label{hexagonal case}

We consider singular disks $D$ in $X^+$, with $D\cap \partial X^+ =\partial D$.
We can assume the components of $\partial F$ are labeled so that
$\partial X^+ = F\cup A_{23} \cup A_{41}$, where $A_{23}$ and
$A_{41}$ are annuli running between boundary components 2,3 and boundary
components 4,1 of $F$, respectively.
By a homotopy we may assume that $\partial D$ meets each of $A_{23}$ and
$A_{41}$ in a finite disjoint union of essential embedded arcs.
We will refer to these arcs as the {\em corners\/} of $D$.
More precisely, if we go around $\partial D$ in some direction we get
a cyclic sequence of $X_2^{\pm1}$ and {\em $X_4^{\pm1}$-corners}, where
$X_2,X_2^{-1}$ indicate that $\partial D$ is running across $A_{23}$ from
2 to 3 or from 3 to 2, respectively, and $X_4,X_4^{-1}$ indicate that
$\partial D$ is running across $A_{41}$ from 4 to 1 or 1 to 4, respectively.
In this way $D$ determines a cyclic word $W = W(X_2^{\pm 1},X_4^{\pm1})$,
well-defined up to inversion, and we say that $D$ is {\em of type\/} $W$.
(Thus $D$ is of type $W$ if and only if it is of type $W^{-1}$.)
We emphasize that $W$ is an unreduced word; for example $X_2$ and $X_2X_4X_4^{-1}$
are distinct.

Let $\widehat A_-\subseteq \widehat F$ be the annulus defined at the beginning of \S \ref{background 2}.
If $\partial D\cap F\subseteq A_-$ we say that $D$ is an {\em $A_-$-disk\/}.

Recall the elements $x_2,x_4$ of $\pi_1(\widehat X^+)$ as defined at the beginning of \S \ref{background 2}. We use $x_j$, respectively $\bar x_j$, to denote the image of $x_j$ in $\pi_1(\widehat X^+)$, respectively $\pi_1(\widehat X^+(\alpha_-))$. Clearly, if $D$ is an $A_-$-disk of type $W(X_2^{\pm1},X_4^{\pm1})$ then $D$
gives the relation $W(\bar x_2^{\pm1}, \bar x_4^{\pm1})$ in $\pi_1(\widehat X^+(\alpha_-))$.

Note that a triangle face of $\Gamma_S$ defines an $A_-$-disk.
Note also that there is a one-one correspondence between triangle faces of
$\Gamma_S$ and faces of the reduced graph $\overline{\Gamma}_S$.
We therefore say that a face of $\overline{\Gamma}_S$ has type $W$ if and only
if the corresponding triangle face of $\Gamma_S$ has type $W$.

Let $v$ be a vertex of $\Gamma_S$.
An endpoint at $v$ of an edge of the reduced graph $\overline{\Gamma}_S$
corresponds to four endpoints of edges of $\Gamma_S$, and we can assume that the label sequence
(reading around $v$ anticlockwise if $v$ is positive and clockwise if $v$
is negative) is either ${3\ 4\ 1\ 2}$ or ${1\ 2\ 3\ 4}$.
We say  that $v$ is of {\em type I} or {\em II}, respectively.
If $v$ is a positive vertex of type~I we will say $v$ is a
$(+,I)$ {\em vertex}, and so on.

\begin{figure}[!ht]
\centerline{\includegraphics{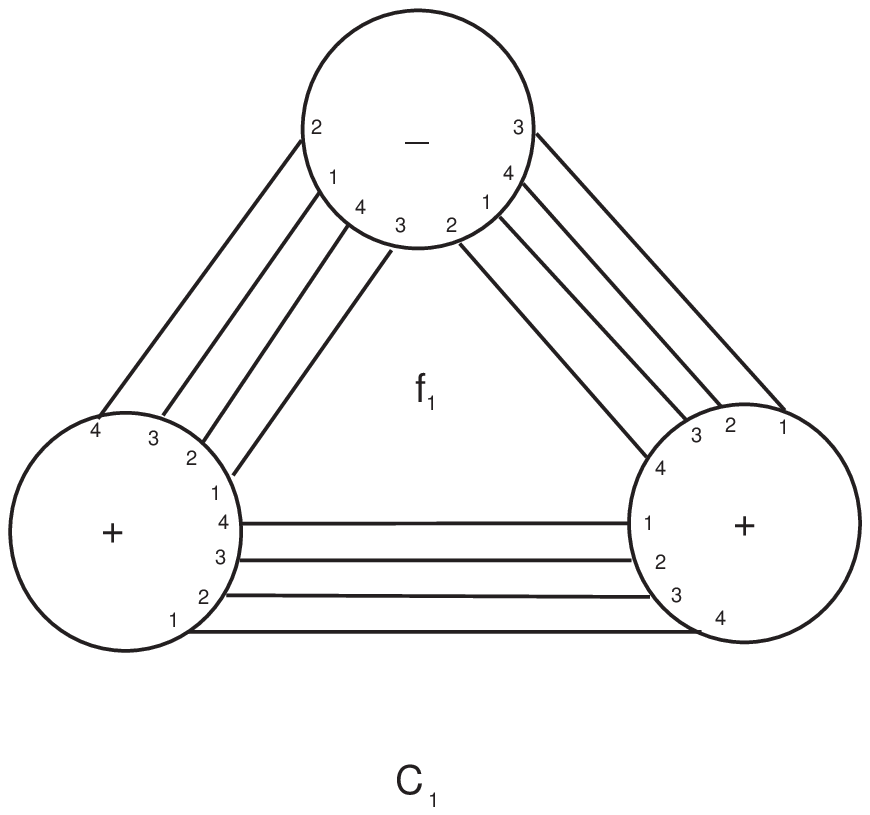}} \caption{ }\label{C1}
\end{figure}

\begin{figure}[!ht]
\centerline{\includegraphics{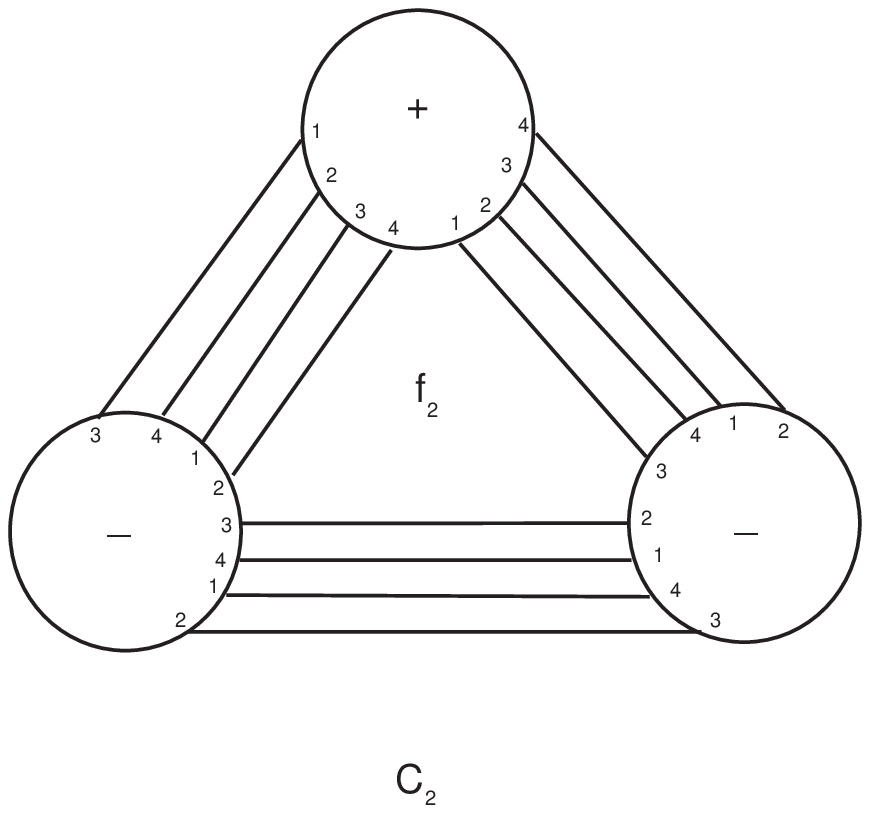}} \caption{ }\label{C2}
\end{figure}

Lemma \ref{no more than m and 2 version 2}(1) implies

\begin{lemma}\label{lem1}
No edge of $\overline{\Gamma}_S$ connects vertices of the same type and
opposite sign.
\end{lemma}

By Proposition~\ref{plus-minus vertices}
$\Gamma_S$ has the same number of positive and negative vertices.
In particular, $\overline{\Gamma}_S$ has a face $\bar{f}_1$ in which not all
vertices have the same sign; without loss of generality we may assume
that two of the vertices are positive and one negative, and that the
negative vertex is of type~I.
It follows from Lemma~\ref{lem1} that the two positive vertices are of type~II.
Thus the face $\bar{f}_1$ has type $X_2^{-1}X_4^2$.
The configuration $C1$ of $\Gamma_S$ corresponding to $\bar{f}_1$ is
shown in Figure~\ref{C1}.

\begin{lemma}\label{lem2}
$\overline{\Gamma}_S$ has a face of at most one of the types $X_4^3$, $X_2X_4^2$,
$X_2^2X_4$.
\end{lemma}

\pf
We consider the relations in $\pi_1(\widehat X^+ (\alpha_-))$ coming from
the corresponding triangle faces of $\Gamma_S$.
A face of type $X_4^3$, $X_2X_4^2$ or $X_2^2X_4$ would give the relation
$\bar x_4^3 = 1$, $\bar x_2 \bar x_4^2 =1$ or $\bar x_2^2 \bar x_4 = 1$, respectively.
It is easy to see that any two of these, together with the relation $\bar x_2=\bar x_4^2$
coming from $f_1$, imply $\bar x_2=\bar x_4=1$, contradicting Lemma~\ref{not in P}.
\qed

Let $C2$ be the configuration shown in Figure~\ref{C2}.

\begin{prop}\label{prop3}
$\Gamma_S$ contains the configuration $C2$.
\end{prop}

We will assume that $\Gamma_S$ does not contain
such a configuration, and show that this leads to a contradiction.
Equivalently, we make the following assumption:
\begin{equation}\label{eq1.1}
\text{\em $\overline{\Gamma}_S$ contains no face with two $(-,I)$ vertices and
one $(+,II)$ vertex.}
\end{equation}

Let $\F_1$ be the set of faces of $\overline{\Gamma}_S$ with two $(+,II)$ vertices
and one $(-,I)$ vertex, and let $\F_2$ be the set of faces with three
positive vertices, at least two of which are of type~II.
Note that $\bar{f}_1 \in \F_1$.
Let $\F = \F_1\cup \F_2$.

\begin{lemma}\label{lem4}
Every face of $\overline{\Gamma}_S$ that shares an edge with a face in $\F$
belongs to $\F$.
\end{lemma}

\pf
Let $\bar{f}$ be a face in $\F$, let $\bar g$ be a face of $\overline{\Gamma}_S$
that shares an edge $\bar e$ with $\bar f$, and let $v$ be the vertex of
$\bar g$ that is not a vertex of $\bar f$.

\medskip
\leftline{\bf Case (1). $\bar f\in \F_1$.}

First suppose that the vertices at the endpoints of $\bar e$ are
$(+,II)$ and $(-,I)$.
If $v$ is negative then by Lemma~\ref{lem1} it is a $(-,I)$ vertex, contradicting
assumption \eqref{eq1.1}.
Therefore $v$ is positive.
By Lemma~\ref{lem1} it is a $(+,II)$ vertex.
Hence $\bar g\in \F_1$.

Now suppose that $\bar e$ connects the two $(+,II)$ vertices of $\bar f$.
If $v$ is negative then by Lemma~\ref{lem1} it is of type~I, and hence $\bar g\in\F_1$.
If $v$ is positive then $\bar g\in \F_2$.

\medskip
\leftline{\bf Case (2). $\bar f\in \F_2$.}

First suppose that $\bar e$ connects two vertices of type~II.
If $v$ is negative then by Lemma~\ref{lem1} it is of type~I so $\bar g\in\F_1$.
If $v$ is positive then $\bar g\in\F_2$.

If $\bar e$ connects a $(+,I)$ vertex and a $(+,II)$ vertex then $v$
is positive by Lemma~\ref{lem1}.
Also, $\bar f$ is of type $X_2X_4^2$, so by Lemma~\ref{lem2} $\bar g$ is also of
type $X_2X_4^2$, and hence $\bar g\in \F_2$.
\qed

Now we prove Proposition \ref{prop3}.
Lemma~\ref{lem4}  implies that every face of $\overline{\Gamma}_S$ has at least two
positive vertices.
But this is easily seen to contradict the fact that $\Gamma_S$ has the
same number of positive and negative vertices.
We conclude that assumption \eqref{eq1.1} is false, i.e.
Proposition~\ref{prop3} holds.
\qed

If $W$ is a word in $X_2^{\pm1}$ and $X_4^{\pm1}$ we denote by $\ep_{X_2}(W)$ and
$\ep_{X_4}(W)$ the exponent sum in $W$ of $X_2$ and $X_4$ respectively, and if $D$
is a disk in $X^+$ of type $W$ then we define $\ep_{X_2}(D) = \ep_{X_2}(W)$,
$\ep_{X_4}(D) = \ep_{X_4}(W)$.

A disk in $X^+$ with 1, 2 or 3 corners will be called a {\em monogon,
bigon} or {\em trigon}, respectively.

\begin{lemma}\label{lem5}
There are no monogons.
\end{lemma}

\pf
Let $D$ be a  monogon.
Applying the Loop Theorem to $D$, among disks with $\ep_{X_2} + \ep_{X_4}\ne0$, we
get an embedded monogon $D'$ of the same type as $D$.
Then $D'$ is a boundary compressing disk for $F$ in $M$, contradicting
the fact that $F$ is essential.
\qed

\begin{figure}[!ht]
\centerline{\includegraphics{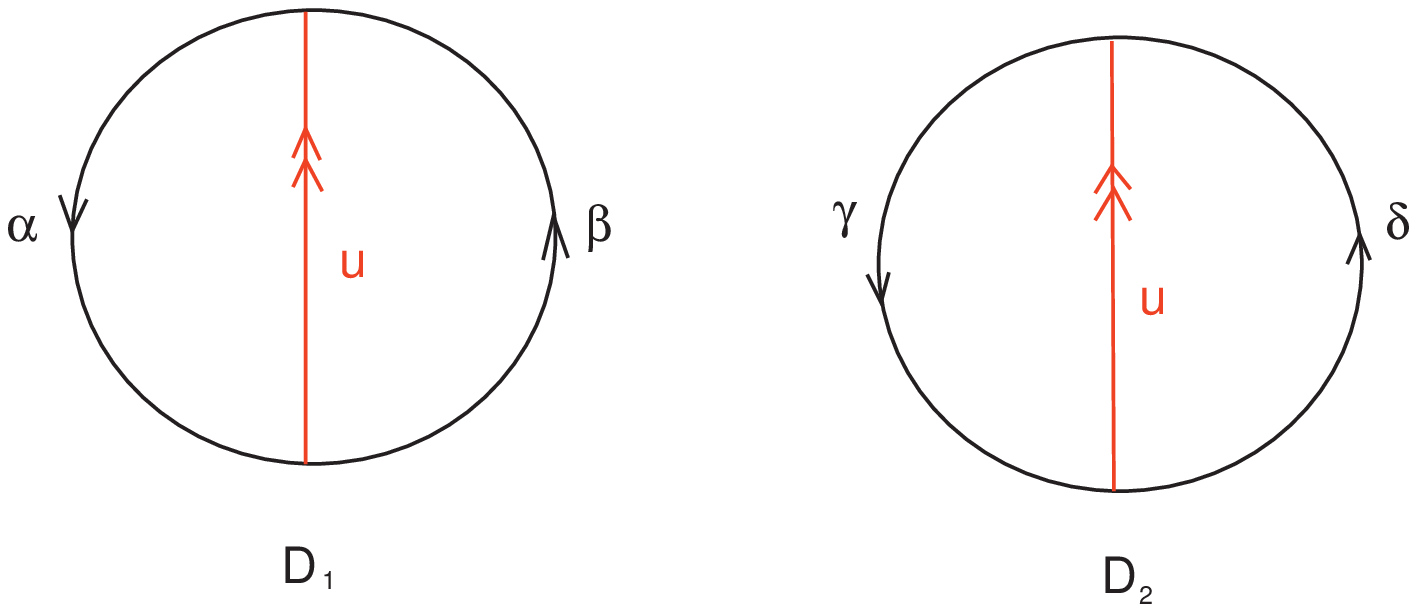}} \caption{ }\label{Fig3}
\end{figure}

\begin{lemma}\label{lem6'}	
There is no $A_-$-trigon of type $W$ with $|\ep_{X_2}(W)| + |\ep_{X_4}(W)| =1$.
\end{lemma}

\pf
Such a disk would give rise to the relation $\bar x_2=1$ or $\bar x_4=1$ in
$\pi_1(\widehat X^+ (\alpha_-))$, contradicting Lemma~\ref{not in P}.
\qed

Let $D$ be a singular disk in $X^+$.
We say that an embedded disk $E$ is {\em nearby\/} $D$ if $\partial E$ is
contained in a small regular neighborhood of $\partial D$ in $\partial X^+$.

\begin{lemma}\label{lem7}		
If there is an $A_-$-trigon of type $W$ then there is a nearby
embedded $A_-$-trigon of type $W$ if $W = X_2^{\pm 3}$ or $X_4^{\pm 3}$ and of type $W$ or
$W^* = W(X_2^{-1}, X_4)$ otherwise.
\end{lemma}

\pf
After possibly interchanging $X_2$ and $X_4$ we may assume without loss of
generality that $\ep_{X_2} (W)\not\equiv 0$ $(\mod 2)$.
Let $D$ be an $A_-$-trigon of type $W$.
The Loop Theorem gives an embedded $A_-$-disk $D'$ with
$\ep_{X_2}(D')\not\equiv 0$ $(\mod 2)$.
Lemmas~\ref{lem5} and~\ref{lem6'} now show that $D'$ is of the desired type.
\qed

\begin{figure}[!ht]
\centerline{\includegraphics{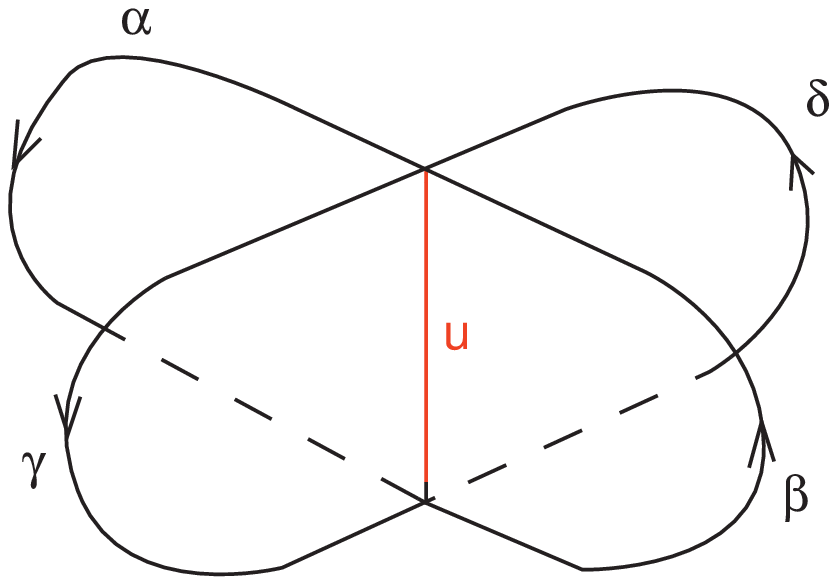}} \caption{ }\label{Fig4}
\end{figure}

Let $D_1,D_2$ be properly embedded disks in $X^+$.
Putting $D_1$ and $D_2$ in general position, $D_1\cap D_2$ will be a
compact 1-manifold.
A standard cutting and pasting argument allows us to eliminate the circle
components of $D_1\cap D_2$, without changing
$\partial D_1$ and $\partial D_2$.
So suppose that $D_1\cap D_2$ consists of $n\ge 1$ arcs.
Let $u$ be one of these arcs.
Then $u$ cuts $D_i$ into disks $D'_i$, $D''_i$, $i=1,2$, and the endpoints
of $u$ cut $\partial D_1$ and $\partial D_2$ into pairs of arcs
$\alpha,\beta$ and $\gamma,\delta$ respectively; see Figure~\ref{Fig3}.

Cutting and pasting $D_1$ and $D_2$ along $u$ we get four disks
$D'_1\cup D'_2$, $D'_1\cup D''_2$, $D''_1\cup D'_2$, and $D''_1\cup D''_2$,
with boundaries $\alpha\gamma^{-1}$, $\alpha\delta$, $\beta\gamma$ and $\beta\delta^{-1}$ respectively.
See Figure~\ref{Fig4}.

After a small perturbation, each of these disks $E$ meets each of $D_1$
and $D_2$ in less than $n$ double arcs, disjoint from the singularities
of $E$.

The arc $u$ is {\em trivial\/} in $D_1$ if either $\alpha$ or $\beta$
contains no corner of $D_1$, and similarly for $D_2$.
If $u$ is trivial in $D_1$ and in $D_2$ then without loss of generality
$\alpha$ contains no corner of $D_1$ and $\gamma$ contains no corner of $D_2$.
Then $D''_1\cup D'_2$ has the same type as $D_1$ and $D'_1\cup D''_2$
has the same type as $D_2$.
If $u$ is trivial in $D_1$ but not in  $D_2$, and $D_2$ is a bigon or
trigon, then at least one of $D'_1\cup D'_2, D'_1\cup D''_2, D_1'' \cup D_2'$, or $D_1'' \cup D_2''$ is a
monogon, contradicting Lemma~\ref{lem5}.

\begin{lemma}\label{lem8}		
If there is no $A_-$-disk of type $X_4^3$ then
there are disjoint embedded $A_-$-disks of types $X_2^{-1}X_4^2$ and $X_2^{-2}X_4$.
\end{lemma}

\pf
The faces $f_1$ and $f_2$ of $\Gamma_S$ in Figures~\ref{C1} and \ref{C2} are $A_-$-disks
of types $X_2^{-1}X_4^2$ and $X_2^{-2}X_4$ respectively. Since $\bar x_2^{-1} \bar x_4^2 = \bar x_2^{-2} \bar x_4 = 1$ in $\pi_1(\widehat X^+(\alpha_-))$,
Lemma \ref{not in P} implies that neither relation $\bar x_2 \bar x_4^2 = 1$ nor $\bar x_2^2 \bar x_4 = 1$ can hold.
Lemma~\ref{lem7} then gives embedded $A_-$-disks $D_1$ and $D_2$ of
type $X_2^{-1}X_4^2$ and $X_2^{-2}X_4$ respectively, which we may assume intersect in double arcs, none of which is trivial
in $D_1$ or $D_2$.
Let $u$ be a double arc.
Ignoring orientations, there are two possibilities for $u$ in each of
$D_1$ and $D_2$, shown in Figures~\ref{Fig5(1)} and \ref{Fig5(2)} respectively.

\begin{figure}[!ht]
\centerline{\includegraphics{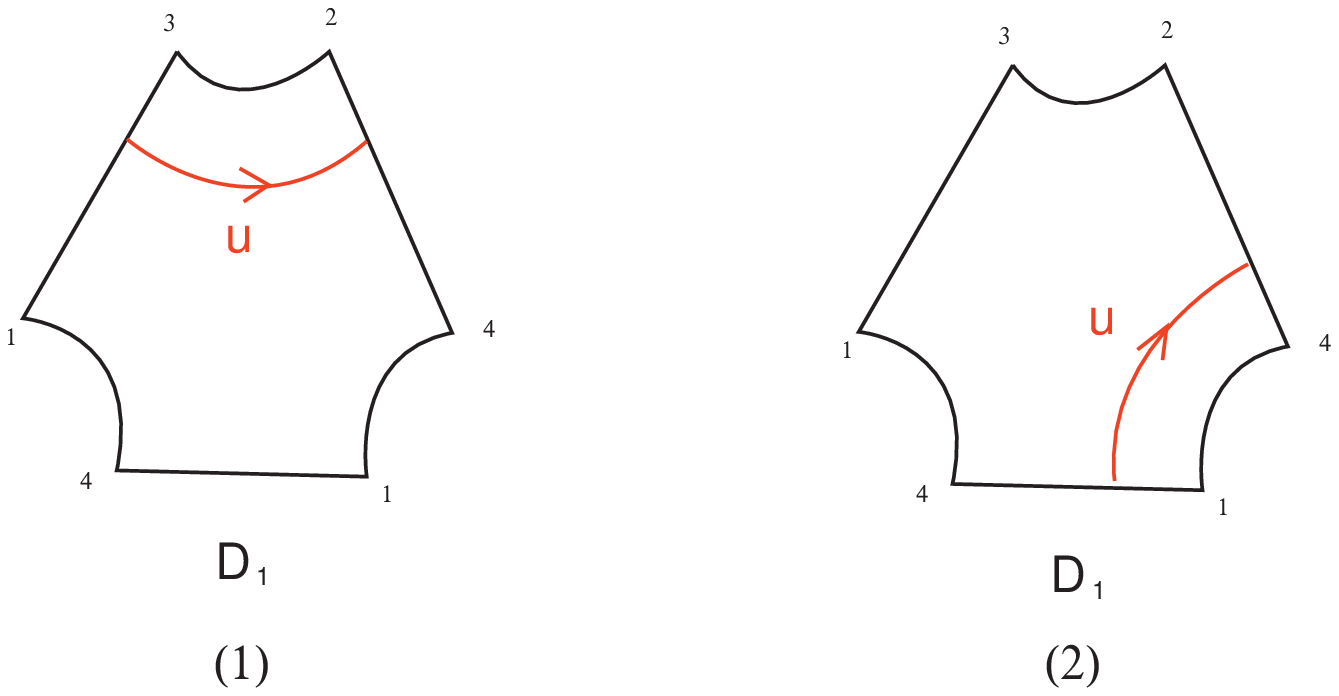}} \caption{ }\label{Fig5(1)}
\end{figure}

\begin{figure}[!ht]
\centerline{\includegraphics{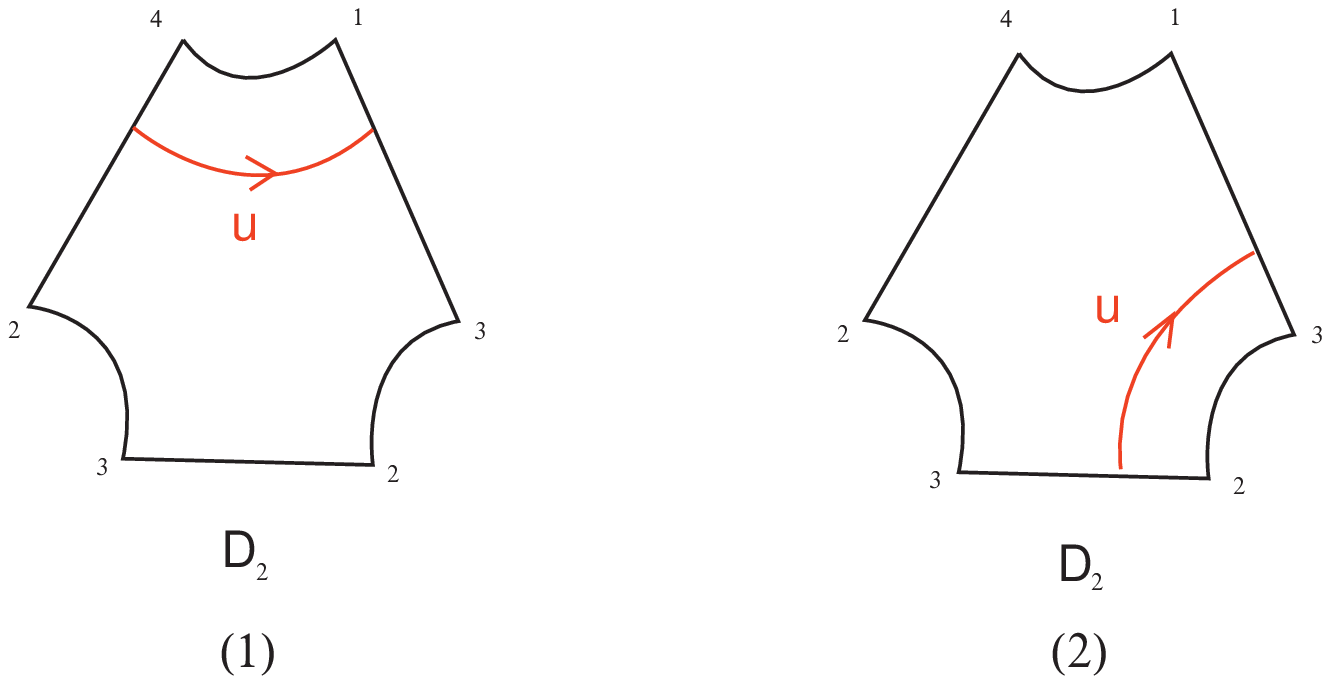}} \caption{ }\label{Fig5(2)}
\end{figure}

Orient $u$ as shown in Figures~\ref{Fig5(1)} and \ref{Fig5(2)}, so in case~(1), $(\alpha,\beta) =(X_4^2,X_2^{-1})$,
and in case~(2), $(\alpha,\beta) = (X_4,X_4X_2^{-1})$.
(Here, we are using the natural convention of labeling the oriented arc
$\alpha$ or $\beta$ by the sequence of corners it contains.)
If $u$ is oriented on $D_2$ as shown in Figure~5, then in case~(1),
$(\gamma,\delta) = (X_2^{-2},X_4)$, and in case~(2),
$(\gamma,\delta) = (X_2^{-1},X_2^{-1}X_4)$.
Note that if $u$ on $D_2$ is oriented in the opposite direction to that
shown then $\gamma$ and $\delta$ are interchanged, so that in
case~$(\overline{1})$, $(\gamma,\delta) = (X_4,X_2^{-2})$, and in
case~$(\overline{2})$, $(\gamma,\delta) = (X_2^{-1}X_4,X_2^{-1})$.
This gives eight possibilities $(i,j)$, where $i$ denotes case~$(i)$ for
$D_1$, $i=1,2$, and $j$ denotes case~$(j)$ for $D_2$, $j=1,2$,
$\overline1$ or $\overline2$.
In each case we choose one of the four disks obtained by cutting and
pasting along $u$.
Below we indicate the chosen disk by the arcs in its boundary and record
its type:
$$\vbox{\halign{$#$\enspace\hfil&$#$&\qquad$#$\hfil&\qquad$#$\hfil\cr
(1,1)&:&\alpha\delta&X_4^3\cr
(1,2)&:&\beta\delta^{-1}&X_2^{-1}X_4^{-1}X_2\cr
(2,1)&:&\beta\delta^{-1}&X_4X_2^{-1}X_4^{-1}\cr
(2,2)&:&\alpha\delta&X_2^{-1}X_4^2\cr
(1,\Bar1)&:&\alpha\gamma^{-1}&X_4^2X_4^{-1}\cr
(1,\Bar2)&:&\alpha\delta&X_2^{-1}X_4^2\cr
(2,\Bar1)&:&\beta\gamma&X_2^{-1}X_4^2\cr
(2,\Bar2)&:&\alpha\gamma^{-1}&X_4X_4^{-1}X_2
\cr }}
$$

In case $(1,1)$ we get an $A_-$-disk of type $X_4^3$, contradicting
our assumption.

Cases $(1,2)$, $(2,1)$, $(1,\Bar1)$ and $(2,\bar2)$ contradict
Lemma~\ref{lem6'}.

In the remaining three cases, $(2,2)$, $(1,\Bar2)$ and $(2,\Bar1)$ we get
an $A_-$-disk $E$ of type $X_2^{-1}X_4^2$.
By Lemma~\ref{lem7} there is a nearby embedded $A_-$-disk $E'$ of
type $X_2X_4^2$ or $X_2^{-1}X_4^2$. The former is impossible as otherwise we would have $\bar x_2 \bar x_4^2 = \bar x_2^{-1} \bar x_4^2 = \bar x_2^{-2} \bar x_4 = 1$ in $\pi_1(\widehat X^+(\alpha_-))$, which implies $\bar x_2 = \bar x_4 = 1$, contrary to Lemma \ref{not in P}.
Thus $E'$ has type $X_2^{-1}X_4^2$. Noting that $|E' \cap D_2| \le |E\cap D_2| < |D_1\cap D_2|$,
if we continue in this manner we eventually get an embedded $A_-$-disk of type
$X_2^{-1}X_4^2$ disjoint from $D_2$.
\qed

\noindent {\bf Note.}
It is easy to see that in cases $(2,2)$, $(1,\Bar2)$ and $(2,\Bar1)$ we
also get a disk of type $X_2^{-2}X_4$, so we could equally well have fixed
$D_1$ and obtained an embedded $A_-$-disk of type $X_2^{-2}X_4$ disjoint from $D_1$.

Let $D$ be an embedded disk in $X^+$.
Recall that the corners of $D$ are the components of
$\partial D\cap (A_{23}\cup A_{41})$.
We will refer to the components of $\partial D\cap F$ as the
{\em edges\/} of $D$.
We label the endpoints of the edges of $D$ with the label of the
corresponding corner.
Thus, if we have a disjoint union $\D$ of embedded disks whose
$X_2^{\pm1}$-corners are labeled so that reading clockwise around boundary
component 2 of $F$ they appear in the order $a,b,c,\ldots$, then they
appear in the same order $a,b,c,\ldots$ reading anticlockwise around
boundary component 3 of $F$.
Similarly, the clockwise order of the $X_4^{\pm1}$-corners of $\D$ at
boundary component 4 is the same as their anticlockwise order at boundary
component~1.
This ordering condition puts constraints on the existence of the
disjoint embedded arcs in $F$ that are the edges of $\D$.

\begin{lemma}\label{lem9bis}
If there is an $A_-$-disk of type $X_4^3$ then $\alpha_- = \varphi_+$.
\end{lemma}

\pf
If there is an $A_-$-disk of type $X_4^3$ then there is an embedded $A_-$-disk
$D$ of type $X_4^3$ by Lemma~\ref{lem7}.
Let the corners of $D$ be $a,b$ and $c$; see Figure~\ref{Fig6}.

\begin{figure}[!ht]
\centerline{\includegraphics{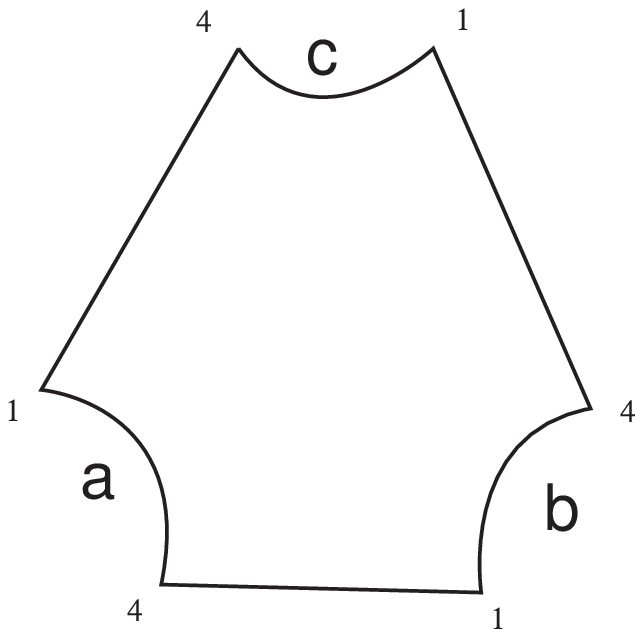}} \caption{ }\label{Fig6}
\end{figure}

Then without loss of generality the edges of $D$ appear in
$\widehat A_-$ as shown in Figure~\ref{Fig7}.

\begin{figure}[!ht]
\centerline{\includegraphics{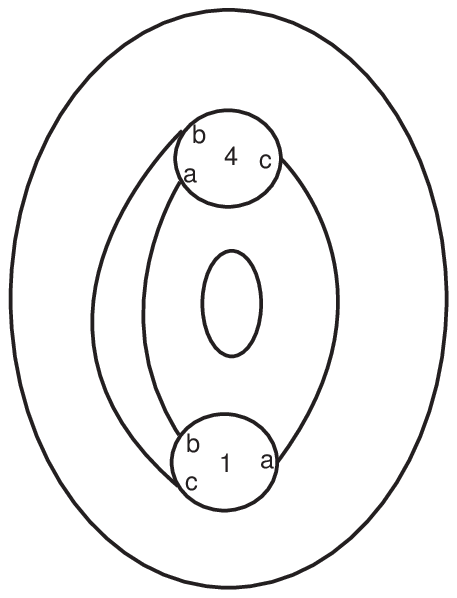}} \caption{ }\label{Fig7}
\end{figure}

Let $V= \widehat A_- \times I \cup H_{(41)} \cup N(D) \subseteq \widehat X^+$.
Then, taking as ``base-point'' a disk in $\widehat A_-$ containing the
two left-hand edges in Figure \ref{Fig7} together with fat vertices $v_1$ and $v_4$,
we get $\pi_1(V) \cong \langle x_4,t : x_4^3 =t\rangle \cong \zed$, where $t$
is represented by $\alpha_-$, the core of $\widehat A_-$.
Hence $V$ is a solid torus and $\alpha_-$ has winding number~3 in $V$.
Let $A'$ be the annulus $\partial V - \Int \widehat A_-$.
By Assumption \ref{assumption minimal}, the torus $(\widehat F- \widehat A_-) \cup A'$
bounds a solid torus $V'$ in $\widehat X^+$.
Therefore $\widehat X^+ = V\cup_{A'} V'$ is a Seifert fibre space with
base orbifold $D^2 (3,b)$, and $\alpha_-$ is the slope of the Seifert
fibre $\varphi_+$.
\qed

\begin{lemma}\label{lem10bis}
If there are disjoint embedded $A_-$-disks of types $X_2^{-1}X_4^2$ and
$X_2^{-2}X_4$ then $\alpha_- = \varphi_+$.
\end{lemma}

\pf
Let $D_1,D_2$ be disjoint embedded $A_-$-disks of types $X_2^{-1}X_4^2$ and
$X_2^{-2}X_4$ respectively.

First note that the union of the edges of $D_1$ and $D_2$ and the fat
vertices $v_1,v_2,v_3,v_4$ cannot be contained in a disk in $\widehat A_-$.
For this would give relations $x_2 = x_4^2$, $x_4 = x_2^2$ in $\pi_1(\widehat X^+)$,
implying $x_2^3 = x_4^3 =1$.
But $x_2$ and $x_4$ are non-trivial (Lemma~\ref{not in P}), so $\pi_1(\widehat X^+)$
would have non-trivial torsion, contradicting the fact that $\widehat X^+$
is a Seifert fibre space with base orbifold $D^2 (2,b)$.

Let the corners of $D_1$ and $D_2$ be $a,b,c$ and $p,q,r$ respectively;
see Figure~\ref{Fig8}.

\begin{figure}[!ht]
\centerline{\includegraphics{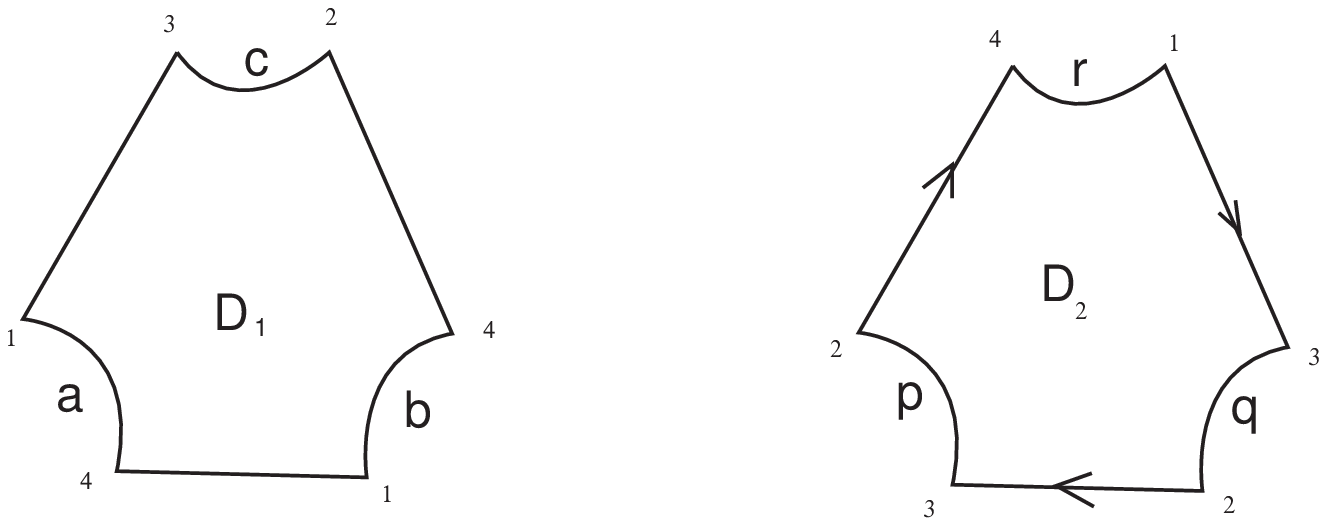}} \caption{ }\label{Fig8}
\end{figure}

Without loss of generality the labels $c,p,q$ appear in this order
anticlockwise around $v_3$.
The possible arrangements of the edges of $D_1$ and $D_2$ in
$\widehat A_-$ are then shown in Figure~\ref{Fig9} (1)--(6).
(For simplicity we have labeled the corners $a,b,c,p,q,r$ only in
Figure~\ref{Fig9}~(1).)

\begin{figure}[!ht]
\centerline{\includegraphics{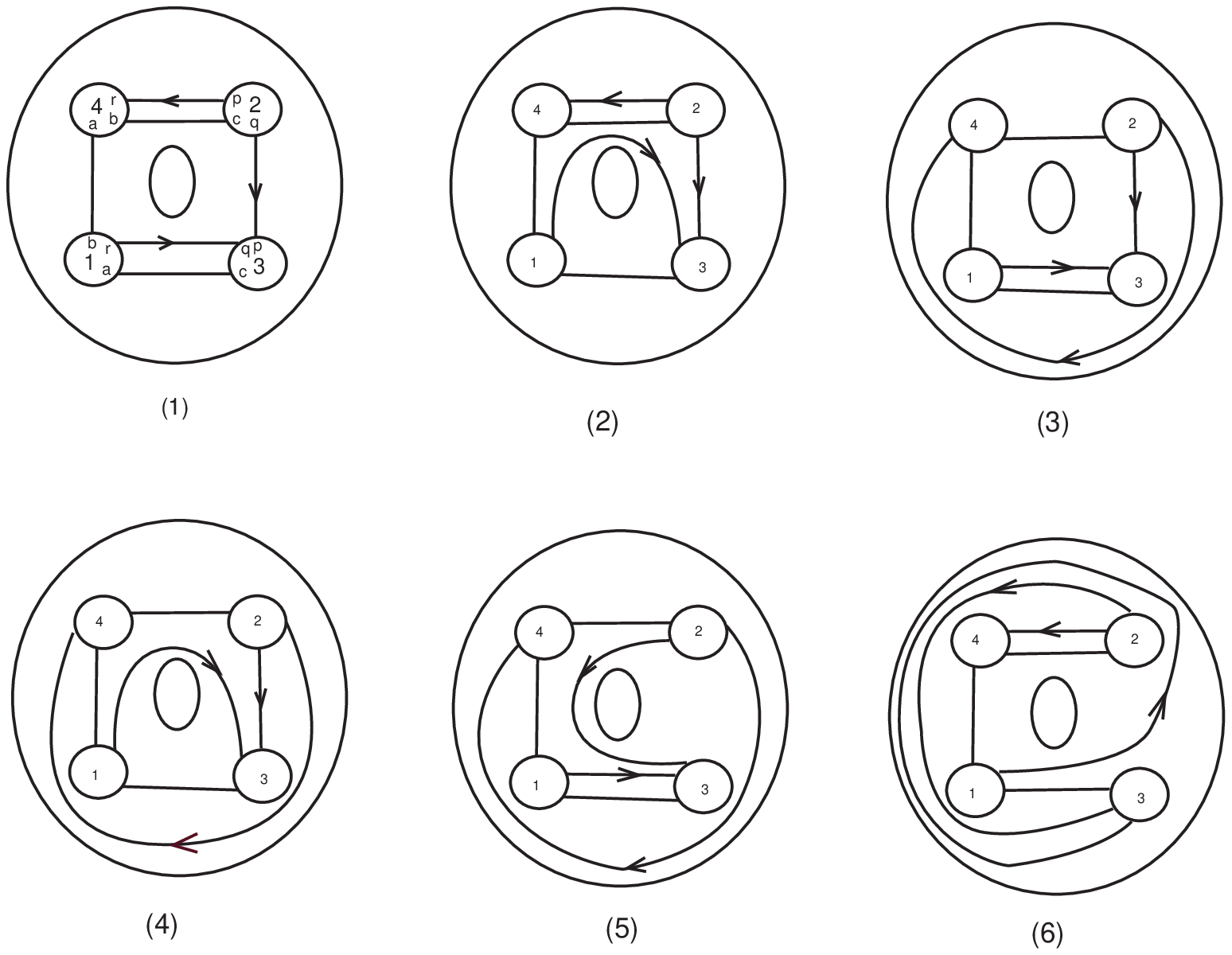}} \caption{ }\label{Fig9}
\end{figure}

Let $V = \widehat A_- \times I \cup H_{(23)} \cup H_{(41)} \cup N(D_1)
\cup N(D_2) \subseteq \widehat X^+$.
Then $\pi_1 (V)$ is generated by $x_2,x_4,t$, where $t$ is represented by
$\alpha_-$, the core of $\widehat A_-$.
We take as ``base-point'' a disk in $\widehat A_-$ containing the edges
of $D_1$ together with the vertices $v_1,v_2,v_3,v_4$.
Then the disk $D_1$ gives the relation $x_2=x_4^2$ in $\pi_1(V)$.
The relation determined by $D_2$ is as follows in cases (1)--(6):
\begin{itemize}
\item[(1)] $x_2^{-1} tx_2^{-1} x_4 =1$
\item[(2)] $x_2^{-1} tx_2^{-1} x_4t =1$
\item[(3)] $x_2^{-1} tx_2^{-1} tx_4 =1$
\item[(4)] $x_2^{-1} tx_2^{-1} tx_4t =1$
\item[(5)] $x_2^{-2} tx_4 =1$
\item[(6)] $x_2^{-2} x_4t =1$
\end{itemize}

In case (1) we get $x_4 = x_2t^{-1} x_2 = x_4^2 t^{-1} x_4^2$, and hence $t= x_4^3$.
Therefore $\pi_1(V)\cong\zed$, generated by $x_4$.
It follows that $V$ is a solid torus and $\alpha_-$ has winding number~3 in $V$.
Hence (see the proof of Lemma~\ref{lem9bis}) $\alpha_- = \varphi_+$.

In case (2) $x_4 = (x_2t^{-1})^2 = z^2$, where $z=x_2t^{-1}$.
Thus $\pi_1(V)$ is generated by $x_2,x_4,z$ with relations $x_2=x_4^2$, $x_4=z^2$.
Therefore $\pi_1(V)\cong\zed$, generated by $z$.
Also $t= z^{-1}x_2 = z^{-1}z^4 =z^3$.
Hence again $\alpha_- = \varphi_+$, as in case~(1).

Cases (3), (5) and (6) are similar and are left to the reader.

In case (4) we have $x_4 = t^{-1} x_2t^{-1} x_2t^{-1}$, so $x_4x_2 = z^3$, where
$z= t^{-1}x_2$.
Since $x_2=x_4^2$, $\pi_1(V)$ has presentation $\langle x_4,z :x_4^3 = z^3\rangle$.
But this contradicts the fact that $\widehat X^+$ is a Seifert fibre space
with base orbifold $D^2 (2,b)$.
\qed

\begin{cor}\label{cor11}
$\alpha_- = \varphi_+$.
\end{cor}

\pf
This follows from Lemmas \ref{lem8}, \ref{lem9bis} and \ref{lem10bis}.
\qed

We complete the analysis by showing that Corollary \ref{cor11} implies that $\dot{\Phi}_3^+$ contains an $\widehat F$-essential annulus, contrary to our assumptions.

\begin{lemma} \label{no such pair}
There is no pair of disjoint embedded $A_-$-disks of types $X_2^{-1}X_4^2$ and $X_2^{-2}X_4$.
\end{lemma}

\pf The manifold $V$ given in the proof of
Lemma \ref{lem10bis} is a solid torus such that $\widehat A_-$ is contained in $\partial V$ with winding number $3$
and thus the annulus $A=\partial V\setminus \widehat A_-$ is a vertical annulus in the Seifert fibred
structure of $\widehat X^+$. But $A$ is contained in $X^+$ and thus it is an essential annulus in $X^+$.
So $\partial A=\partial \widehat A_-$ can be isotoped in $F$ into the interior of $\dot{\Phi}_1^+$. Therefore
$\dot{\Phi}_1^+\cap A_-= \dot{\Phi}_1^+\cap \t_-(\dot{\Phi}_1^+)=\dot{\Phi}_1^+\cap \dot{\Phi}_2^-$ is a pair of $\widehat
F$-essential annulus components. Hence $\dot{\Phi}_3^+$ is a  pair of $\widehat F$-essential annulus components.
But this contradicts our assumption that $\dot{\Phi}_3^+$ is  a set of tight components. Thus there is no such pair of embedded disks.
\qed

Thus the situation given by Lemma \ref{lem10bis} cannot arise. Then by Lemma \ref{lem8}, there is an $A_-$-disk of type $X_4^3$. We also have a $A_-$ disk of type $X_2^{-1}X_4^2$ given by configuration $C1$. By Lemma \ref{lem7}, there is an embedded $A_-$-disks $D_1$ of type $X_4^3$ and another $D_2$ of type $X_2^{-1}X_4^2$ or $X_2X_4^2$. In the latter case, we have the relations $\bar x_4^3 = \bar x_2^{-1} \bar x_4^2 = \bar x_2 \bar x_4^2 = 1$ in $\pi_1(\widehat X^+(\alpha_-))$, which imply that $\bar x_2 = \bar x_4 = 1$, contrary to Lemma \ref{not in P}. Thus $D_2$ has type $X_2^{-1}X_4^2$.

\begin{lemma}
 There are disjoint embedded $A_-$-disks $D_1$ and $D_2$ of types $X_4^3$ and $X_2^{-1}X_4^2$ respectively.
\end{lemma}

\pf The proof is similar to that of Lemma \ref{lem8}. We may  assume that among all such pairs of embedded
$A_-$-disks of types $X_4^3$ and $X_2^{-1}X_4^2$ respectively, $D_1$ and $D_2$ have been chosen to have
the minimal number of intersection components. If the disks $D_1$ and $D_2$  are disjoint then we are done.
So suppose they intersect. We may assume that they intersect transversely in double arcs, none of which
is trivial in $D_1$ or $D_2$.

Let $u$ be an oriented  double arc which is outermost in $D_1$ with respect to the corner it cuts off (i.e.
the interior of the corner is disjoint from $D_2$) as shown in Figure \ref{ab} (a). Then there are six possibilities for the oriented arc $u$ in $D_2$, as shown in Figure \ref{ab} (b1)-(b6) respectively.

If case (b1) of Figure \ref{ab} occurs, then cutting and pasting  $D_1$ and $D_2$ will produce an embedded
$A_-$-disk  of type $X_2^{-1}X_4^2$ having fewer intersection components with $D_1$ than does $D_2$,
contradicting our assumption on $D_1$ and $D_2$.

\begin{figure}[!ht]
\centerline{\includegraphics{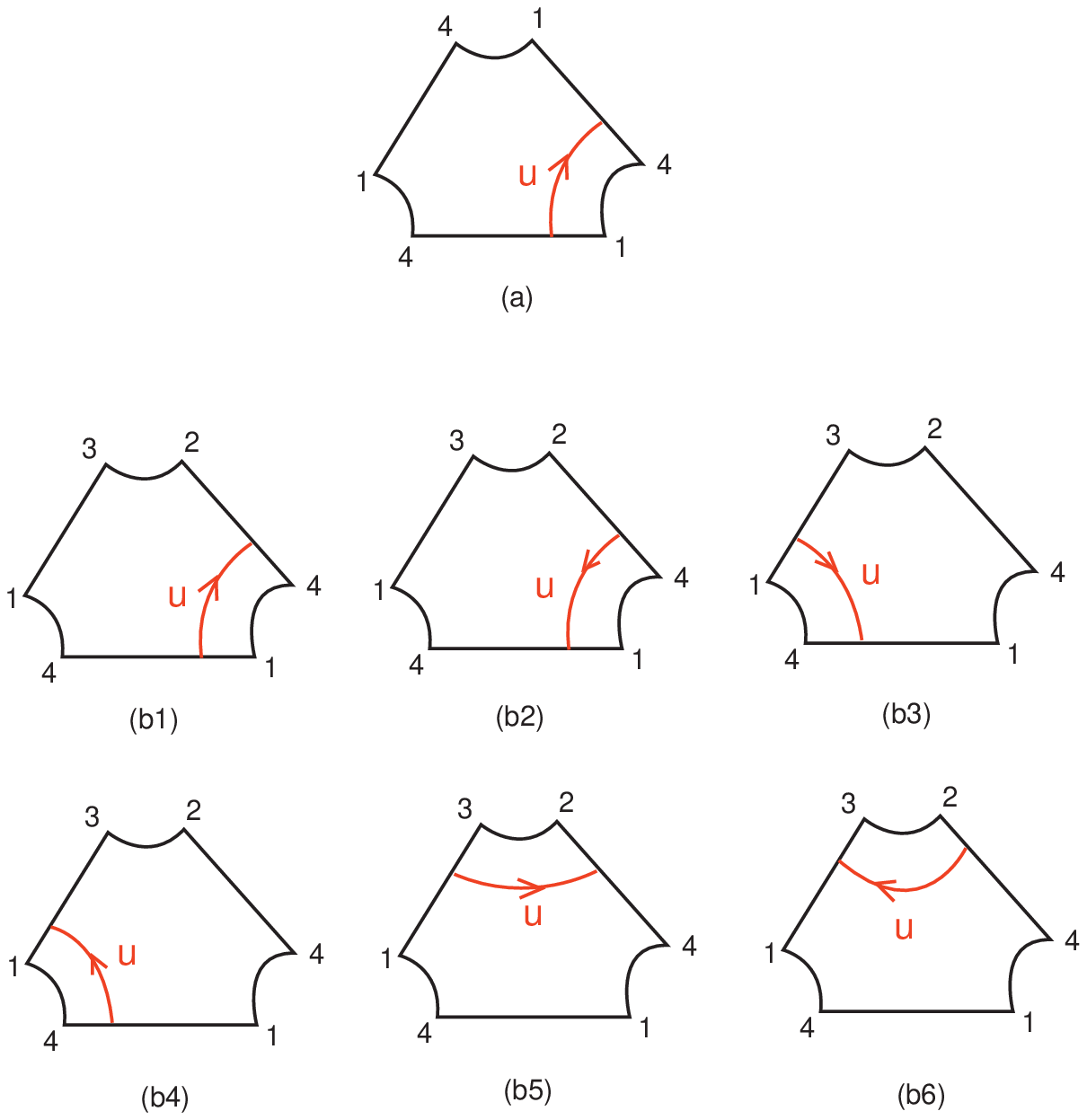}} \caption{ }\label{ab}
\end{figure}

If case (b2) of Figure \ref{ab} occurs, then cutting and pasting will produce an $A_-$-disk
of type $X_4^2$. So in $\pi_1(\widehat X^+(\a_-)$ we have the relation  $\bar x_4^2 = 1$. Together with the
relations $\bar x_4^3 = 1$ and $\bar x_2^{-1} \bar x_4^2 = 1$ this implies that $\bar x_2 = \bar x_4 = 1$ in $\pi_1(\widehat X^+(\a_-))$, a
contradiction.

Cases (b3) and (b4) of Figure \ref{ab} can be treated similarly to the cases (b1) and (b2) respectively.

If case (b5) of Figure \ref{ab} occurs, then cutting and pasting $D_1$ and $D_2$ will produce an $A_-$-disk
of type $X_2^{-1}X_4$. Thus in $\pi_1(\widehat X^+(\a_-)$ we have $\bar x_2^{-1} \bar x_4=1$. Since $\bar x_4^3=1$ and $\bar x_2^{-1} \bar x_4^2=1$, we deduce $\bar x_2 = \bar x_4 = 1$ in $\pi_1(\widehat X^+(\a_-))$, a
contradiction.

Finally, if case (b6) of Figure \ref{ab} occurs, then cutting and pasting will produce an embedded
$A_-$-disk of type $X_4^3$ which is disjoint from $D_2$, giving an obvious contradiction.
\qed

Now let $V$ be a regular neighborhood in $\widehat X^+$ of the set  $\widehat A_-\cup H_{(23)}\cup
H_{(41)}\cup D_1\cup D_2$. As in the proof of Lemma \ref{lem10bis}, the union of the edges of $D_1$ and the fat
vertices $v_1$ and $v_4$ cannot lie in a disk in $\widehat A_-$ and can be assumed to appear as shown in
Figure \ref{Fig7}. Thus two edges of $D_1$ connect the fat vertices $v_1$ and $v_4$ from the left hand side and one
edge of $D_1$ connects $v_1$ and $v_4$ from the right hand side.

Let $e_1$ be the edge of $D_2$ connecting $v_1$ and $v_4$, $e_2$ the edge of $D_2$ connecting $v_1$ and
$v_3$, and $e_3$ the edge of $D_2$ connecting $v_2$ and $v_4$.

Now take as ``base-point'' a disk in $\widehat A_-$ containing the union of the two left-hand side edges
of $D_1$, the fat vertices $v_1, v_2, v_3, v_4$, and the edges $e_2, e_3$. Then the disk $D_1$ will give
the relation
$$x_4^3=t,$$
and the disk $D_2$ will give either the relation $$x_2^{-1}x_4^2=1$$ (when $e_1$ connects $v_1$ and $v_4$
from the left hand side, cf. Figure \ref{Fig7}) or the relation $$x_2^{-1}x_4t^{-1}x_4=1$$ (when $e_1$ connects $v_1$ and
$v_4$ from the right hand side, cf. Figure \ref{Fig7}), where $t$ is represented by $\a_-$. In either case we see
that $V$ has the fundamental group $$\pi_1(V) = \langle x_4,t: x_4^3=t \rangle.$$ So  the manifold $V$  is a solid torus
such that $\widehat A_-$ is contained in $\partial V$ with winding number $3$. Now argue as in the proof of Lemma \ref{no such pair} to see that $\Phi_3^+$ cannot be a set of tight components, yielding the final contradiction.
\qed

\subsubsection{The case $m = 4$, $\Delta(\alpha, \beta) = 6$ and $\overline{\Gamma}_S$ rectangular with edges of weight $6$} \label{rectangular}

As $m=4$, we may assume:
\begin{itemize}
  \item both $\dot{\Phi}_3^+$ and $\dot{\Phi}_5^-$ consist of a pair of tight components, each a twice-punctured disk;
 \vspace{.3cm}  \item $\dot{\Phi}_5^+$ is a collar on $\partial F$, and so contains no large components.
\end{itemize}

Recall that $b_1,...,b_4$ denote the four boundary components $\partial F$ appearing successively along $\partial M$. These four circles cut $\partial M$ into four annuli $A_{i,i+1}, i=1,...,4$, such that $\p
A_{i,i+1}=b_i\cup b_{i+1}$ (indexed mod (4)). We may assume that $\partial X^+ =F\cup A_{2,3}\cup
A_{4,1}$.

As in \S \ref{hexagonal case}, an {\it $n$-gon} (disk) in $X^+$ means a singular disk $D$ with $\partial D \subseteq \partial X^+$ such that $\partial D \cap (A_{2,3} \cup A_{4,1})$ is a set of $n$ embedded essential arcs in $A_{2,3} \cup A_{4,1}$, called the {\it corners} of $D$, and $\partial D\cap F$ is a set of $n$ singular arcs, called the {\it edges} of $D$. Recall that a $1$-gon, $2$-gon or $3$-gon will be called a monogon, bigon or trigon.

There are no monogons in $X^+$ (cf. Lemma \ref{lem5}).

\begin{lemma}\label{Phi5+NoLarge}
There is no bigon $D$ in $X^+$ whose edges $e_1, e_2$ are essential paths in $(\dot\Phi_5^-, \partial F)$ and for which the inclusion $(D, e_1 \cup e_2) \to (X^+, \dot\Phi_5^-)$ is essential as a map of pairs.
\end{lemma}

\pf Suppose otherwise that such a bigon $D$ exists. Then $D$ gives rise to an essential homotopy
between its two edges and  thus the edges of $D$ can be homotoped, relative to their end points, into $\dot\Phi_1^+$. Then the essential intersection $\dot\Phi_5^- \wedge \dot\Phi_1^+$ contains a
large component and therefore so does $\dot\Phi_6^+ = \tau_+(\dot\Phi_5^- \wedge \dot\Phi_1^+)$, contrary to our assumption that $\dot\Phi_5^+$ has no large components.
\qed

Recall that $h$ is the $\pi_1$-injective map from the torus $T$ into $M(\alpha)$
which induces the graph $\Gamma_S$ in $T$. For a subset $s$ of $T$ we use $s^*$ to denote its image under the map $h$.

The image under $h$ of every edge of a rectangular face of $\Gamma_S$ is contained in $\dot\Phi_5^-$. The images of the middle two edges of every family of six parallel edges of $\Gamma_S$ are contained in
$A_-$.

As before the classes $x_j \in \pi_1(\widehat X^+)$ are defined and we use $\bar x_j$ to denote their images in $\pi_1(\widehat X^+(\alpha_-))$.

For notational simplicity, let us write $\dot\Phi_5^- = Q$, a pair
of twice-punctured disks.
A singular disk $D\subset X^+$ whose edges are contained in $Q$ will
be called a {\em $Q$-disk\/}.
An {\em essential\/} $Q$-disk is a $Q$-disk $D$ such that $\partial D$
is essential in $\partial X^+$. The following two lemmas are key to our analysis.

\begin{lemma} \label{PropA}
An essential $Q$-$n$-gon, $n\le4$, is a $4$-gon of type $X_2X_4^{-1}X_4X_4^{-1}$
or $X_4X_2^{-1}X_2X_2^{-1}$.
\end{lemma}

\begin{lemma} \label{PropB}
There cannot be essential $Q$-$4$-gons of both types
$X_2X_4^{-1}X_4X_4^{-1}$ and\break $X_4X_2^{-1}X_2X_2^{-1}$.
\end{lemma}

The proofs of these two lemmas will be given after we develop several necessary background results.

All the edges of $\bar \Gamma_S$ have weight 6.
We may assume without loss of generality that there is a family of
parallel edges of $\Gamma_S$ at one end of which the label sequence is
${1\ 2\ 3\ 4\ 1\ 2}$.

\begin{lemma}\label{lem1'}
$b_1$ and $b_2$ belong to different components of $Q$.
\end{lemma}

\begin{proof}
Suppose otherwise so that there is a rectangle face of $\Gamma_S$ as depicted in Figure \ref{rect1}.

\begin{figure}[!ht]
\centerline{\includegraphics{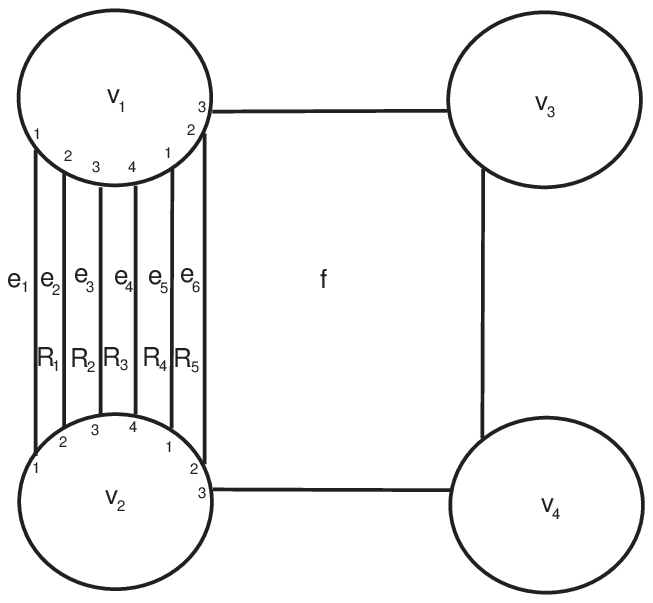}} \caption{ }\label{rect1}
\end{figure}

Here $e_1,e_2, e_3, e_4, e_5, e_6$ is a family of six successive
parallel edges which connect vertices $v_1$ and $v_2$ and whose label-permutation is the identity.
Let $R_i$ be the bigon face between $e_i, e_{i+1}$ for $i=1,...,5$ and $R$ the disk $R_1 \cup ... \cup R_5$. We know $R_2^*, R_4^*, f^* \subseteq X^+$ while
$R_1^*, R_3^*, R_5^* \subseteq X^-$.

There is a product structure $(R_i, e_i, e_{i+1}) = (e_i \times I, e_i \times \{0\}, e_i \times \{1\})$ such that for each $x \in e_i$, $(\{x\} \times I)^*$ is an $I$-fibre of $\Sigma_1^{(-1)^i}$. Thus $\tau_{(-1)^i}(e_i^*) = e_{i+1}^*$, so the free involution $h_5^-: \dot\Phi_5^- \to \dot\Phi_5^-$ (see the end of \S \ref{characteristic subsurfaces}) sends $e_1^*\cup b_1$ to $e_{6}^*\cup b_2$.
Proposition \ref{tight not invariant} then shows that $b_1$ and $b_2$ lie in different components of $\dot{\Phi}_5^-$. Hence $b_3$ and $b_4$ also lie in different components of $\dot{\Phi}_5^-$.
\end{proof}

It follows that $b_3$ and $b_4$ also belong to different components of $Q$.

If $(D,\partial D) \subset (X^+,\partial X^+)$, we will denote the type of
$D$ (see \S \ref{hexagonal case}) by $W(D)$.
Recall that $W(D)$ is defined up to cyclic permutation and inversion.

\begin{cor}\label{cor1.1}
Let $D$ be a $Q$-disk.
Then $W(D)$ does not contain the syllable $X_2X_4$ or $X_4X_2$.
\end{cor}

\begin{proof}
These give rise to a 34- or 12-edge in $\partial D$, respectively.
\end{proof}

\begin{lemma}\label{lem2'}
Let $D$ be a $Q$-disk.
Then in $W(D)$ no $Z\in \{ X_2^{\pm1},X_4^{\pm1}\}$ can be followed or preceded
by two distinct letters $\ne Z^{-1}$.
\end{lemma}

\begin{proof}
If $Z$ were followed by two distinct letters $\ne Z^{-1}$ the same component
of $Q$ would contain three boundary components $b_i$.
For example, if $W(D)$ contained syllables $X_2X_2$ and $X_2X_4^{-1}$ then
$\partial D$ would contain a 32-edge and a 31-edge, implying that
$b_1,b_2$ and $b_3$ belong to the same component of $Q$.
\end{proof}

\begin{proof}[Proof of Lemma \ref{PropA}]
Let $E$ be an essential $Q$-$n$-gon, $n\le 4$.
By the Loop Theorem we get an essential embedded $Q$-disk $D$, with
$\{\text{corners of }D\} \subset \{\text{corners of }E\}$.
\renewcommand{\qed}{}
\end{proof}

\begin{lemma}\label{lem3}
$D$ is a $4$-gon.
\end{lemma}

\begin{proof}
$D$ cannot be a monogon, since then $D$ would be a boundary-compressing
disk for $F$.

$D$ cannot be a bigon by Lemma \ref{Phi5+NoLarge}.

So suppose $D$ is a trigon.
It is easy to see that Corollary~\ref{cor1.1} and Lemma~\ref{lem2'} imply
that $D$ contains only, say, $X_2$-corners.
By Lemma \ref{lem8} $|\ep_{X_2}(D)| \ne1$, so $W(D) = X_2^3$.

Let $U = \widehat F\times I \cup H_{(23)}\cup N(D) \subset \widehat X^+$.
Then $\pi_1 (U) \cong \pi_1 (\widehat F) * \zed/3$.
It follows that $U$, and hence $\widehat X^+$, has a closed summand with
fundamental group $\zed/3$, a contradiction.
\end{proof}

There are three possibilities: $D$ has either
\begin{itemize}
\item[(A)] {\em all $X_2$-corners (or all $X_4$-corners)};
\item[(B)] {\em two $X_2$-corners and two $X_4$-corners};
\item[(C)] {\em one $X_2$-corner and three $X_4$-corners (or vice versa)}.
\end{itemize}

\begin{lemma}\label{lem4'}
Case (A) is impossible.
\end{lemma}

\begin{proof}
We may suppose that $D$ has all $X_2$-corners.
Note that $|\ep_{X_2} (D)|$ is not 1 by Lemma \ref{lem8}
and if it is $>1$ then we get a contradiction as in the last part
of the proof of Lemma~\ref{lem3}.
Hence $\ep_{X_2} (D)=0$.
Thus $W(D)= X_2^2 X_2^{-2}$ or $X_2X_2^{-1} X_2X_2^{-1}$.

In the first case, $\partial D$ contains a 23-edge, and hence
$b_2$ and $b_3$ belong to the same component of $Q$.
But $\partial D$ also contains a 2-loop and a 3-loop, which
clearly must intersect, contradicting the fact that $D$ is embedded.

In the second case, label the corners of $D$ $a,b,c,d$ as shown in
Figure \ref{R1}.

\begin{figure}[!ht]
\centerline{\includegraphics{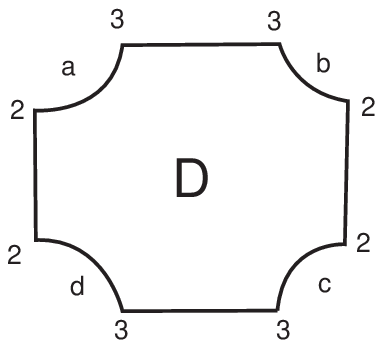}} \caption{ }\label{R1}
\end{figure}

Then $\partial D$ is as shown in Figure \ref{R2}.
Let $V = \widehat F\times I\cup H_{(23)}$.
Note that $\partial V = \widehat F \times \{0\}\,
\raise.4ex\hbox{$\scriptstyle\coprod$}\, G$,
where $G$ is a surface of genus~2.
We see from Figure \ref{R2} that $\partial D$ is isotopic in $G$ to a
meridian of $H_{(23)}$, and so bounds a non-separating disk $D'\subset V$.
Then $D\cup D'$ is a non-separating 2-sphere $\subset V\cup N(D)\subset
\widehat X^+$, a contradiction.
\end{proof}

\begin{figure}[!ht]
\centerline{\includegraphics{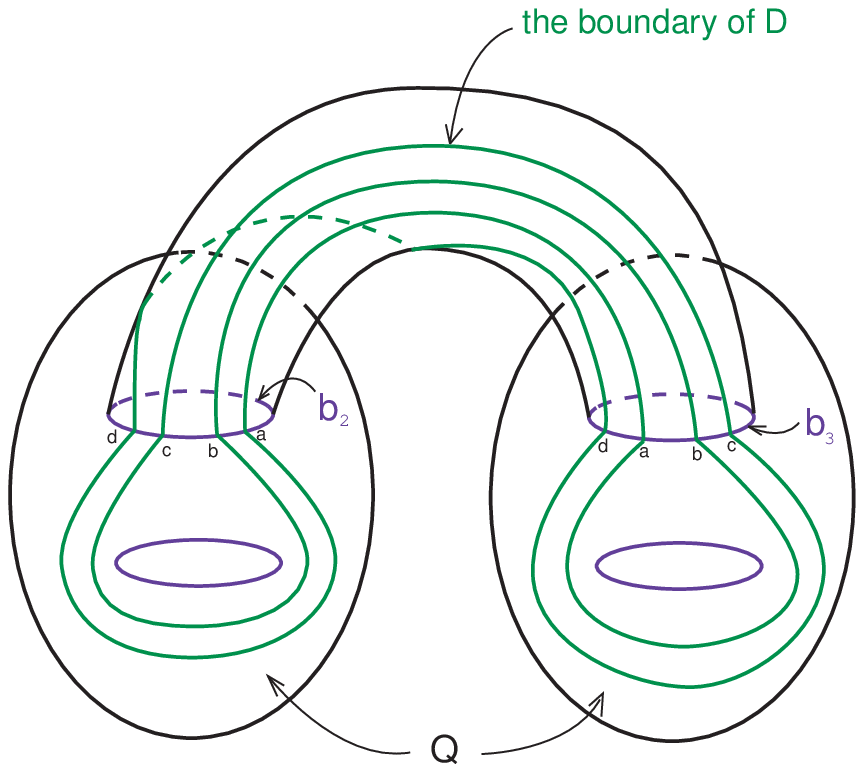}} \caption{ }\label{R2}
\end{figure}

\begin{lemma}\label{lem5'}
Case (B) is impossible.
\end{lemma}

\begin{proof}
By Corollary~\ref{cor1.1} and Lemma~\ref{lem2'}, the only possibilities
for $W(D)$ are $X_2X_2^{-1}X_4X_4^{-1}$ and $X_2X_4^{-1}X_2X_4^{-1}$.

In the first case, $\partial D$ contains a 24-edge.
Therefore $b_2$ and $b_4$ belong to the same component of $Q$,
and hence $b_1$ and $b_3$ belong to the same component of $Q$.
But  $\partial D$ also contains a 1-loop and a 3-loop, which must intersect.

In the second case, let $U = \widehat F\times I\cup H_{(23)} \cup
H_{(41)} \cup N(D) \subset \widehat X^+$.
Then $\pi_1(U) \cong \pi_1 (\widehat F) * \zed * \zed/2$, implying
that $\widehat X^+$ has a closed summand with fundamental group $\zed/2$,
a contradiction.
\end{proof}

By Lemmas~\ref{lem4'} and \ref{lem5'}, Case (C) must hold; so suppose
that $D$ has one $X_2$-corner and three $X_4$-corners.
Since $\{\text{corners of }D\} \subset \{\text{corners of }E\}$, $E$
is also a 4-gon with one $X_2$-corner and three $X_4$ corners.
Corollary~\ref{cor1.1} rules out all possibilities for $W(E)$ except
$X_2X_4^{-3}$ and $X_2X_4^{-1} X_4X_4^{-1}$, and the first is ruled out by
Lemma~\ref{lem2'}.

This completes the proof of Lemma \ref{PropA}.
\qed

\begin{lemma}\label{lem6''}
There do not exist disjoint $Q$-disks of types $X_2X_4^{-1}X_4X_4^{-1}$
and $X_4X_2^{-1}X_2X_2^{-1}$.
\end{lemma}

\begin{proof}
Let $D_1,D_2$ be $Q$-disks of types $X_2X_4^{-1}X_4X_4^{-1}$ and $X_4X_2^{-1}X_2X_2^{-1}$,
respectively.
Since $\partial D_1$ contains a 31-edge, $b_1$ and $b_3$ must belong to
the same component of $Q$.
But $\partial D_1$ contains a 1-loop and $\partial D_2$ contains a 3-loop,
and these must intersect.
\end{proof}

\begin{proof}[Proof of Lemma \ref{PropB}]
Let $E_1,E_2$ be $Q$-disks of types $X_2X_4^{-1}X_4X_4^{-1}$ and $X_4X_2^{-1}X_2X_2^{-1}$
respectively.
By the Loop Theorem and Lemma \ref{PropA} we get embedded $Q$-disks $D_1$ and
$D_2$ of these types.
By Lemma~\ref{lem6''}, $D_1$ and $D_2$ must intersect; consider an arc of
intersection, coming from the identification of arcs $u_i \subset D_i$,
$i=1,2$.
We may assume that the endpoints of $u_i$ lie on distinct edges of $D_i$,
$i=1,2$.
Then  $u_i$ separates $D_i$ into two disks, $D'_i$ and $D''_i$, say, where
$D'_i$ contains either one or two corners of $D_i$.

If $D'_1$ and $D'_2$ each contain a single corner, and these corners are
distinct, then $D'_1 \cup D'_2$ is a $Q$-bigon with one $X_2$- and one
$X_4$-corner, contradicting Lemma \ref{Phi5+NoLarge}.

If $D'_1$ and $D'_2$ both contain, say, a single $X_2$-corner, then $u_1$ is as
shown in Figure \ref{R3}, which also shows one of the three possibilities for $u_2$.
Since $b_1$ and $b_3$ lie in one component of $Q$, say $Q_1$, and $b_2$
and $b_4$ lie in the other component, say $Q_2$, and each of the arcs $u_1$
and $u_2$ has one endpoint in $Q_1$ and one in $Q_2$, $u_1$ and $u_2$ must
be identified as shown in Figure \ref{R3}.
Then $D_1^* = D''_1 \cup D'_2$ is a $Q$-disk of type $X_2X_4^{-1}X_4X_4^{-1}$
having fewer intersections than $D_1$ with $D_2$.

\begin{figure}[!ht]
\centerline{\includegraphics{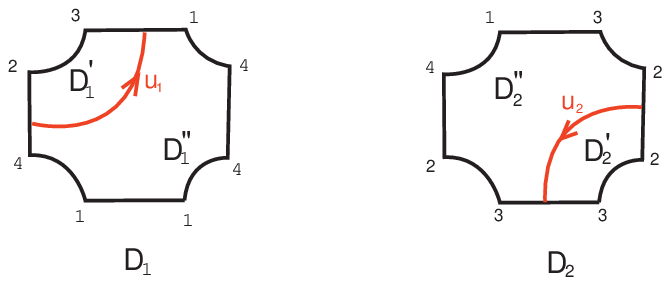}} \caption{ }\label{R3}
\end{figure}

\begin{figure}[!ht]
\centerline{\includegraphics{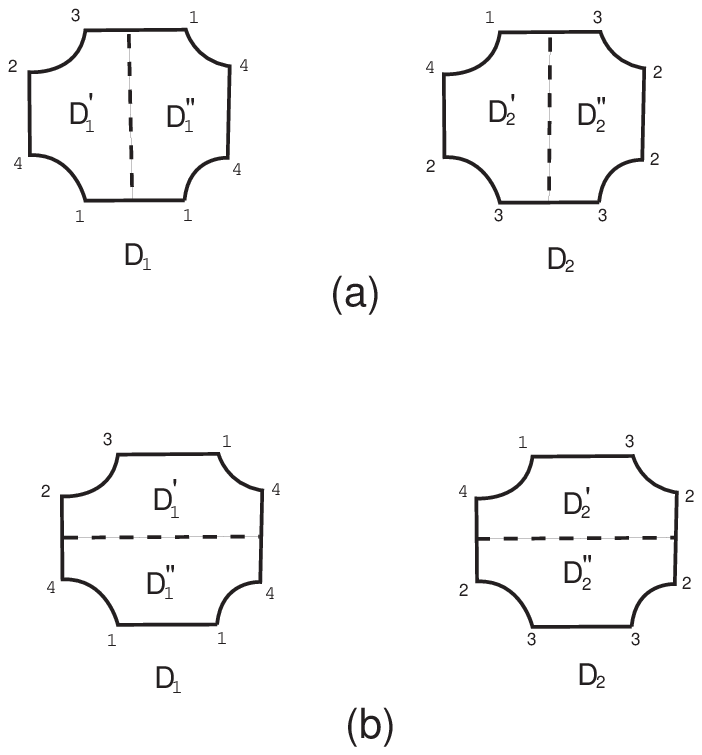}} \caption{ }\label{R4}
\end{figure}

If each of $D'_i$ and $D''_i$ contains two corners, $i=1,2$, the two
possibilities for $u_1$ and $u_2$ are illustrated in Figure \ref{R4}, (a) and (b).
In both cases, $D^*_1 = D''_1 \cup D'_2$ is again a $Q$-disk of type
$X_2X_4^{-1}X_4X_4^{-1}$ having fewer intersections with $D_2$.

Applying the Loop Theorem to the disk $D^*_1$ constructed above, and
using Lemma \ref{PropA}, we get an embedded $Q$-disk of type $X_2X_4^{-1}X_4X_4^{-1}$
having fewer intersections with $D_2$ than $D_1$.
Continuing, we eventually get disjoint embedded $Q$-disks of types
$X_2X_4^{-1}X_4X_4^{-1}$ and $X_4X_2^{-1}X_2X_2^{-1}$, contradicting Lemma~\ref{lem6''} .

This completes the proof of Lemma \ref{PropB}.
\end{proof}

Since each edge of $\bar\Gamma_S$ has weight 6, consecutive 4-gon corners
of $\Gamma_S$ at a given vertex are distinct.
Hence the total number of $X_2$-corners in the 4-gon faces of $\Gamma_S$ is
the same as the total number of $X_4$-corners.
Since a 4-gon face of $\Gamma_S$ is an essential $Q$-disk, this
contradicts Lemmas \ref{PropA} and \ref{PropB}.

This completes the proof for the case where $\bar\Gamma_S$ is rectangular.

\subsection{Proof when $\dot \Phi_3^+$ is not a union of tight components} \label{proof not all tight}

In this section we suppose that $X^-$ is a twisted $I$-bundle and $\dot \Phi_3^+$ is not a union of tight components. Proposition \ref{sep annulus-in-3-+} implies that
\begin{itemize}

\vspace{-.3cm} \item $M(\beta)$ is Seifert with base orbifold $P^2(2,n)$ for some $n > 2$;

\vspace{.3cm} \item $\widehat F$ is vertical in $M(\beta)$;

\vspace{.3cm} \item $\breve \Phi_1^+$ is connected and completes to an $\widehat F$-essential annulus;

\vspace{.3cm} \item  $\breve\Phi_3^+$ completes to the union of two $\widehat F$-essential annuli.

\end{itemize}
\vspace{-.3cm}
By Corollary \ref{max weight of edge}, the edges of $\overline{\Gamma}_S$ have weight bounded above by $m+4$. Hence for any vertex $v$ of $\overline{\Gamma}_S$ we have
\begin{equation}\label{eq delta 2}
\text{\em $\Delta(\alpha, \beta) \leq \hbox{valency}_{\overline{\Gamma}_S}(v) + 4 \Big(\frac{\hbox{valency}_{\overline{\Gamma}_S}(v)}{m} \Big)$}
\end{equation}

As the Seifert structure on $\widehat X^+$ is unique, it is the restriction of the Seifert structure of $M(\beta)$ and therefore its base orbifold is $D^2(2, n)$. Recall from \S \ref{background 2} that $\phi_+$ is the fibre slope on $\widehat F$ of this structure. By hypothesis, it is also the fibre slope of the Seifert structure on $\widehat X^-$, a twisted $I$-bundle over the Klein bottle with base orbifold a M\"{o}bius band. Hence $\phi_+ = \tau_-(\phi_+) = \alpha_-$, so the class $t \in \pi_1(\widehat F) \leq \pi_1(\widehat X^+)$ is the fibre class.

\begin{prop} \label{kbeta}
Suppose that conditions \ref{eq background 2} hold and $\dot \Phi_3^+$ is not a union of tight components. If $m = 4$ there is a presentation $\langle a, b, z : a^2, b^n, abz^{-2}\rangle$ of $\Gamma = \pi_1(P^2(2,n))$ such that the image in $\Gamma$ of the core $K_\beta$ of the $\beta$-filling solid torus in $M(\beta)$ represents the element $\kappa = az^{-1}b^{-1}z \in \Gamma$, at least up to conjugation and taking inverse.
\end{prop}

\pf
Let $E_0$ be the $\widehat F$-essential annulus $\widehat{\breve \Phi_1^+}$. Then $\partial E_0$ is a pair of $\widehat F$-essential curves  $c_1, c_2$. By Proposition \ref{sep annulus-in-3-+}, $\dot \Phi_3^+$ is the union of two $\widehat F$-essential annuli, and there are disjoint, non-separating annuli $A_1^-, A_2^-$ properly embedded in $X^-$ such that $\partial A_1^- \cup \partial A_2^- \subseteq \dot \Phi_1^+$ and for each $j$, $\partial \dot \Phi_1^+ \cap \partial A_j^-$ is a boundary component of $\dot \Phi_1^+$ which we can take to be $c_j$.

We can assume that each $A_j^-$ is $\tau_-$-invariant. Then $A_1^- \cup A_2^-$ splits $\widehat X^-$ into two $\tau_-$-invariant solid tori $V_1, V_2$ where $V_1 \cap M \supset \partial M \cap X^-$ and $V_2 \subset M$. The reader will verify that $\widehat{\breve \Phi_1^+} \cap V_1$ is the union of disjoint $\widehat F$-essential annuli $E_1, E_2$ where $\tau_-(E_1) = E_2$ while $F \cap V_2$ is the union of disjoint $\widehat F$-essential annuli $E_3, E_4$ such that $\tau_-(E_3) = E_4$. Without loss of generality we can suppose that $c_j \subset E_j$ ($j = 1, 2$) and $\widehat{\breve \Phi_1^+} = E_1 \cup E_2 \cup E_3$. Then ${\breve \Phi}_3^- = \tau_-(\widehat{\breve \Phi_1^+}) = E_1 \cup E_2 \cup E_4$.

Number the components of $\partial F$ so that $\partial M \cap X^+$ consists of two annuli, one with boundary $b_1 \cup b_4$, the other with boundary $b_2 \cup b_3$. Let $x = x_4$ and $y = x_2$ be the elements of
$\pi_1(\widehat X^+) \leq \pi_1(M(\beta))$ defined using the disk $D \subset A_- = {\breve \Phi}_3^-$.

The intersection of $\partial M$ with $X^-$ consists of two annuli, one with boundary $b_1 \cup b_2$ and the other with boundary $b_3 \cup b_4$. Let $w_1, w_3$ be the associated elements of $\pi_1(\widehat X^-) \leq \pi_1(M(\beta))$ determined by $D$. Since $V_1 \cap M$ is a twice-punctured annulus cross an interval we see that $w_1 = w_3^{\pm 1}$. We claim that $w_1 = w_3^{-1}$. To see this, exchange $E_1$ and $E_2$, if necessary, so that $b_j \subset E_j$ for $j = 1, 2$. We will be done if $b_4 \in E_1$ and $b_3 \in E_2$. Suppose otherwise that $b_3 \in E_1$ and $b_4 \in E_2$. Then $\tau_+(E_1)$ is an $\widehat F$-essential annulus in $E_0$ containing $c_2 \cup b_2 \cup b_4$ while $\tau_+(E_2)$ is an $\widehat F$-essential annulus in $E_0$ containing $c_1 \cup b_1 \cup b_3$. It follows that $\partial \tau_+(E_1) \setminus c_2$ is an $\widehat F$-essential curve in $\breve \Phi_1^+$ which separates $b_2 \cup b_4$ from $b_1 \cup b_3$. A similar conclusion holds for $\partial \tau_+(E_2) \setminus c_1$. It follows that up to isotopy we can assume $\tau_+(E_1 \cap F) = E_2 \cap F$. On the other hand, by construction we have $\tau_-(E_1 \cap F) = E_2 \cap F$ and therefore $(\tau_- \circ \tau_+)(E_1 \cap F) = E_1 \cap F$. Hence the inclusion of $E_2 \cap F$ in $F$ admits essential homotopies of arbitrarily large length, contrary to the results of \S \ref{length}. Thus $w_1 = w_3^{-1}$.  Let $z$ be the image of $w_1$ in $\Gamma$.

The class of $\pi_1(M(\beta))$ carried by $K_\beta$ is given by $xw_1yw_1^{-1}$. Let $\kappa$ be its image in $\Gamma$.

The base orbifold of $\widehat X^+$ is $D^2(2, n)$ with fundamental group $\pi_1(D^2(2, n)) = \langle a, b : a^2 = 1, b^n = 1 \rangle$. Here $a, b$ are chosen to be represented by oriented simple closed curves in the complement $P$ of the cone points of $D^2(2,n)$.

We can assume that the $E_i$ are vertical in the Seifert structure on $M(\beta)$. Since $D \subset E_1 \cup E_2 \cup E_4$, it projects to a proper subarc of the circle in $P^2(2,n)$ given by the image of the vertical torus $\widehat F$. Thus the images of $x$ and $y$ in $\pi_1(D^2(2,n))$ lie in $\{a^{\pm 1}, b^{\pm 1}\}$ (cf. the proof of Proposition \ref{not peripheral}). Further $z^2 \in \{ab, ab^{-1}, ba, b^{-1}a \} \subset \Gamma$. By construction $b_1 \cup b_4 \subset E_1$ and $b_2 \cup b_3 \subset E_2$ and so as $w_1$ is obtained by contenating an arc in $\partial M \cap X^-$ from $b_1$ to $b_2$ with an arc in $D$ from $b_2$ to $b_1$, it follows that one of the following four possibilities arises:
\begin{enumerate}

\item $x \mapsto a, y \mapsto b, z^2 = ba$ and $\kappa = azbz^{-1}$.

\item $x \mapsto a, y \mapsto b^{-1}, z^2 = b^{-1}a$ and $\kappa = azb^{-1}z^{-1}$.

\item $x \mapsto b, y \mapsto a, z^2 = ab$ and $\kappa = bzaz^{-1}$.

\item $x \mapsto b^{-1}, y \mapsto a, z^2 = ab^{-1}$ and $\kappa = b^{-1}zaz^{-1}$.

\end{enumerate}
In case (3) we have $\Gamma = \langle a,b,z : a^2,b^n,ab=z^2 \rangle$ where $\kappa = bzaz^{-1} = z(az^{-1}b^{-1}z)^{-1}z^{-1}$. In case (4) we have $\Gamma = \langle a,b,z : a^2,b^n,ab^{-1}=z^2 \rangle$ where $\kappa = b^{-1}zaz^{-1}$. Replacing $b$ by $b^{-1}$ gives the presentation stated in the proposition and $\kappa = bzaz^{-1} = z(az^{-1}b^{-1}z)^{-1}z^{-1}$ as before. In case (2) we replace $z$ by $z^{-1}$ and note that then $\kappa = az^{-1}b^{-1}z$. Finally in case (1) we replace $b$ by $b^{-1}$ and $z$ by $z^{-1}$ after which again we have $\kappa = az^{-1}b^{-1}z$.
\qed

\begin{prop} \label{delta = 2}
Suppose that conditions \ref{eq background 2} hold and $\dot \Phi_3^+$ is not a union of tight components. If $m \equiv 2$ $($mod $4$$)$ and $\Delta(\alpha, \beta)$ is even, then $\Delta(\alpha, \beta) = 2$.
\end{prop}

\pf Suppose otherwise. Consider the $2$-fold cover of $\widetilde M \to M$ which restricts to the cover $F \times I \to X^-$ on the $-$-side of $F$ and the trivial double cover on the $+$-side of $F$. Since $m \equiv 2$ (mod $4$) the boundary of $\widetilde M$ is connected. Now $\beta$ lifts to a slope $\beta'$ on $\partial \widetilde M$ with associated filling a Seifert manifold with base orbifold $S^2(2,n,2,n) \ne S^2(2,2,2,2)$. Hence $\beta'$ is a singular slope of some closed essential surface $S \subseteq \widetilde M$. Since the distance of $\alpha$ to $\beta$ is even, $\alpha$ also lifts to a slope $\alpha'$ on $\partial \widetilde M$ with the associated filling Seifert with base orbifold a $2$-sphere with three or four cone points. It's easy to see that the distance between $\alpha'$ and $\beta'$ is $\Delta(\alpha, \beta)/2$. Hence as $\beta'$ is a singular slope for $S$, $S$ is incompressible in $\widetilde M(\alpha')$. As $\widetilde M$ is hyperbolic, $S$ cannot be a torus and therefore must be horizontal in $\widetilde M(\alpha')$. It cannot be separating as the base orbifold of $\widetilde M(\alpha')$ is orientable. Thus it is non-separating. But then \cite[Theorem 1.5]{BGZ1} implies the distance between $\alpha'$ and $\beta'$ is at most $1$, so $\Delta(\alpha, \beta) = 2$.
\qed

\subsubsection{$M(\alpha)$ is very small}

We assume that $M(\alpha)$ is very small in this subsection and prove $\Delta(\alpha, \beta) \leq 3$.

\begin{lemma}\label{only torus}
$M(\beta)$ contains no horizontal essential surfaces. Thus every closed orientable incompressible surface in $M(\beta)$ is a vertical  torus.
\end{lemma}

\pf Suppose $M(\beta)$ contains a horizontal essential surface $G$.
Then for each $\epsilon$, the components of $G \cap \widehat X^\epsilon$ are horizontal incompressible surfaces in                    $\widehat X^\epsilon$. Hence if $\lambda$ denotes the slope on $\widehat F$ of the curves $G \cap \widehat F$, then $\lambda$ is the fibre slope of the Seifert structure on $\widehat X^-$ with base orbifold $D^2(2,2)$. In particular, $\Delta(\lambda, \phi_+) =  \Delta(\lambda, \alpha_-) =  1$. Then $\widehat X^+(\lambda)$ is a Seifert manifold with base orbifold $S^2(2, n)$ which admits a horizontal surface. Thus it must be $S^1 \times S^2$. But then $n = 2$ and therefore $X^+$ is a twisted $I$-bundle (Proposition \ref{order two}), contrary to our assumptions.
\qed

Note that closed, essential surfaces in $M$ have genus $2$ or larger. Hence we deduce the following corollary.

\begin{cor}\label{compresses}
If $M$ contains a closed orientable essential surface, then the surface must compress in $M(\beta)$.
\qed
\end{cor}

\begin{lemma}\label{at least 4 bdry comps}
If $\beta$ is not a singular slope, then any orientable essential surface $H$ in $M$ with boundary slope $\beta$
has at least $4$ boundary components.
\end{lemma}

\pf We may assume that $|\partial H|$ is minimal among all such surfaces.
Then by \cite[Theorem 2.0.3]{CGLS}, either $\beta$ is a singular slope or
$\widehat H$ is incompressible in $M(\beta)$.
So by our assumption $\widehat H$ is incompressible in $M(\beta)$.
Thus by Lemma \ref{only torus}, $\widehat H$ is an incompressible torus in
$M(\beta)$. Hence $|\partial H| \geq m$ and so is at least $4$.
\qed

We complete this part of the proof of Theorem \ref{main} using $PSL_2(\mathbb C)$-character variety methods. We refer the reader to \S 6 of \cite{BCSZ2} for the explanations of the relevant notation, background results, and references.

Now let $X_0 \subseteq X_{PSL_2}(M(\beta)) \subseteq X_{PSL_2}(M)$ be an irreducible curve which contains a
character of a non-virtually-reducible representation. Let $x$ be any ideal point of $\tilde X_0$. If
$\tilde f_\alpha$ has finite value at $x$, then \cite[Proposition 4.10]{BZ1} and Corollary \ref{compresses} imply that $\beta$ is a singular slope, in which case we would have  $\D(\alpha,\beta)\leq 1$. So every ideal point of $\tilde X_0$ is a pole of
$\tilde f_\alpha$. In particular  $X_0$ provides a non-zero Culler-Shalen seminorm $\|\cdot\|_{X_0}$
on $H_1(\partial M; \mathbb R)$ with $\beta$ the unique slope with $\|\beta\|_{X_0} = 0$.

By \cite[Proposition 10.2]{BCSZ2}  and \cite{BZ1} we have
$$\| \alpha \|_{X_0} \leq s_{X_0} + 5$$
Let $H$ be an essential surface associated to an ideal point
$x$ of $\tilde X_0$.
As $x$ is a pole of $\tilde f_\alpha$, $H$ has boundary slope $\beta$.
By Lemma \ref{at least 4 bdry comps},
$|\partial H| \geq 4$. This implies, by the arguments in \cite[Proposition 6.6]{BCSZ2},
that $s_{X_0}\geq 2$. Thus
$$\Delta(\alpha, \beta) = \frac{\|\alpha\|_{X_0}}{s_{X_0}} \leq 1 + 5/2 = 3.5$$	
Thus $\Delta(\alpha, \beta) \leq 3$, which completes the proof when $M(\alpha)$ is very small.

\subsubsection{$M(\alpha)$ is not very small}

We suppose that $M(\alpha)$ is not very small in this subsection and that $Y$ is a torus.

\begin{lemma} \label{mu larger}
Suppose that conditions \ref{eq background 2} hold and $\dot \Phi_3^+$ is not a union of tight components. If $\Delta(\alpha, \beta) > 5$ then and there is a vertex $v$ of $\overline{\Gamma}_S$ such that $\mu(v) > m \Delta(\alpha, \beta) - 4$, then $m = 4$.
\end{lemma}

\pf Proposition \ref{possible values for mu} and Inequality \ref{eq delta 2} show that  $3 \leq \hbox{valency}_{\overline{\Gamma}_S}(v) \leq 5$ and if $v$ has valency $3$, then $\Delta(\alpha, \beta) \leq 6$ with equality only if $m = 4$. If it has valency $4$, Proposition \ref{possible values for mu} shows that $\varphi_3(v) \geq 1$. Lemma \ref{no more than m and 2 version 2} then implies that $\Delta(\alpha, \beta) m$, the sum of the weights of the edges incident to $v$, is bounded above by $\hbox{max}\{3m + 14, 4m + 4\}$. Hence if $\Delta(\alpha, \beta) > 5$, then $m = 4$ and $\Delta(\alpha, \beta) = 6$. Finally suppose that $v$ has valency $5$. In this case $\varphi_3(v) \geq 4$ (cf. Corollary \ref{possible values for mu}) so Lemma \ref{no more than m and 2 version 2} implies that $\Delta(\alpha, \beta) m \leq \hbox{max}\{3m + 16, 4m + 6, 5m\}$. Hence if $\Delta(\alpha, \beta) > 5$, then $m = 4$.
\qed

In the absence of vertices $v$ of $\overline{\Gamma}_S$ for which $\mu(v) > m \Delta(\alpha, \beta) - 4$, Corollary \ref{mu constant} implies that $\mu(v) = m \Delta(\alpha, \beta) - 4$ for all vertices.

\begin{lemma} \label{if equal}
Suppose that conditions \ref{eq background 2} hold and $\dot \Phi_3^+$ is not a union of tight components. Assume moreover that $\mu(v) = m \Delta(\alpha, \beta) - 4$ for all vertices $v$ of $\overline{\Gamma}_S$. If $\Delta(\alpha, \beta) > 5$,  then either

$(i)$ $m = 4$, or

$(ii)$ $m = 8, \Delta(\alpha, \beta) = 6$, each edge has weight $12$, and $\overline{\Gamma}_S$ is rectangular.

\end{lemma}

\pf Proposition \ref{possible values for mu} shows that $4 \leq \hbox{valency}_{\overline{\Gamma}_S}(v) \leq 6$ for all vertices of $\overline{\Gamma}_S$. Further, Proposition \ref{local structure} shows that if
\vspace{-.1cm}
$$\eqno{(\ref{if equal}.1)} \;\;\;\;\;\;\;\;\;\;\; \left\{
\begin{array}{l}
\hbox{valency}_{\overline{\Gamma}_S}(v) = 4, \hbox{ then } \varphi_3(v) = 0, \varphi_4(v) = 4, \hbox{ and } \varphi_j(v) = 0 \hbox{ for } j > 4 \\ \\
\hbox{valency}_{\overline{\Gamma}_S}(v) = 5, \hbox{ then } \varphi_3(v) = 3 \hbox{ and } \varphi_4(v) = 2, \hbox{ and } \varphi_j(v) = 0 \hbox{ for } j > 4  \\ \\
\hbox{valency}_{\overline{\Gamma}_S}(v) = 6, \hbox{ then } \varphi_3(v) = 6, \hbox{ and } \varphi_j(v) = 0 \hbox{ for } j > 3
\end{array} \right. $$
 \vspace{.05cm}

Let $v$ be a vertex of valency $6$. Since the weight of each edge of $\overline{\Gamma}_S$ is at most $m+4$, Lemma \ref{no more than m and 2 version 2} implies that if some edge incident to $v$ has weight larger than $m$ then $\Delta(\alpha, \beta) m$, the sum of the weights of the edges incident to the vertex, is bounded above by $\hbox{max}\{3m + 18, 4m + 8\}$. Hence Proposition \ref{delta = 2} implies that $m = 4$ and $\Delta(\alpha, \beta) =  6$. If, on the other hand, each edge incident to $v$ has weight $m$ or less, then Inequality \ref{eq delta 2} shows that $\Delta(\alpha, \beta) = 6$ and each such edge has weight $m$. If some edge incident to $v$ connects it to a vertex $v_1$ of valency less than $6$,  \ref{if equal}.1 implies that the valency of $v_1$ is $5$ and $\varphi_3(v_1) = 3$. Then Lemma \ref{no more than m and 2 version 2} shows that $6m$, the sum of the weights of the edges incident to $v_1$, is bounded above by $\hbox{max}\{4m + 10, 5m + 4\}$. In either case, $m = 4$. Assume then that each edge incident to $v$ connects it to a vertex $v_1$ of valency $6$. Proceeding inductively we see that if $m > 4$, then each vertex in the component of $\overline{\Gamma}_S$ containing $v$ has valency $6$. It follows that $\overline{\Gamma}_S$ is hexagonal (cf. the proof of Lemma \ref{even weight}) and each edge of $\overline{\Gamma}_S$ has weight $m$. Since $\overline{\Gamma}_S$ must have a positive edge, Lemma \ref{weight of positive version 2} shows that $m = 6$.  But this is impossible by  Proposition \ref{delta = 2}. Thus $m = 4$.

Next let $v$ be a vertex of valency $5$. Then $\Delta(\alpha, \beta) m$, the sum of the weights of the edges incident to $v$, is bounded above by $\hbox{max}\{3m + 16, 4m + 10, 5m\}$. Since $\Delta(\alpha, \beta) > 5$, the only possibility is for $m = 4$.

Finally if there are no vertices of valency $5$ or $6$, each vertex of $\overline{\Gamma}_S$ has valency $4$ and thus Identities \ref{if equal}.1 implies that it has no triangle faces. Lemma \ref{euler} then shows that $\overline{\Gamma}_S$ is rectangular. Inequality \ref{eq delta 2} shows that $m \leq 8$ and
$$\Delta(\alpha, \beta) \leq \left\{
\begin{array}{ll} 8 & \hbox{ if } m = 4 \\
6  & \hbox{ if } m = 6, 8
\end{array} \right. $$
Since $\Delta(\alpha, \beta) \geq 6$, Proposition \ref{delta = 2} implies that $m \ne 6$. If $m = 8$, it is easy to see that each edge of $\overline{\Gamma}_S$ has weight $12$. This completes the proof.
\qed

By the last two results, the proof of Theorem \ref{main} when $\dot \Phi_3^+$ is not a union of tight components reduces to proving the following two propositions.

\begin{prop} \label{m=8rectangular}
If $m = 8, \Delta(\alpha, \beta) = 6$, $\overline{\Gamma}_S$ is rectangular, each of its edges has weight $12$, and $\dot \Phi_3^+$ is not a union of tight components, then $\Delta(\alpha, \beta) \leq 5$.
\end{prop}

\begin{prop} \label{m=4}
If $m = 4$ and $\dot \Phi_3^+$ is not a union of tight components, then $\Delta(\alpha, \beta) \leq 5$.
\end{prop}

\begin{proof}[Proof of Proposition \ref{m=8rectangular}]
Each component of $\dot\Phi_{j}^-$ is tight for $j \geq 5$ (Proposition \ref{tightness for large j}) and so $\dot\Phi_{11}^-$ has at least six tight components (Proposition \ref{tight increase}(2)). On the other hand,
since the weight of each edge of $\overline{\Gamma}_S$ is $12$, at least two components of $\dot\Phi_{11}^-$ have two or more outer boundary components. It follows that $\dot\Phi_{11}^-$ has two components, each having two outer boundary components. We shall call the union of these two large components $Q$.
By Lemma \ref{weight of positive version 2} each edge of $\overline{\Gamma}_S$ is negative.
Without loss of generality we may assume that there is a parallel family of
edges $\bar e$ of $\Gamma_S$ whose label sequence at one of the vertices
$v$ adjacent to $\bar e$ is ${1\ 2\ 3\ 4\ 5\ 6\ 7\ 8\ 1\ 2\ 3\ 4}$.
Therefore $b_1$ and $b_4$ belong to $Q$, and by looking at the corners
of the 4-gons of $\Gamma_S$ contiguous to $\bar e$ at $v$ we see that
$b_5$ and $b_8$ also belong to $Q$.
As in Lemma~\ref{lem1'}, $b_1$ and $b_4$ belong to different components
of $Q$, as do $b_5$ and $b_8$.

This case is now ruled out exactly as in \S \ref{rectangular}, with the corners (45) and (81)
replacing (23) and (41).
\end{proof}

The proof of Proposition \ref{m=4} requires a certain amount of preparatory work. We use $\Delta$ to denote $\Delta(\alpha, \beta)$ and assume it is at least $6$.

Let $\gamma_\beta \in \pi_1(M(\beta))$ be the element represented by the core $K_\beta$ of the Dehn filling solid torus. Then $[\alpha] \in \pi_1(M)$ is sent to $\gamma_\beta^\Delta \in \pi_1(M(\beta)) = \pi_1(M)/ \langle \langle [\beta] \rangle \rangle$. Hence $\pi_1(M)/ \langle \langle [\alpha], [\beta] \rangle \rangle \cong  \pi_1(M(\beta))/ \langle \langle \gamma_\beta^\Delta  \rangle \rangle$. Note that this group is a quotient of $\pi_1(M)/ \langle \langle [\alpha] \rangle \rangle \cong \pi_1(M(\alpha))$.

The quotient of $\pi_1(M(\beta))$ by the fibre-class is $\Gamma = \pi_1(P^2(2, n))$. As before, denote the image of $\gamma_\beta$ in $\Gamma$ by $\kappa$. By Proposition \ref{kbeta} $\Gamma$ admits a presentation $\langle a, b, z : a^2, b^n, abz^{-2}\rangle$ such that up to conjugation and taking inverse, $\kappa = a z^{-1} b^{-1} z$. Thus if we set $G = \Gamma / \langle \langle \kappa^\Delta \rangle \rangle$, then $G$ has a presentation
$$G = \langle a, b, z: a^2, b^n, abz^{-2}, (az^{-1}b^{-1}z)^\Delta \rangle$$
Since $\pi_1(M(\beta))/ \langle \langle \gamma_\beta^\Delta  \rangle \rangle$ is a quotient of $\pi_1(M(\alpha))$, the same is true for $G$. We will show that this is impossible when $\Delta \geq 6$.

First we give an alternate presentation of $G$ which will be useful in the sequel.

\begin{lemma} \label{new presentation}
$G \cong \langle a, d, z: a^2, d^n, (ad)^\Delta, a z^3 d z^{-1} \rangle$.
\end{lemma}

\begin{proof}
Let $d = z^{-1}b^{-1}z$ and eliminate $b = zd^{-1}z^{-1}$. This gives the stated presentation.
\end{proof}

Lemma \ref{new presentation} shows that $G$ is obtained from the triangle group $T = T(2, n, \Delta)$ by adding a new generator $z$ and the relation $az^3dz^{-1} = 1$. Such {\it relative presentations} (\cite{BP}) have been studied extensively. In particular, since $T$ is residually finite, a result of Gerstenhaber and Rothaus \cite{GR} implies

\begin{lemma} \label{T injects}
The natural map $T \to G$ is injective.
\qed
\end{lemma}

The specific relation $az^3dz^{-1}$ is analysed by Edjvet and Howie in \cite{EH}, in the more general setting where $T$ is replaced by an arbitrary group $H$ generated by $a$ and $d$. They show, using the method of Dehn (or Van Kampen) diagrams, that the natural map $H \to G$ is injective \cite[Proposition 1]{EH}. Combining this proof with a result of Bogley and Pride \cite{BP} gives us the following.

\begin{lemma} \label{finite in T}
Any finite subgroup of $G$ is contained in a conjugate of $T$.
\end{lemma}

\begin{proof}
Proposition 1 in \cite{EH} is proved by showing that the relative presentation in question admits no non-empty spherical diagram, except for some special cases where the group $H$ generated by $a$ and $d$ is small. We observe that these do not arise in our situation where $H = T(2, n, \Delta)$. The part of the proof of Proposition 1 that is relevant there is Case 2 (\cite[page 353]{EH}). In the exceptional cases that arise either $H$ is finite, or there is a relation in $H$, other than $a^2$, which contains at most three occurrences each of $a$ and $d$, or a relation which is a product of between one and five words of the form $(ad^{\pm 1})^{\pm 1}$. Since none of these hold in our case ($H = T(2, n, \Delta)$ where $n \geq 3$ and $\Delta \geq 6$), we conclude that our relative presentation of $G$ admits no non-empty spherical diagram. In the dual language of {\it pictures}, this says that it admits no reduced spherical picture \cite{BP}. Since the element $az^3dz^{-1} \in T * \langle z \rangle$ is not a proper power, Lemma \ref{finite in T}  follows from \cite[(0.4)]{BP}.
\end{proof}

\begin{lemma} \label{centre finite}
The centre of $G$ is finite.
\end{lemma}

\begin{proof}
The orbifold Euler characteristic
$$\chi(\Gamma) = \chi^{orb}(P^2(2,n)) = 1 - (\frac12 + \frac{n-1}{n}) = \frac1n - \frac12$$
Hence, unless $n = 3$ and $\Delta = 6$, $\chi(\Gamma) + \frac{1}{\Delta} < 0$, and so by \cite[Theorem 1.2]{BZ2} $G$ has a normal subgroup $G_0$ of finite index with deficiency $\hbox{def}(G_0) \geq 2$ as long as there is a representation $\rho: \Gamma \to PSL_2(\mathbb C)$ which preserves the orders of the torsion elements of $\Gamma$ and which sends $ad$ to an element of order $\Delta$. This is easy to do by hand in our case, but we can also appeal to \cite[Lemma 8.1]{DT} where the result is proven in a broader context. By \cite[Corollaries 2.3.1 and 2.4.1]{Hil2}, the centre $Z(G_0)$, and hence $Z(G)$, is finite.

Suppose then that $n = 3$ and $\Delta = 6$. In this case $\chi(\Gamma) + \frac{1}{\Delta} = 0$ and by  \cite[Theorem 1.2]{BZ2} and \cite[Lemma 8.1]{DT} $G$ has a normal subgroup $G_0$ of finite index with deficiency $\hbox{def}(G_0) \geq 1$. If $\hbox{def}(G_0) > 1$ we argue as above. If $\hbox{def}(G_0) = 1$, \cite[Corollary 1, page 38]{Hil1} implies that if $Z(G_0)$ is infinite then the commutator subgroup $[G_0, G_0]$ is free. But $[G, G]$ contains $[T, T]$ (by Lemma \ref{T injects}), which is isomorphic to $\mathbb Z \oplus \mathbb Z$, and hence $[G_0, G_0]$ contains a copy of $\mathbb Z \oplus \mathbb Z$. It follows that  $Z(G_0)$, and therefore $Z(G)$, is finite in this case also.
\end{proof}

Since the triangle group $T$ has trivial centre, Lemmas \ref{finite in T} and \ref{centre finite} give

\begin{prop} \label{centre trivial}
The centre of $G$ is trivial.
\qed
\end{prop}

\begin{proof}[Proof of Proposition \ref{m=4}]
Suppose that $\Delta(\alpha, \beta) > 5$ and let $\varphi: \pi_1(M(\alpha)) \to G$ be the epimorphism described above. Recall that $M(\alpha)$ is a small Seifert fibred manifold with hyperbolic base orbifold $S^2(a,b,c)$. Let $Z$ be the (infinite cyclic) center of $\pi_1(M(\alpha))$. By Proposition \ref{centre trivial} $\varphi(Z) = \{1\}$, and hence $\varphi$ factors through $\pi_1(M(\alpha))/ Z \cong T(a,b,c) = T'$. Since $T'$ is generated by elements of finite order, its image under the induced homomorphism is contained in $\langle \langle T \rangle \rangle$ by Lemma \ref{finite in T} . Since $G/ \langle \langle T \rangle \rangle \cong \mathbb Z / 2$, this is a contradiction.
\end{proof}

\end{document}